\documentclass[11pt,
]{amsart}
\usepackage{
amssymb, graphicx
}
\usepackage{mathrsfs} 
\usepackage{amsthm}
\usepackage{graphicx,color}
\usepackage{hyperref}

\newtheorem{theorem}{Theorem}
\newtheorem{lemma}{Lemma}
\newtheorem{proposition}{Proposition}
\newtheorem{definition}{Definition}

\theoremstyle{remark}
\usepackage{upgreek}

 \usepackage{amsmath}
\allowdisplaybreaks[4]

\date{\today}

\title[An inverse problem for Euler's equations]{An inverse problem for compressible Euler's equations}

      \author[G. Uhlmann]{Gunther Uhlmann}
\address{G. Uhlmann: Department of Mathematics, University of Washington, Seattle, WA 98195, USA (\tt{gunther@math.washington.edu})
}
\thanks{G. Uhlmann is partly supported by NSF}

  \author[Y. Yi]{Yuchao Yi}
\address{Department of Mathematics, University of California San Diego, La Jolla, CA 92037, USA (\tt{yuyi@ucsd.edu})}
\thanks{}
  \author[J. Zhai]{Jian Zhai}
\address{School of Mathematical Sciences,
  Fudan University, Shanghai 200433, China
  (\tt{jianzhai@fudan.edu.cn}).}
    \thanks{J. Zhai is supported by National Key Research and Development Programs of China (No. 2023YFA1009103), NSFC(No. 12471396), Science and Technology Commission of Shanghai Municipality (23JC1400501)}

\begin{document}
\begin{abstract}
We consider an inverse problem for the compressible Euler's equations in polytropic fluid. We show that by taking active measurements near a particle trajectory one can determine the background flow in a set where pressure waves can propagate from and return to the particle trajectory, under the additional assumption that the flow has nonzero vorticity.
\end{abstract}
\maketitle

\section{Introduction}

Consider the Euler's equations in polytropic compressible fluid
\begin{eqnarray}
(\rho \mathbf{v})_t+\nabla\cdot(\rho\mathbf{v}\otimes\mathbf{v})+\nabla p&= \mathbf{f}&\quad \text{in  }\mathbb{R}^+\times\mathbb{R}^3,\\
\rho_t+\nabla\cdot(\rho \mathbf{v})&=0&\quad \text{in  }\mathbb{R}^+\times\mathbb{R}^3,\\
p(\rho)&=A\rho^\gamma&\quad\text{in  }\mathbb{R}^+\times\mathbb{R}^3.
\end{eqnarray}
Here $\mathbf{v}(t,x)$ is the velocity of a particle at $x\in\mathbb{R}^3$ at time $t$, $p(t,x),\rho(t,x)$ represent the pressure and density of the fluid, and $\mathbf{f}$ represents the body acceleration acting on the fluid. We take $A>0$ and $\gamma>1$ to be constants. This is a standard model for ideal gas and typically $\gamma\in (1,5/3]$ according to basic kinetic theory. In particular, this system of equations with $\gamma=1.4$ is a good model for air.

Throughout the paper we rewrite the Euler's equations as
\begin{eqnarray}
\rho\frac{\partial\mathbf{v}}{\partial t}+\rho\mathbf{v}\cdot\nabla\mathbf{v}+A\gamma\rho^{\gamma-1}\nabla \rho&= \mathbf{f}&\quad \text{in  }\mathbb{R}^+\times\mathbb{R}^3,\label{eulereq1}\\
\rho_t+\nabla\cdot(\rho \mathbf{v})&=0&\quad \text{in  }\mathbb{R}^+\times\mathbb{R}^3,\label{eulereq2}
\end{eqnarray}

We first take $(\rho_0,\mathbf{v}_0)$ to be solution of \eqref{eulereq1} and \eqref{eulereq2} with $f=0$, that is,
\begin{align}
\rho_0\frac{\partial \mathbf{v}_0}{\partial t}+\rho_0\mathbf{v}_0\cdot\nabla \mathbf{v}_0+A\gamma\rho_0^{\gamma-1}\nabla\rho_0&=0,\label{eqv0}\\
\frac{\partial\rho_0}{\partial t}+\nabla\cdot(\rho_0\mathbf{v}_0)&=0.
\end{align}
We assume $\rho_0>0$ and $\mathbf{v}_0$ are both smooth in $(0,T)\times\mathbb{R}^3$ for some $T>0$. For the existence of local smooth solutions of the compressible Euler's equations, we refer to discussions in \cite[Chapter 16.2]{taylor2023partial}. The lifespan $T$ of classical solutions is typically finite and we refer to \cite[Chapter 4.8]{dafermos2005hyberbolic} for more discussions.

\subsection{Lorentzian structure}
Denote $c_0^2=A\rho_0^{\gamma-1}$ and $\mathbf{v}_0=(v_0^1,v_0^2,v_0^3)$. We consider $M=(0,T)\times\mathbb{R}^3$ as a Lorentzian manifold with Lorentzian metric 
\begin{equation}\label{lorentzianmetric}
g=(-1+c_0^{-2}\mathbf{v}_0^T\mathbf{v}_0)(\mathrm{d}t)^2-2c_0^{-2}v_{0}^{j}\mathrm{d}x^j\mathrm{d}t+c_0^{-2}\mathrm{d}x^j\mathrm{d}x^j,
\end{equation}
which can be written in matrix form as
  \[
 \left( \begin{array}{cc}
 -1+c_0^{-2}\mathbf{v}_0^T\mathbf{v}_0&-c_0^{-2}\mathbf{v}_0^T\\
 -c_0^{-2}\mathbf{v}_0 &c_0^{-2}I
  \end{array}\right).
  \]
One will see that pressure/acoustic waves in the fluid will propagate under this Lorentzian metric.
  The inverse metric is then
    \[
 \left( \begin{array}{cc}
 -1&-\mathbf{v}_0^T\\
 -\mathbf{v}_0 &c_0^2I-\mathbf{v}_0 \mathbf{v}_0 ^T
  \end{array}\right).
  \]
Recall that a smooth curve $\mu:(a,b)\rightarrow M$ is called causal if $g(\dot{\mu}(s),\dot{\mu}(s))\leq 0$ and $\dot{\mu}(s)\neq 0$ for all $s\in(a,b)$. Given $p,q\in M$, we denote $p\leq q$ if $p= q$ or $p$ can be joined to $q$ by a future-pointing causal curve. We say $p<q$ if $p\leq q$ and $p\neq q$. We denote the causal future of $p\in M$ by $J_g^+(p)=\{q\in M: p\leq q\}$ and the causal past of $q\in M$ by $J_g^-(q)=\{p\in M: p\leq q\}$. The curve $\mu$ is called time-like if $g(\dot{\mu}(s),\dot{\mu}(s))< 0$ for all $s\in(a,b)$. We denote $p\ll q$ if $p\neq q$ and there is a future-pointing time-like path from $p$ to $q$. Then the chronological future of $p\in M$ is the set $I_g^+(p)=\{q\in M \colon p\ll q\}$ and the chronological past of $q\in M$ is $I_g^-(q)=\{p\in M \colon p\ll q\}$. We also denote $J_g(p,q):=J^+_g(p)\cap J^-_g(q)$ and $I_g(p,q):=I^+_g(p)\cap I^-_g(q)$.\\

We take $\mu$ to be one integral curve of vector field $\partial_t+v_{0}^j\partial_{x^j}$. Then $\mu$ is a future-pointing time-like curve. We take $-1<s_-<s_+<1$, and $p_-=\mu(s_-),p_+=\mu(s_+)$. Take $V$ to be an open connected neighborhood of $\mu([-1,1])$. For any $\mathbf{f}\in C_0^m(V)$ with $\|\mathbf{f}\|_{C^m}<\varepsilon$, we define the source-to-solution map $L_V$ by
\[
L_V\mathbf{f}=(\rho,\mathbf{v})\vert_V\in C(V),
\]
where $(\rho,\mathbf{v})$ is the solution to \eqref{eulereq1}-\eqref{eulereq2} with
\[
(\rho,\mathbf{v})\vert_{t=0}=(\rho_0,\mathbf{v}_0)\vert_{t=0}.
\]
The well-definedness of this operator is a consequence of well-posedness of \eqref{eulereq1}-\eqref{eulereq2} with small sources $\mathbf{f}$ (cf. Proposition \ref{localwellposedness} in the following).
The inverse problem under investigation is to use $L_V$ to recover $(\rho_0,\mathbf{v}_0)$ in the set $I_g(p^-,p^+)$. See Figure \ref{fig:setup}.
\begin{figure}[htbp]
\centering
\includegraphics[width=0.5\textwidth]{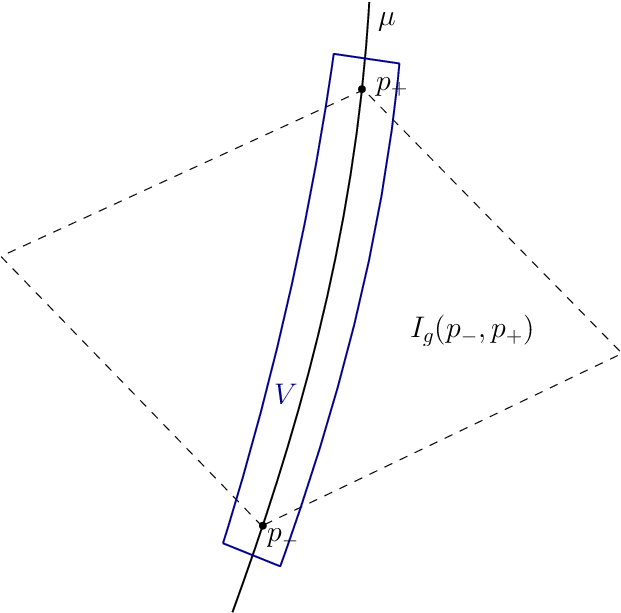}
\caption{$\mu$ is an integral curve of $\partial_t+\mathbf{v}_0\cdot\nabla$ and $p_-=\mu(s_-),p_+\in\mu(s_+)$. $V$ is an open neighborhood of $\mu$. We will recover the background solution in the set $I_g(p_-,p_+)$ where the waves can propagate from  and return back to $\mu([s_-,s_+])$.}
\label{fig:setup}
\end{figure}

\subsection{Well-posedness with small source}
\begin{proposition}\label{localwellposedness}
    Fix $T>0$ and $m\geq 5/2$. If $\|\mathbf{f}\|_{C^m(V)}$ is sufficiently small, then there exists a unique solution 
    \[
    (\rho,\mathbf{v})\in L^\infty([0,T]; H^m(\mathbb{R}^3))\cap W^{1,\infty}([0,T]; H^{m-1}(\mathbb{R}^3))
    \]
    to \eqref{eulereq1}-\eqref{eulereq2} with $(\rho,\mathbf{v})\vert_{t=0}=(\rho_0,\mathbf{v}_0)\vert_{t=0}$.
\end{proposition}
\begin{proof}
    We write
    \[
    \mathbf{v}=\mathbf{v}_0+\delta\mathbf{v},\quad \rho=\rho_0+\delta\rho.
    \]
    Then $(\delta\rho,\delta\mathbf{v})$ satisfies the equation
    \begin{eqnarray}
    \frac{\partial\delta\mathbf{v}}{\partial t}+(\mathbf{v}_0+\delta\mathbf{v})\cdot\nabla\delta\mathbf{v}+\delta\mathbf{v}\cdot\nabla\mathbf{v}_0+A\gamma(\rho_0+\delta\rho)^{\gamma-2}\nabla(\rho_0+\delta\rho)-A\gamma\rho_0^{\gamma-2}\nabla\rho_0&=\frac{f}{\rho_0+\delta\rho},\label{linearized1}\\
    \frac{\partial\delta\rho}{\partial t}+\nabla\cdot(\rho_0\delta\mathbf{v}+\delta\rho\mathbf{v}_0+\delta\rho\delta\mathbf{v})&=0\label{linearized2},
    \end{eqnarray}
    with $(\delta\rho,\delta\mathbf{v})\vert_{t=0}=(0,0)$.

    Consider the nonlinear map $\mathscr{T}(\varrho,\mathbf{w})=(\delta\rho,\delta\mathbf{v})$ where $(\delta\rho,\delta\mathbf{v})$ solves
     \[
     \begin{split}
    &\frac{\partial\delta\mathbf{v}}{\partial t}+(\mathbf{v}_0+\mathbf{w})\cdot\nabla\delta\mathbf{v}+A\gamma(\rho_0+\varrho)^{\gamma-2}\nabla\delta\rho+\delta\mathbf{v}\cdot\nabla\mathbf{v}_0+A\gamma(\gamma-2)\delta\rho\rho_0^{\gamma-3}\nabla\rho_0\\
    =& \frac{f}{\rho_0+\varrho}-A\gamma((\rho_0+\varrho)^{\gamma-2}-\rho_0^{\gamma-2}-(\gamma-2)\rho_0^{\gamma-3}\varrho)\nabla\rho_0,\\
    &\frac{\partial\delta\rho}{\partial t}+(\mathbf{v}_0+\mathbf{w})\cdot\nabla\delta\rho+(\rho_0+\varrho)\nabla\cdot\delta\mathbf{v}+\delta\mathbf{v}\cdot\nabla\rho_0+\delta\rho\nabla\cdot\mathbf{v}_0=0,
    \end{split}
  \]
  with $(\delta\rho,\delta\mathbf{v})\vert_{t=0}=(0,0)$. Notice that $\mathscr{T}(\delta\rho,\delta\mathbf{v})=(\delta\rho,\delta\mathbf{v})$ if and only if $(\delta\rho,\delta\mathbf{v})$ is a solution to \eqref{linearized1} and \eqref{linearized2}.

By \cite[Theorem 1.2.1]{ticequasilinear}, there exists a unique solution $(\delta\rho,\delta\mathbf{v})\in \bigcap_{k=0}^1 W^{k,\infty}([0,T];H^{m-k}(\mathbb{R}^3))$ to the above equations with estimate
\[
\|(\delta\rho,\delta\mathbf{v})\|_Z\leq C(\|f\|_{C^m}+\|\varrho\|_Z^2)e^{KT}
\]
with some constants $C,K$.
So if $\|\mathbf{f}\|_{C^m}$ is sufficiently small, then $\mathscr{T}$ is map from $Z(\varepsilon_0,T)$ to itself where $Z(\varepsilon_0,T)$ is the set of all function $u$ satisfying
  \[
  u\in\bigcap_{k=0}^1 W^{k,\infty}([0,T];H^{m-k}(\mathbb{R}^3)),\quad \|u\|^2_Z:=\sup_{t\in[0,T]}\sum_{k=0}^1\|\partial^k_tu(t)\|^2_{H^{m-k}}\leq \varepsilon_0
  \]
  with $\varepsilon_0>0$ sufficiently small.

  Now let $(\varrho_j,\mathbf{w}_j)\in Z(\epsilon_0,T)$, $j=1,2$, and denote $(\delta\rho_j,\delta\mathbf{v}_j)=\mathscr{T}(\varrho_j,\mathbf{w}_j)$. Then $(\delta\rho_1-\delta\rho_2,\delta\mathbf{v}_1-\delta\mathbf{v}_2)$ solves the equations
  \[
  \begin{split}
  &\frac{\partial(\delta\mathbf{v}_1-\delta\mathbf{v}_2)}{\partial t}+(\mathbf{v}_0+\mathbf{w}_1)\cdot\nabla(\delta\mathbf{v}_1-\delta\mathbf{v}_2)+A\gamma(\rho_0+\varrho_1)^{\gamma-2}\nabla(\delta\rho_1-\delta\rho_2)\\
  &+(\delta\mathbf{v}_1-\delta\mathbf{v}_2)\cdot\nabla\mathbf{v}_0+A\gamma(\gamma-2)(\delta\rho_1-\delta\rho_2)\rho_0^{\gamma-3}\nabla\rho_0=\mathbf{F},\\
  &\frac{\partial(\delta\rho_1-\delta\rho_2)}{\partial t}+(\mathbf{v}_0+\mathbf{w}_1)\cdot\nabla(\delta\rho_1-\delta\rho_2)+(\rho_0+\varrho_1)\nabla\cdot(\delta\mathbf{v}_1-\delta\mathbf{v}_2)\\
  &+(\delta\mathbf{v}_1-\delta\mathbf{v}_2)\cdot\nabla\rho_0+(\delta\rho_1-\delta\rho_2)\nabla\cdot\mathbf{v}_0=H,
  \end{split}
  \]
  where
  \[
  \begin{split}
  \mathbf{F}=& -\frac{f(\varrho_1 - \varrho_2)}{(\rho_0+\varrho_1)(\rho_0+\varrho_2)} -A\gamma\left((\rho_0+\varrho_1)^{\gamma-2}-(\rho_0+\varrho_2)^{\gamma-2}-(\gamma-2)\rho_0^{\gamma-3}(\varrho_1-\varrho_2)\right)\nabla\rho_0\\
  &+(\mathbf{w}_1-\mathbf{w}_2)\cdot\nabla\delta\mathbf{v}_2+A\gamma\left((\rho_0+\varrho_2)^{\gamma-2}-(\rho_0+\varrho_1)^{\gamma-2}\right)\nabla\delta\rho_2,
  \end{split}
  \]
  \[
  H=(\mathbf{w}_2-\mathbf{w}_1)\cdot\nabla\delta\rho_2+(\varrho_2-\varrho_1)\nabla\cdot\delta\mathbf{v}_2.
  \]
  We see that for any $t\in [0,T]$
  \[
  \|(H(t),\mathbf{F}(t))\|_{L^2}\leq C\varepsilon_0(\|\varrho_1(t)-\varrho_2(t)\|_{L^2}+\|\mathbf{w}_1(t)-\mathbf{w}_2(t)\|_{L^2}).
  \]
  If we equip $Z(\varepsilon_0,T)$ with a metric $d$ defined as
  \[
  d(u_1,u_2)=\sup_{t\in[0,T]}\|u_1(t)-u_2(t)\|_{L^2},
  \]
  the space $Z(\varepsilon_0,T)$ is a complete metric space \cite[Theorem 2.2.2]{ticequasilinear}. 
  Then one can show that
  \[
  d(\mathscr{T}(\varrho_1,\mathbf{w}_1),\mathscr{T}(\varrho_2,\mathbf{w}_2))^2=d((\delta\rho_1,\delta\mathbf{v}_1),(\delta\rho_2,\delta\mathbf{v}_2))^2\leq C\varepsilon_0 Te^{KT}d((\varrho_1,\mathbf{w}_1),(\varrho_2,\mathbf{w}_2))^2.
  \]
  Therefore $\mathscr{T}$ is a contraction if $\varepsilon_0$ is sufficiently small. By the contraction mapping theorem we can conclude that the equations \eqref{eulereq1}-\eqref{eulereq2} has a unique solution in $Z(\varepsilon_0,T)$.
\end{proof}
\subsection{Main result and literature review}
The main result of the paper is summarized in the following theorem.
\begin{theorem}\label{thm:mainthm}
Assume $I_g(p^-,p^+)$ contains no cut points and $\nabla\times\mathbf{v}_0\neq 0$ in $I_g(p^-,p^+)$. Assume also that $(\rho_0,\mathbf{v}_0)\vert_V$ is already known. Then $L_V$ uniquely determines the set $I_g(p^-,p^+)$ and the background solution $(\rho_0,\mathbf{v}_0)$ in the set $I_g(p^-,p^+)$.
\end{theorem}

The Euler's equations we study here is a quasilinear hyperbolic system. In particular the background solution $(\rho_0,\mathbf{v}_0)$ induces a Lorentzian metric, which we aim to recover. The background solution $(\rho_0,\mathbf{v}_0)$ actually serves as some coefficients in the equation for the perturbed flow. To recover time-dependent coefficients in linear wave equations, in general the \textit{Boundary Control Method} \cite{belishev1987approach,belishev1992reconstruction} does not work. But certain cases under strong geometrical assumptions can be dealt with \cite{alexakis2022lorentzian,alexakis2024lorentzian}. For the determination of time-independent metrics, we refer to to work \cite{belishev1992reconstruction,kachalov2001inverse,lassas2014inverse,kurylev2018connection,helin2018correlation} using the \textit{Boundary Control Method} or its variant, and \cite{stefanov2005stable} via reducing to geometric inverse problems.

Inverse problems for nonlinear hyperbolic equations have been studied extensively since the work \cite{kurylev2018inverse}. Actually the nonlinear interactions of waves produce rich singularities, which carry extra information for inverse problems. This fact enables many inverse problems for nonlinear equations to be settled when their linear counterparts remain open. For the recovery of nonlinear terms in semilinear wave equations, we refer to \cite{lassas2018inverse, hintz2022inverse, lassas2020uniqueness, lassas2021stability,sa2022recovery,uhlmann2022inverse,sa2024recovery,lassas2025coefficient,lassas2025gaussian}. For the reconstruction of the metric using nonlinearity, \cite{kurylev2018inverse} reduces the source-to-solution map to the earliest light observations sets from which a Lorentzian metric can be determined up to a conformal diffeomorphim. See also \cite{lassas2018inverse,wang2019inverse,hintz2022dirichlet,uhlmann2022inverse} for further results. The work \cite{feizmohammadi2021inverse} introduces the concept of three-to-one scattering relations and use it to recover the metric (via relating to earliest light observation sets). To recover a metric using observation at single point, we refer to \cite{tzou2023determining,nursultanov2025determining}. The works \cite{feizmohammadi2019recovery,chen2021detection,oksanen2025inverse} utilize the nonlinearity to recover the linear coefficients in semilinear waves equations. For results in physically more realistic models, we refer to \cite{kurylev2022inverse, uhlmann2020determination} for results on the Einstein's equations, \cite{chen2025retrieving,chen2021inverse,chen2025inverse} on equations arising from particle physics, \cite{de2018nonlinear,uhlmann2021inverse,uhlmann2024determination} on a nonlinear elastic wave equation, \cite{balehowsky2022inverse} on the relativistic Boltzmann equation, \cite{acosta2022nonlinear,uhlmann2023inverse,kaltenbacher2023simultaneous,eptaminitakis2024weakly,qiu2026inverse} for models arising from nonlinear ultrasound imaging. We also refer to \cite{hintz2024inverse,alexakis2024inverse} for studies on inverse scattering problems for nonlinear wave equations. Stability issues have been addressed in \cite{lassas2021stability,lassas2022inverse, chen2025stable}.

\subsection{Outline of the proof}
In general we will use small sources of the form $\mathbf{f}=\sum_{i=1}^k\epsilon_i\mathbf{f}_i$. Each $\mathbf{f}_i$ gives a distorted plane pressure wave $(\rho_i,\mathbf{v}_i)$, which is the solution to the associated linearized equation. Those distorted plane wave would interact with each other due to the nonlinear nature of the model and generate new waves. We will analyze the (higher-order) asymptotic behaviors of $(\rho,\mathbf{v})$, especially their singularities in $V$, as those $\epsilon_i$'s approach $0$. This is the general strategy adopted in most of the work cited above dealing with inverse problems for nonlinear equations.

We will first use the ``conical" waves (shock waves caused by supersonic moving point sources) generated by the nonlinear interaction of three pressure waves to extract the three-to-one scattering relations associated with the Lorentzian metric \eqref{lorentzianmetric}. Then we use the three-to-one scattering relations to obtain the earliest light observation sets, from which one can determine the metric up to a conformal diffeomorphism. The determination of the conformal class of a Lorentzian manifold from earliest light observation sets is established in \cite{kurylev2018inverse} and used in various subsequent works.

Then we will show that the nonlinear interaction of two pressure waves can generate an advective flow along the integral curve of $\partial_t+\mathbf{v}_0\cdot\nabla$ under the assumption that $\nabla\times\mathbf{v}_0\neq 0$. The advective flow can interact with two other pressure waves and generate again a ``conical" wave, by observing which we can determine the integral curve of $\partial_t+\mathbf{v}_0\cdot\nabla$ up to the same diffeomorphism.\footnote{See the \href{https://github.com/Kerwinyyc/Advective-flow/blob/22cfc90bc212450003efb0a2cdb3a450e9a33341/wave_animation.mp4}{\textit{supplementary video}} illustrating the advective flow interactions.}

Combining the above results, we can determine the background solution $(\rho_0,\mathbf{v}_0)$ completely.

\subsection{Organization of the paper}
In Section \ref{chmultilinearization}, we carry out a third-order linearization of the Euler's equations. In Section \ref{chmicrolocal}, we review some preliminaries from microlocal analysis needed in our analysis and use them to construct distorted plane waves which are solutions of the linearized Euler's equations. Nonlinear interactions of distorted plane waves are analyzed in Section \ref{nonlinearinteraction} and used in Section \ref{chconformal} to reconstruct the conformal structure of the Lorentzian metric. In Section \ref{chadvective}, we analyze the advective flow generated by nonlinear interaction of two distorted plane waves. In Section \ref{cahinterafaw}, we analyze the nonlinear interaction of an advective and two distorted plane waves. Finally, all ingredients are integrated to prove the main result of the paper in Section \ref{chfinal}.

\section{Multilinearization}\label{chmultilinearization}
To carry out asymptotic analysis with multiple small parameters, we use a higher order linearization (multilinearization) programme.  This has been used extensively in the study of inverse problems for nonlinear equations. See, for example, \cite{kurylev2018inverse,lassas2018inverse,hintz2022inverse}.
Let us first do a third-order linearization in detail in this section. Denote $\vec{\epsilon}=(\epsilon_1,\epsilon_2,\epsilon_3)$. Take the source of the form $\mathbf{f}=\epsilon_1\mathbf{f}_1+\epsilon_2\mathbf{f}_2+\epsilon_3\mathbf{f}_3$, and assume the asymptotics
 \begin{equation}\label{asympvrho}
\begin{split}
\mathbf{v}=\mathbf{v}_0+\sum_{i=1}^3\epsilon_i \mathbf{v}_i+\frac{1}{2}\sum_{i, j=1}^3\epsilon_i\epsilon_j \mathbf{v}_{ij}+\sum_{i,j,k=1}^3\frac{1}{6}\epsilon_i\epsilon_j\epsilon_k\mathbf{v}_{ijk}+o(|\vec{\epsilon}|^3),\\
\rho=\rho_0+\sum_{i=1}^3\epsilon_i \rho_i+\frac{1}{2}\sum_{i, j=1}^3\epsilon_i\epsilon_j \rho_{ij}+\sum_{i,j,k=1}^3\frac{1}{6}\epsilon_i\epsilon_j\epsilon_k\rho_{ijk}+o(|\vec{\epsilon}|^3).
\end{split}
\end{equation}
Note that
\begin{align*}
&(\mathbf{v}_0,\rho_0)=(\mathbf{v},\rho)\vert_{\epsilon_1=\epsilon_2=\epsilon_3=0},\\
&(\rho_i,\mathbf{v}_i)=\frac{\partial}{\partial\epsilon_i}(\mathbf{v},\rho)\vert_{\epsilon_1=\epsilon_2=\epsilon_3=0},\\
&(\mathbf{v}_{ij},\rho_{ij})=\frac{\partial^2}{\partial\epsilon_i\partial\epsilon_j}(\mathbf{v},\rho)\vert_{\epsilon_1=\epsilon_2=\epsilon_3=0},\\
&(\mathbf{v}_{ijk},\rho_{ijk})=\frac{\partial^3}{\partial\epsilon_i\partial\epsilon_j\partial\epsilon_k}(\mathbf{v},\rho)\vert_{\epsilon_1=\epsilon_2=\epsilon_3=0}.
\end{align*}
We will use the asymptotics
\[
\begin{split}
p=&A\rho^\gamma=A\rho_0^\gamma+A\gamma\rho_0^{\gamma-1}\sum_{i=1}^3\epsilon_i\rho_i+\frac{A\gamma(\gamma-1)}{2}\rho_0^{\gamma-2}\sum_{i,j=1}^3\epsilon_i\epsilon_j\rho_i\rho_j+\frac{A\gamma}{2}\rho_0^{\gamma-1}\sum_{i,j=1}^3\epsilon_i\epsilon_j\rho_{ij}\\
&+\frac{A\gamma}{6}\rho_0^{\gamma-1}\sum_{i,j,k=1}^3\epsilon_i\epsilon_j\epsilon_k\rho_{ijk}+\frac{A\gamma(\gamma-1)(\gamma-2)}{6}\rho_0^{\gamma-3}\sum_{i,j,k=1}^3\epsilon_i\epsilon_j\epsilon_k\rho_i\rho_j\rho_k\\
&+\frac{A\gamma(\gamma-1)}{6}\sum_{i,j,k=1}^3\rho_0^{\gamma-2}\epsilon_i\epsilon_j\epsilon_k(\rho_{ij}\rho_k+\rho_{ik}\rho_j+\rho_{jk}\rho_i)+o(|\vec{\epsilon}|^3).
\end{split}
\]

Substituting the asymptotics into the Euler's equations, and collecting terms of the same order in $\vec{\epsilon}$, we will get a sequence of equations for $(\rho_i,\mathbf{v}_i)$, $(\rho_{ij},\mathbf{v}_{ij})$ and $(\rho_{ijk},\mathbf{v}_{ijk})$.
The first order derivative $(\rho_i,\mathbf{v}_i)$ satisfy the equations
\[
\begin{split}
\rho_0\frac{\partial\mathbf{v}_i}{\partial t}+\rho_i\frac{\partial\mathbf{v}_0}{\partial t}
+\rho_0\mathbf{v}_0\cdot\nabla\mathbf{v}_i+\rho_i\mathbf{v}_0\cdot\nabla\mathbf{v}_0+\rho_0\mathbf{v}_i\cdot\nabla \mathbf{v}_0\\
+A\gamma\rho_0^{\gamma-1}\nabla\rho_i+A\gamma(\gamma-1)\rho_0^{\gamma-2}\nabla\rho_0\rho_i=\mathbf{f}_i,
\end{split}
\]
which, using \eqref{eqv0}, can be written as
\[
\begin{split}
\rho_0\frac{\partial\mathbf{v}_i}{\partial t}+\rho_0\mathbf{v}_0\cdot\nabla\mathbf{v}_i+\rho_0\mathbf{v}_i\cdot\nabla \mathbf{v}_0+A\gamma\rho_0^{\gamma-1}\nabla\rho_i+A\gamma(\gamma-2)\rho_0^{\gamma-2}\nabla\rho_0\rho_i=\mathbf{f}_i,
\end{split}
\]
and
\[
\frac{\partial\rho_i}{\partial t}+\rho_0\nabla\cdot\mathbf{v}_i+\mathbf{v}_i\cdot\nabla\rho_0+\rho_i\nabla\cdot\mathbf{v}_0+\mathbf{v}_0\cdot\nabla\rho_i=0.
\]

Collecting $\mathcal{O}(\epsilon^2)$ terms, we have, for $i\neq j$,
\[
\begin{split}
&\rho_i\frac{\partial\mathbf{v}_j}{\partial t}+\rho_j\frac{\partial \mathbf{v}_i}{\partial t}+\rho_{ij}\frac{\partial \mathbf{v}_0}{\partial t}+\rho_0\frac{\partial \mathbf{v}_{ij}}{\partial t}\\
+&\rho_0\mathbf{v}_i\cdot\nabla\mathbf{v}_j+\rho_0\mathbf{v}_j\cdot\nabla\mathbf{v}_i+\rho_i\mathbf{v}_j\cdot\nabla\mathbf{v}_0+\rho_i\mathbf{v}_0\cdot\nabla\mathbf{v}_j+\rho_j\mathbf{v}_i\cdot\nabla\mathbf{v}_0+\rho_j\mathbf{v}_0\cdot\nabla\mathbf{v}_i\\
+&\rho_{ij}\mathbf{v}_0\cdot\nabla\mathbf{v}_0+\rho_0\mathbf{v}_{ij}\cdot\nabla\mathbf{v}_0+\rho_0\mathbf{v}_0\cdot\nabla\mathbf{v}_{ij}\\
+&A\gamma(\gamma-1)\rho_0^{\gamma-2}\nabla\rho_0\rho_{ij}+A\gamma\rho_0^{\gamma-1}\nabla\rho_{ij}\\
+&A\gamma(\gamma-1)(\gamma-2)\rho_0^{\gamma-3}\nabla\rho_0\rho_i\rho_j+A\gamma(\gamma-1)\rho_0^{\gamma-2}(\nabla\rho_i\rho_j+\rho_i\nabla\rho_j)=0,
\end{split}
\]
which can be simplified to
\begin{equation}\label{momentumij}
\begin{split}
&\rho_0\frac{\partial\mathbf{v}_{ij}}{\partial t}+\rho_0\mathbf{v}_0\cdot\nabla\mathbf{v}_{ij}+\rho_0\mathbf{v}_{ij}\cdot\nabla \mathbf{v}_0+A\gamma\rho_0^{\gamma-1}\nabla\rho_{ij}+A\gamma(\gamma-2)\rho_0^{\gamma-2}\nabla\rho_0\rho_{ij}\\
=&-\rho_0\mathbf{v}_i\cdot\nabla\mathbf{v}_j-\rho_0\mathbf{v}_j\cdot\nabla\mathbf{v}_i-A\gamma(\gamma-2)\rho_0^{\gamma-2}(\nabla\rho_i\rho_j+\rho_i\nabla\rho_j)\\
&-A\gamma(\gamma-2)(\gamma-3)\rho_0^{\gamma-3}\nabla\rho_0\rho_i\rho_j,
\end{split}
\end{equation}
and
\begin{equation}\label{continuityij}
\begin{split}
&\frac{\partial\rho_{ij}}{\partial t}+\rho_0\nabla\cdot\mathbf{v}_{ij}+\mathbf{v}_{ij}\cdot\nabla\rho_0+\rho_{ij}\nabla\cdot\mathbf{v}_0+\mathbf{v}_0\cdot\nabla\rho_{ij}\\
=&-\rho_i\nabla\cdot\mathbf{v}_j-\mathbf{v}_j\cdot\nabla\rho_i-\rho_j\nabla\cdot\mathbf{v}_i-\mathbf{v}_i\cdot\nabla\rho_j.
\end{split}
\end{equation}

Collecting $\mathcal{O}(\epsilon^3)$ terms, we obtain
\begin{align*}
&\rho_0\frac{\partial\mathbf{v}_{123}}{\partial t}+\rho_0\mathbf{v}_0\cdot\nabla\mathbf{v}_{123}+\rho_0\mathbf{v}_{123}\cdot\nabla \mathbf{v}_0+A\gamma\rho_0^{\gamma-1}\nabla\rho_{123}+A\gamma(\gamma-2)\rho_0^{\gamma-2}\nabla\rho_0\rho_{123}\\
=&-\frac{1}{2}\rho_{\sigma(1)}\frac{\partial\mathbf{v}_{\sigma(2)\sigma(3)}}{\partial t}-\frac{1}{2}\rho_{\sigma(1)\sigma(2)}\frac{\partial\mathbf{v}_{\sigma(3)}}{\partial t}-\frac{1}{2}\rho_0\mathbf{v}_{\sigma(1)\sigma(2)}\cdot\nabla\mathbf{v}_{\sigma(3)}-\frac{1}{2}\rho_0\mathbf{v}_{\sigma(1)}\cdot\nabla\mathbf{v}_{\sigma(2)\sigma(3)}\\
&-\rho_{\sigma(1)}\mathbf{v}_{\sigma(2)}\cdot\nabla\mathbf{v}_{\sigma(3)}-\frac{1}{2}\rho_{\sigma(1)}\mathbf{v}_0\cdot\nabla\mathbf{v}_{\sigma(2)\sigma(3)}-\frac{1}{2}\rho_{\sigma(1)}\mathbf{v}_{\sigma(2)\sigma(3)}\cdot\nabla\mathbf{v}_0\\
&-\frac{1}{2}\rho_{\sigma(1)\sigma(2)}\mathbf{v}_{\sigma(3)}\cdot\nabla\mathbf{v}_0-\frac{1}{2}\rho_{\sigma(1)\sigma(2)}\mathbf{v}_0\cdot\nabla\mathbf{v}_{\sigma(3)}\\
&-\nabla\left(\frac{A}{2}\gamma(\gamma-1)\rho_0^{\gamma-2}\rho_{\sigma(1)}\rho_{\sigma(2)\sigma(3)}+A\gamma(\gamma-1)(\gamma-2)\rho_0^{\gamma-3}\rho_1\rho_2\rho_3\right)\\
=&-\frac{1}{2}\rho_0\mathbf{v}_{\sigma(1)\sigma(2)}\cdot\nabla\mathbf{v}_{\sigma(3)}-\frac{1}{2}\rho_0\mathbf{v}_{\sigma(1)}\cdot\nabla\mathbf{v}_{\sigma(2)\sigma(3)}-\rho_{\sigma(1)}\mathbf{v}_{\sigma(2)}\cdot\nabla\mathbf{v}_{\sigma(3)}\\
&+\frac{1}{2}\rho_{\sigma(1)}\left(A\gamma\rho_0^{\gamma-2}\nabla\rho_{\sigma(2)\sigma(3)}+\mathbf{v}_{\sigma(2)}\cdot\nabla\mathbf{v}_{\sigma(3)}+\mathbf{v}_{\sigma(3)}\cdot\nabla\mathbf{v}_{\sigma(2)}+A\gamma(\gamma-2)\rho_0^{\gamma-3}(\nabla\rho_{\sigma(2)}\rho_{\sigma(3)}+\rho_{\sigma(2)}\nabla\rho_{\sigma(3)})\right)\\
&+\frac{1}{2}\rho_{\sigma(1)\sigma(2)}A\gamma\rho_0^{\gamma-2}\nabla\rho_{\sigma(3)}\\
&-\frac{A}{2}\gamma(\gamma-1)\rho_0^{\gamma-2}(\nabla\rho_{\sigma(1)}\rho_{\sigma(2)\sigma(3)}+\rho_{\sigma(1)}\nabla\rho_{\sigma(2)\sigma(3)})\\
&-A\gamma(\gamma-1)(\gamma-2)\rho_0^{\gamma-3}(\nabla\rho_1\rho_2\rho_3+\rho_1\nabla\rho_2\rho_3+\rho_1\rho_2\nabla\rho_3)\\
&+\mathrm{l.o.t.}\\
=&-\frac{1}{2}\rho_0\mathbf{v}_{\sigma(1)\sigma(2)}\cdot\nabla\mathbf{v}_{\sigma(3)}-\frac{1}{2}\rho_0\mathbf{v}_{\sigma(1)}\cdot\nabla\mathbf{v}_{\sigma(2)\sigma(3)}\\
&-\frac{A}{2}\gamma(\gamma-2)\rho_0^{\gamma-2}(\nabla\rho_{\sigma(1)}\rho_{\sigma(2)\sigma(3)}+\rho_{\sigma(1)}\nabla \rho_{\sigma(2)\sigma(3)})\\
&-A\gamma(\gamma-2)(\gamma-3)\rho_0^{\gamma-3}(\nabla\rho_1\rho_2\rho_3+\rho_1\nabla\rho_2\rho_3+\rho_1\rho_2\nabla\rho_3)+\mathrm{l.o.t.}\\
=&:\mathbf{f}_{123}
\end{align*}
Here $\sigma$ runs over all permutations of $(1,2,3)$.
Also we have
\[
\partial_t\rho_{123}+\nabla\cdot\left(\rho_0\mathbf{v}_{123}+\rho_{123}\mathbf{v}_0+\frac{1}{2}\rho_{\sigma(1)}\mathbf{v}_{\sigma(2)\sigma(3)}+\frac{1}{2}\rho_{\sigma(1)\sigma(2)}\mathbf{v}_{\sigma(3)}\right)=0.
\]
which can be written as
\[
\begin{split}
&\partial_t\rho_{123}+\rho_0\nabla\cdot\mathbf{v}_{123}+\mathbf{v}_{123}\cdot\nabla\rho_0+\rho_{123}\nabla\cdot\mathbf{v}_0+\mathbf{v}_0\cdot\nabla\rho_{123}\\
=&-\frac{1}{2}\nabla\cdot\left(\rho_{\sigma(1)}\mathbf{v}_{\sigma(2)\sigma(3)}+\rho_{\sigma(1)\sigma(2)}\mathbf{v}_{\sigma(3)}\right)\\
=&:h_{123}.
\end{split}
\]
We will carefully analyze this third-order linearized solution $(\rho_{123},\mathbf{v}_{123})$.



\section{Microlocal analysis of linear waves}\label{chmicrolocal}
We will use solutions $(\rho_i,\mathbf{v}_i)$ with conormal singularities. The nonlinear interaction of such waves will generate new singularties. Let us first give a brief review of the techniques from microlocal analysis that are needed.
\subsection{Preliminaries}
Let $X$ be an $n$-dimensional manifold. A submanifold $\Lambda\subset T^*X\setminus 0$ is called a Lagrangian submanifold if $\dim \Lambda=n$ and the canonical $2$-form $\sum_{j=1}^n\mathrm{d}\xi_j\wedge\mathrm{d}x^j$ vanishes on $\Lambda$.
 Assume $\Lambda$ is a smooth conic Lagragian submanifold of $T^*X\setminus 0$.  If a non-degenerate phase function $\phi(x,\theta):U\times\mathbb{R}^N\rightarrow\mathbb{R}$ (homogeneous of degree $1$ in $\theta$) locally parametrizes $\Lambda$, i.e.,
\[
\{(x,\mathrm{d}_x\phi)\in T^*X\setminus 0:x\in U,\mathrm{d}_\theta\phi=0\}\subset\Lambda,
\]
We say that a distribution $u\in\mathcal{D}'(X)$ is a classical Lagrangian distribution of order $\mu$ and denote $u\in I^\mu(\Lambda)$ if $u$ can be represented by an oscillatory integral,
\[
u(x)=\int_{\mathbb{R}^N}e^{\mathrm{i}\phi(x,\theta)}a(x,\theta)\mathrm{d}\theta
\]
with $a\in S^{\mu+\frac{n}{4}-\frac{N}{2}}(U\times\mathbb{R}^N)$ is a classical symbol of order $\mu+\frac{n}{4}-\frac{N}{2}$. We refer the readers to \cite{grigis1994microlocal,hormander2009analysis} for more details. For a classical Lagrangian distribution $u\in \mathcal{I}^\mu(\Lambda)$ we can define its principal symbol $\sigma(u)(\zeta)$ at $\zeta\in\Lambda$
\[
\sigma(u)(\zeta)\in S^{\mu+\frac{n}{4}}(\Lambda,\Omega^{1/2}\otimes L)/S^{\mu+\frac{n}{4}-1}(\Lambda,\Omega^{1/2}\otimes L),
\]
where $L$ is the Maslov-Keller line bundle and $\Omega^{1/2}$ is the half-density on $\Lambda$ (cf. \cite[Theorem 11.10]{grigis1994microlocal}). 

If $Y\subset X$ is a submanifold, then the conormal bundle $N^*Y$ is a conic Lagrangian submanifold. We still denote $\mathcal{I}^\mu(Y)=\mathcal{I}^\mu(N^*Y)$. Assume that we can find local coordinates $x=(x',x'')$, where $x'=(x^1,\cdots, x^k)$, $x''=(x^{k+1},\cdots, x^n)$, such that $Y=\{x'=0\}$. Also write $\xi=(\xi',\xi'')$, then $N^*Y=\{x'=0,\xi''=0,\xi'\neq 0\}$. Then one can write
\[
u=\int_{\mathbb{R}^k}e^{\mathrm{i}x'\cdot\xi'}a(x,\xi')\mathrm{d}\xi'
\]
with $a\in S^{\mu+\frac{n}{4}-\frac{k}{2}}(\mathbb{R}^n\times\mathbb{R}^k)$. Ignoring the half-density $|\mathrm{d}x''|^{1/2}|\mathrm{d}\xi'|^{1/2}$ and trivializing the Maslov-Keller line bundle, the principal symbol of $u$ can be identified with a function $a_0(x,\xi')\in S^{\mu+\frac{n}{4}-\frac{k}{2}}(\mathbb{R}^n\times\mathbb{R}^k)$ such that $a_0$ is homogeneous of degree $\mu+\frac{n}{4}-\frac{k}{2}$ in $\xi'$ and $a-a_0\in S^{\mu+\frac{n}{4}-\frac{k}{2}-1}(\mathbb{R}^n\times\mathbb{R}^k)$.\\

Next we recall the concept of paired Lagrangian distributions \cite{melrose1979lagrangian,guillemin1981oscillatory,greenleaf1993recovering}.
First consider $X=\mathbb{R}^n$ with Euclidean coordinates $(x^1,x^2,\cdots,x^n)=(x',x'',x''')$ where $x'=(x^1,\cdots,x^k)$, $x''=(x^{k+1},\cdots,x^{n-d})$ and $x'''=(x^{n-d+1},\cdots,x^n)$. Let $\Lambda_1=N^*S_1, \Lambda_2=N^*S_2$, where
\[
S_1=\{x'=x''=0\},\quad S_2=\{x''=0\}.
\]
 We can then define $u\in \mathcal{I}^{p,\ell}(\Lambda_1,\Lambda_2)$ if
\[
u=\int e^{\mathrm{i}(x'\cdot\xi'+x''\cdot\xi'')}a(x,\xi',\xi'')\mathrm{d}\xi'\mathrm{d}\xi''+u_0,
\]
where $u_0\in\mathcal{I}^p(\Lambda_2)$ and $a\in S^{p-\frac{n}{4}+\frac{k}{2}+\frac{d}{2},\ell-\frac{k}{2}}(\mathbb{R}^n;\mathbb{R}^{n-k-d}_{\xi''};\mathbb{R}^k_{\xi'})$ (cf. \cite{de2015diffraction}). Here $a\in S^{M_1,M_2}(\mathbb{R}^n;\mathbb{R}^{n-k-d}_{\xi''};\mathbb{R}^k_{\xi'})$ means
\[
|\partial^\gamma_{x'''}\partial^\beta_{\xi'}\partial^\alpha_{\xi''}a(x,\xi',\xi'')|\leq C_{K,\alpha,\beta,\gamma}(1+|\xi'|+|\xi''|)^{M_1-|\alpha|}(1+|\xi'|)^{M_2-|\beta|}
\]
with some constant $C_{K,\alpha,\beta,\gamma}$ for all $x\in K$.

When $X$ is a manifold, assume $\Lambda_1$ and $\Lambda_2\subset T^*X\setminus 0$ are two cleanly intersecting Lagrangian submanifolds. Recall that $\Lambda_1$ and $\Lambda_2$ intersect cleanly if $\Sigma=\Lambda_1\cap \Lambda_2$ is a smooth manifold and its tangent space satisfies $T_\lambda\Sigma=T_\lambda\Lambda_1\cap T_\lambda\Lambda_2$ for all $\lambda \in\Sigma$. We define the class $\mathcal{I}^{p,\ell}(\Lambda_1,\Lambda_2)\subset\mathcal{D}'(X)$ to consist of distributions of the form
\[
u=u_1+u_2+\sum F_jv_j,
\]
where $u_1\in \mathcal{I}^{p+\ell}(\Lambda_1\setminus\Lambda_2)$, $u_2\in \mathcal{I}^{p}(\Lambda_2\setminus\Lambda_1)$, $v_i\in\mathcal{I}^{p,\ell}(\mathbb{R}^n;N^*S_1,N^*S_2)$ for some linear subspaces $S_1,S_2\subset\mathbb{R}^n$ with $S_1\subset S_2$ and $F_j$ is a zeroth order FIO associated with a canonical transformation $\chi_j:X\rightarrow T^*\mathbb{R}^n$ mapping (locally) $\Lambda_1$ into $N^*S_1$ and  $\Lambda_2$ into $N^*S_2$. We have that if $u\in \mathcal{I}^{p,\ell}(X,\Lambda_1,\Lambda_2)$ then $u\in\mathcal{I}^{p+\ell}(\Lambda_1\setminus\Lambda_2)$ and $u\in\mathcal{I}^{p}(\Lambda_2\setminus\Lambda_1)$ (see \cite{melrose1979lagrangian,greenleaf1993recovering} for more details).

\subsection{Analysis of the linearized equation}
We consider the linear equation\label{sec:linearizedeq}
\begin{equation}\label{lineareq1}
\rho_0\frac{\partial\mathbf{v}}{\partial t}+\rho_0\mathbf{v}_0\cdot\nabla\mathbf{v}+\rho_0\mathbf{v}\cdot\nabla \mathbf{v}_0+A\gamma\rho_0^{\gamma-1}\nabla\rho+A\gamma(\gamma-2)\rho_0^{\gamma-2}\nabla\rho_0\rho=\mathbf{f},
\end{equation}
and
\begin{equation}\label{lineareq2}
\frac{\partial\rho}{\partial t}+\rho_0\nabla\cdot\mathbf{v}+\mathbf{v}\cdot\nabla\rho_0+\rho\nabla\cdot\mathbf{v}_0+\mathbf{v}_0\cdot\nabla\rho=h.
\end{equation}
Here $(\rho_0,\mathbf{v}_0)$ are known. The equation \eqref{lineareq1} can be rewritten as
\[
\frac{\partial\mathbf{v}}{\partial t}+\mathbf{v}_0\cdot\nabla\mathbf{v}+\mathbf{v}\cdot\nabla\mathbf{v}_0+\frac{c_0^2}{\rho_0}\nabla\rho+\nabla\left(\frac{c_0^2}{\rho_0}\right)\rho=\frac{\mathbf{f}}{\rho_0}.
\]
We denote the Fourier dual variable of $z=(t,x)$ as $\zeta=(\tau',\xi)$ and $\tau=\tau'+\sum_{j=1}^3\mathbf{v}_0^j\xi_j$.
We will use the notation $D_t=\partial_t+\mathbf{v}_0\cdot\nabla$ throughout the paper and denote $\zeta=(\tau,\xi)$. We can write the above system of equations as
\begin{equation}\label{lineareq_nonh}
P\left(\begin{array}{c}
\rho\\
\mathbf{v}
\end{array}\right)=\left(\begin{array}{cc}
D_t&\rho_0\nabla\cdot\\
\frac{c_0^2}{\rho_0}\nabla &D_t
\end{array}\right)\left(\begin{array}{c}
\rho\\
\mathbf{v}
\end{array}\right)+\mathrm{l.o.t.}=\left(\begin{array}{c}
h\\
\frac{\mathbf{f}}{\rho_0}
\end{array}\right).
\end{equation}
The principal symbol of $\frac{P}{\mathrm{i}}$ as a pseudodifferential operator is
\[
\left(\begin{array}{cc}
\tau&\rho_0\xi^T\\
\frac{c_0^2}{\rho_0}\xi &\tau
\end{array}\right).
\]
 The above matrix has four eigenvalues (counted with multiplicity) $\tau-c_0|\xi|$, $\tau+c_0|\xi|$, $\tau$, $\tau$. For the associated eigenvectors, we notice that
\[
\left(\begin{array}{cc}
\tau&\rho_0\xi^T\\
\frac{c^2_0}{\rho_0}\xi &\tau
\end{array}\right)\left(\begin{array}{c}
\rho_0\\
c_0\hat{\xi}
\end{array}\right)=(\tau+c_0|\xi|)\left(\begin{array}{c}
\rho_0\\
c_0\hat{\xi}
\end{array}\right),
\]
\[
\left(\begin{array}{cc}
\tau&\rho_0\xi^T\\
\frac{c_0^2}{\rho_0}\xi &\tau
\end{array}\right)\left(\begin{array}{c}
\rho_0\\
-c_0\hat{\xi}
\end{array}\right)=(\tau-c_0|\xi|)\left(\begin{array}{c}
\rho_0\\
-c_0\hat{\xi}
\end{array}\right),
\]
and for any $\xi^\perp\in\mathbb{R}^3\setminus 0$ perpendicular to $\xi$
\[
\left(\begin{array}{cc}
\tau&\rho_0\xi^T\\
\frac{c_0^2}{\rho_0}\xi &\tau
\end{array}\right)\left(\begin{array}{c}
0\\
\hat{\xi^\perp}
\end{array}\right)=\tau\left(\begin{array}{c}
0\\
\hat{\xi^\perp}
\end{array}\right),
\]
Since $\tau\pm c_0|\xi|$ and $\tau$ are symbols homogeneous of degree $1$ and of real principal type, the wavefront set of $(\rho,\mathbf{v})$, excluding the wavefront set of $(h,\mathbf{f})$, is contained in the bicharacteristic sets of $H_{\tau\pm c_0|\xi|}$ and $H_\tau$ (cf. \cite{grigis1994microlocal}). Here $H_p$ is the Hamiltonian vector field of $p(z,\zeta)\in T^*M$, that is,
\[
H_p=\frac{\partial p}{\partial \xi_i}\partial_{x^i}-\frac{\partial p}{\partial x^i}\partial_{\xi_i}.
\]

It was shown in \cite{melrose1979lagrangian} that $P$ has a parametrix
\[
Q\in \mathcal{I}^{-\frac{1}{2},-\frac{1}{2}}(\Delta'_{T^*M},\Lambda_{\tau-c_0|\xi|})+\mathcal{I}^{-\frac{1}{2},-\frac{1}{2}}(\Delta'_{T^*M},\Lambda_{\tau+c_0|\xi|})+\mathcal{I}^{-\frac{1}{2},-\frac{1}{2}}(\Delta'_{T^*M},\Lambda_{\tau}).
\]
In above we considered the Schwartz kernel of $Q$ as a paired Lagrangian distribution.
Denote
\[
\Sigma_{p}=\{(t,x,\tau,\xi)\in T^*M:p(t,x,\tau,\xi)=0\}
\]
Here $\Delta'_{T^*M}=N^*(\{z,z\};z\in M)$ and $\Lambda_p\subset T^*M\times T^*M$ is the Lagrangian submanifold obtained by flowing out $\Delta'_{T^*M}\cap\Sigma_p$ under $H_p$, that is,
\[
\Lambda_p=\{(z,\zeta,z',-\zeta');(z,\zeta)\in T^*M\setminus0,p(z,\zeta)=0,(z',\zeta')\in\Theta_{z,\zeta}\},
\]
where $\Theta_{z,\zeta}$ is the bicharacteristic of $p$ containing $(z,\zeta)$.

We denote $L_p^*M=\{\zeta\in T^*_pM\setminus\{0\},g(z,\zeta)=-\tau^2+c_0^2|\xi|^2=0\}$ be the set of light-like vectors in the cotangent space $T^*_pM$. Also $L^{*,+}_pM\subset L^*_pM$ and $L_p^{*,-}M\subset L^*_pM$ denote the future and the past light-like covectors.\\

To construct the parametrix more explicitly, we first take $h,\mathbf{f}=0$ and construct FIO solutions of the form
\begin{equation}\label{FIOsolution}
\begin{split}
&\left(\begin{array}{c}
\rho\\
\mathbf{v}
\end{array}\right)(t,x)\\
=&\sum_{\pm}\int e^{\mathrm{i}\omega\phi_\pm(t,x,\tau,\xi)}\left(\begin{array}{c}
a_\pm\\
\mathbf{b}_\pm
\end{array}\right)(t,x,\tau,\xi)\mathrm{d}\tau\mathrm{d}\xi+\int e^{\mathrm{i}\omega\psi(t,x,\tau,\xi)}\left(\begin{array}{c}
c\\
\mathbf{d}
\end{array}\right)(t,x,\tau,\xi)\mathrm{d}\tau\mathrm{d}\xi,
\end{split}
\end{equation}
where the phase functions $\phi_\pm$ and $\psi$ satisfies
\[
D_t\phi_\pm \pm c_0|\nabla\phi_\pm|=0
\]
and
\[
D_t\psi=0.
\]
Also the amplitudes are of the asymptotic form
\[
(a_\pm,\mathbf{b}_\pm,c,\mathrm{d})\sim \sum_{k=0}^\infty\omega^{-k}(a^{(k)}_\pm,\mathbf{b}^{(k)}_\pm,c^{(k)},\mathbf{d}^{(k)}).
\]

For the ease of notations, we denote $\phi=\phi_\pm$, $a^{(k)}=a^{(k)}_\pm$, $\mathbf{b}^{(k)}=\mathbf{b}^{(k)}_\pm$. Then $\phi$ satisfies the Eikonal equation
\[
(D_t\phi)^2-c_0^2|\nabla\phi|^2=0.
\]
Substituting the ansatz \eqref{FIOsolution} into the equations \eqref{lineareq1} and \eqref{lineareq2} and collecting terms of different orders in $\omega$, we get the following equations
\[
a^{(0)}D_t\phi+\rho_0\mathbf{b}^{(0)}\cdot\nabla\phi=0,
\]
\[
\mathbf{b}^{(0)}D_t\phi+a^{(0)}\frac{c^2_0}{\rho_0}\nabla\phi=0,
\]
\begin{equation}\label{transporteqa0}
D_ta^{(0)}+\mathrm{i}a^{(1)}D_t\phi+\rho_0\nabla\cdot\mathbf{b}^{(0)}+\mathrm{i}\rho_0\mathbf{b}^{(1)}\cdot\nabla\phi+\mathbf{b}^{(0)}\cdot\nabla\rho_0+a^{(0)}\nabla\cdot\mathbf{v}_0=0,
\end{equation}
\begin{equation}\label{transporteqb0}
D_t\mathbf{b}^{(0)}+\mathrm{i}\mathbf{b}^{(1)}D_t\phi+\mathbf{b}^{(0)}\cdot\nabla\mathbf{v}_0+\frac{c^2_0}{\rho_0}\nabla a^{(0)}+\mathrm{i}a^{(1)}\frac{c^2_0}{\rho_0}\nabla\phi+\nabla\frac{c^2_0}{\rho_0}a^{(0)}=0.
\end{equation}
Notice that $a^{(0)}$ and $\mathbf{b}^{(0)}$ satisfy
\[
\left(\begin{array}{cc}
D_t\phi&\rho_0\nabla^T\phi\\
\frac{c^2_0}{\rho_0}\nabla\phi&D_t\phi
\end{array}\right)\left(\begin{array}{c}
a^{(0)}\\
\mathbf{b}^{(0)}
\end{array}\right)=0.
\]
The matrix in the above equation is singular since $\phi$ satisfies the Eikonal equation, and then $\left(\begin{array}{c}
a^{(0)}\\
\mathbf{b}^{(0)}
\end{array}\right)$ needs to be in its kernel. Therefore 
\[
\left(\begin{array}{c}
a^{(0)}\\
\mathbf{b}^{(0)}
\end{array}\right)=\left(\begin{array}{c}
-\alpha D_t\phi\\
\frac{ c_0^2}{\rho_0}\alpha\nabla\phi
\end{array}\right)
\]
with some scalar function $\alpha$.


Next we use the equations \eqref{transporteqa0} and \eqref{transporteqb0} to derive a transport equation for $\alpha$. Notice that 
\[
(\mathrm{i}a^{(1)}D_t\phi+\mathrm{i}\rho_0\mathbf{b}^{(1)}\cdot\nabla\phi)D_t\phi-\rho_0(\mathrm{i}\mathbf{b}^{(1)}D_t\phi+\mathrm{i}a^{(1)}\frac{c_0^2}{\rho_0}\nabla\phi)\cdot\nabla\phi=0.
\]
Then we obtain from \eqref{transporteqa0} and \eqref{transporteqb0} the following equation
\[
\begin{split}
&(D_ta^{(0)}+\rho_0\nabla\cdot\mathbf{b}^{(0)}+\mathbf{b}^{(0)}\cdot\nabla\rho_0+a^{(0)}\nabla\cdot\mathbf{v}_0)D_t\phi\\
&-\rho_0\left(D_t\mathbf{b}^{(0)}+\mathbf{b}^{(0)}\cdot\nabla\mathbf{v}_0+\frac{c^2_0}{\rho_0}\nabla a^{(0)}+\nabla\frac{c^2_0}{\rho_0}a^{(0)}\right)\cdot\nabla\phi=0
\end{split}
\]
Using the facts that $a^{(0)}=-\alpha D_t\phi$ and $\mathbf{b}^{(0)}=\frac{c_0^2}{\rho_0}\alpha\nabla\phi$, we can reduce the above equation to
\[
\begin{split}
D_t(\alpha D_t\phi)D_t\phi-\rho_0\nabla\cdot(\frac{c_0^2}{\rho_0}\alpha\nabla\phi)D_t\phi-\frac{c_0^2}{\rho_0}\alpha\nabla\phi\cdot\nabla\rho_0D_t\phi+\alpha(D_t\phi)^2\nabla\cdot\mathbf{v}_0\\
+\rho_0D_t(\frac{c_0^2}{\rho_0}\alpha\nabla\phi)\cdot\nabla\phi+c_0^2\alpha(\nabla\phi\cdot\nabla\mathbf{v}_0)\cdot\nabla\phi-\rho_0\nabla(\frac{c_0^2}{\rho_0}\alpha D_t\phi)\cdot\nabla\phi=0.
\end{split}
\]

The tangent vector field along the null geodesic is
\[
-\partial_t\phi\partial_t-\partial_t\phi\mathbf{v}_0\cdot\nabla-\mathbf{v}_0\cdot\nabla\phi\partial_t+c_0^2\nabla\phi\cdot\nabla-(\mathbf{v}_0\cdot\nabla\phi)\mathbf{v}_0\cdot\nabla=-D_t\phi D_t+c_0^2\nabla\phi\cdot\nabla.
\]
One can also verify that
\[
\begin{split}
&\square_g\phi\\
=&-c_0^n\partial_t(c_0^{-n}\partial_t\phi)-c_0^n\partial_t(c_0^{-n}v_{0j}\partial_j\phi)-c_0^n\partial_j(c_0^{-n}v_{0j}\partial_t\phi)+c_0^{n}\partial_j(c_0^{-n+2}\partial_j\phi)-c_0^{n}\partial_j(c_0^{-n}v_{0j}v_{0k}\partial_k\phi)\\
=&-D_t^2\phi+c_0^2\Delta\phi-(\nabla\cdot\mathbf{v}_0)D_t\phi+(2-n)c_0^{-1}\nabla c_0\cdot\nabla\phi+nc_0^{-1}D_tc_0D_t\phi.
\end{split}
\]
We can then end up with the following equation for $\alpha$,
\[
\begin{split}
-2(-D_t\phi D_t+c_0^2\nabla\phi\cdot\nabla)\alpha+(D_t^2\phi-c_0^2\Delta\phi)\alpha-4 c_0(\nabla c_0\cdot\nabla\phi)\alpha\\
+D_t\phi(\nabla\cdot\mathbf{v}_0)\alpha+2c_0^{-1}D_t\phi D_tc_0\alpha+ \rho_0^{-1}(-D_t\phi D_t\rho_0+c_0^2\nabla\phi\cdot\nabla\rho_0)\alpha=0,
\end{split}
\]
where we have used the fact
\[
D_t\nabla\phi\cdot\nabla\phi-\nabla D_t\phi\cdot\nabla\phi+(\nabla\phi\cdot\nabla\mathbf{v}_0)\cdot\nabla\phi=0.
\]
If we denote
\[
\partial_s=g^{jk}(\partial_j\phi)\partial_k=-D_t\phi D_t+c_0^2\nabla\phi\cdot\nabla,
\]
we can write the above equation as a transport equation for $\alpha$ along the null geodesic
\[
-2\partial_s\alpha-(n+2)c_0^{-1}\partial_sc_0\alpha+\rho_0^{-1}\partial_s\rho_0\alpha-\square_g\phi\alpha=0.
\]
So
\[
\alpha=C\rho_0^{\frac{1}{2}}c_0^{-\frac{n+2}{2}}e^{-\frac{1}{2}\int_0^s\square_g\phi(\gamma(\sigma))\mathrm{d}\sigma}.
\]
For our purpose, the only thing important about the form of $\alpha$ is that we can take $\alpha$ non-vanishing along the null-geodesic. Construction of $(c,\mathbf{d})$ can also be carried out in a similar way, and we note here that $c^{(0)}$ must be $0$.

If $(h,\mathbf{f})$ is a conormal distribution, then the equation \eqref{lineareq_nonh} has a microlocal solution $(\rho,\mathbf{v})$ whose principal symbol is
\begin{equation}
    \left(\begin{array}{c}
\sigma(\rho)\\
\sigma(\mathbf{v})
\end{array}\right)(\eta)=\sigma(Q)(\eta,\zeta)\left(\begin{array}{c}
\sigma(h)\\
\frac{\sigma(\mathbf{f})}{\rho_0}
\end{array}\right)(\zeta),
\end{equation}
 where either $\eta$ is on the null-bicharacteristics of $\tau\pm c_0|\xi|$ starting from $\zeta\in L^{*,+}$, or $\eta$ is on the bicharacteristics of $\tau$ starting from $\zeta=(0,\xi)$, and for both cases $\zeta\in\mathrm{WF}(h,\mathbf{f})$. Recall that $Q$ is the parametrix of $P$. Roughly speaking, one can decompose
 \[
 \left(\begin{array}{c}
\sigma(h)\\
\frac{\sigma(\mathbf{f})}{\rho_0}
\end{array}\right)(\zeta)
 \]
 into a linear combination of the eigenvectors of $\sigma(P)$,
 \[
  \left(\begin{array}{c}
\sigma(h)\\
\frac{\sigma(\mathbf{f})}{\rho_0}
\end{array}\right)(\zeta)=\alpha_1\left(\begin{array}{c}
\rho_0\\
c_0\hat{\xi}
\end{array}\right)+\alpha_2\left(\begin{array}{c}
\rho_0\\
-c_0\hat{\xi}
\end{array}\right)+\alpha_3\left(\begin{array}{c}
0\\
\hat{\xi^\perp}
\end{array}\right).
 \]
 Then
 \[
 \sigma\left(\frac{P}{\mathrm{i}}\right)^{-1}\left(\begin{array}{c}
\sigma(h)\\
\frac{\sigma(\mathbf{f})}{\rho_0}
\end{array}\right)(\zeta)=\alpha_1(\tau+c_0|\xi|)^{-1}\left(\begin{array}{c}
\rho_0\\
c_0\hat{\xi}
\end{array}\right)+\alpha_2(\tau-c_0|\xi|)^{-1}\left(\begin{array}{c}
\rho_0\\
-c_0\hat{\xi}
\end{array}\right)+\alpha_3\tau^{-1}\left(\begin{array}{c}
0\\
\hat{\xi^\perp}
\end{array}\right)
 \]
 for $\zeta$ not in the characteristic set of $\tau\pm c_0|\xi|$ or $\tau$.
The $\alpha_1,\alpha_2,\alpha_3$ would specify the initial values for the solutions of certain transport equations for the principal symbol of $(\rho,\mathbf{v})$ along different bicharacteristics associated with $H_{\tau+c_0|\xi|}$, $H_{\tau-c_0|\xi|}$ and $H_\tau$ respectively. See \cite{melrose1979lagrangian} for more details.

 \subsection{Construction of distorted plane waves}
Let us construct solutions of the linearized equations that represent distorted plane waves, whose singularities propagate along the bicharacteristics of $\tau\pm c_0|\xi|$.
Let $z_0=(t_0,x_0)\in U$, $\zeta_0=(\tau_0,\xi_0)\in L^{*,+}_{z_0}M$. Take a Riemannian metric $g^+$ on $M$ and let
\[
\mathcal{V}_{\zeta_0,s_0}=\{\eta\in T^*_{z_0}M:\|\eta-\zeta_0\|_{g^+}< s_0,\|\eta\|_{g^+}=\|\zeta_0\|_{g^+}\},
\]
and
\[
\mathcal{W}_{\zeta_0,s_0}=\mathcal{V}_{\zeta_0,s_0}\cap L^{+,*}_{z_0}M.
\]

Denote
\[
\Sigma(\zeta_0,s_0)=\{(z_0,r\eta)\in T^*M;\eta\in\mathcal{V}_{\zeta_0,s_0},r\in\mathbb{R}\setminus\{0\}\},
\]
\[
\Lambda(\zeta_0,s_0)=\{(\gamma_{\eta}(t),r\dot{\gamma}_{\eta}(t))^\flat)\in T^*M;\eta\in\mathcal{W}_{\zeta_0,s_0},t\in(0,+\infty),r\in\mathbb{R}\setminus\{0\}\},
\]
where $\gamma_{\eta}$ is the null-geodesic with $\eta(0)=\pi(\eta),\dot{\eta}(0)^\flat=\eta$.
Notice that $\Lambda(\zeta_0,s_0)$ is the Lagrangian submanifold that is the flowout from $\Sigma(\zeta_0,s_0)\cap\{(\tau,\xi):\tau\pm c_0|\xi|=0\}$ by the Hamiltonian vector field of $H_{\tau\pm c_0|\xi|}$. We will use sources $\mathbf{f}\in\mathcal{I}^{\mu+\frac{1}{2}}(\Sigma(\zeta_0,s_0))$, $h=0$, which can be constructed by microlocally cutting off and regularizing the Delta distribution $\delta_{z_0}$. See \cite{kurylev2018inverse} for more details. We choose $s_0$ sufficiently small such that $\Sigma(\zeta_0,s_0)\cap\Sigma_\tau=\emptyset$ and
\[
\sigma(\mathbf{f})(\zeta_0)=\sigma(\mathbf{f})(-\zeta_0)
\]
is non-vanishing. For this case, there is no singularities propagating along the bicharacteristic of $\tau$.
 Then the solution $(\rho,\mathbf{v})\vert_{M\setminus\{z_0\}}\in \mathcal{I}^{\mu}(\Lambda(\zeta_0,s_0))$ represents a distorted plane wave.

Assume $\zeta\in \mathrm{WF}(\mathbf{f})$. To construct the phase functions in \eqref{FIOsolution} we denote $(\tau',\xi)=(\partial_t\phi_\pm,\nabla\phi_\pm)$ and solve the Hamilton equations
\[
\frac{\partial t}{\partial s}=1,\quad\frac{\partial x}{\partial s}=\mathbf{v}_0\pm c_0\hat{\xi},\quad \frac{\partial \tau'}{\partial s}=-\frac{\partial\mathbf{v}_0}{\partial t}\cdot\xi \mp \partial_tc_0|\xi|,\quad \frac{\partial \xi}{\partial s}=-\partial_x(\mathbf{v}_0\cdot\xi)\mp|\xi|\partial_x c_0,
\]
with initial conditions
\[
t(0)=t_0,\quad x(0)=x_0, \quad (\tau'(0), \xi(0))=\mp \zeta.
\]
The above null bicharacteristics project to the same null geodesic $(t,x)$ with $\frac{\partial t}{\partial s}(0)=1$, $\frac{\partial x}{\partial s}(0)=\mathbf{v}_0-c_0\hat{\zeta}_x$.

So we can take phase functions such that
\[
(D_t,\nabla)\phi_+=-(D_t,\nabla)\phi_-.
\]
Consequently 
\[
\tau\sigma(\rho)(\tau,\xi)+\rho_0\xi\cdot\sigma(\mathbf{v})(\tau,\xi)=0,
\]
for any $\eta=(\tau,\xi)\in \Lambda(\zeta_0,s_0)$.

\section{Nonlinear interactions of distorted plane waves}\label{nonlinearinteraction}

In this section we will use three sources $\mathbf{f}_i\in \mathcal{I}^{\mu+1/2}(\Sigma(\zeta_i,s_0))$, $j=1,2,3$, where $z_i=\pi(\zeta_i)\in V$ and $\zeta_i\in L^{*,+}_{z_i}M$. We can take $\mu$ sufficiently negative such that $\mathbf{f}_i\in C^{5/2}_0(V)$. They generate solutions to the linearized equations $(\rho_j,\mathbf{v}_j)\vert_{M\setminus\{z_j\}}\in \mathcal{I}^{\mu}(\Lambda(\zeta_i,s_0))$. Note that the principal symbol can be written as
\[
\sigma(\rho_i)(\tau_i,\xi_i)=-\alpha_i\tau_i,\quad \sigma(\mathbf{v}_i)(\tau_i,\xi_i)=\alpha_i \frac{c_0^2}{\rho_0}\xi_i,
\]
where $\alpha_i$ is non-vanishing. We assume that $\gamma_{\zeta_i}\cap \gamma_{\zeta_j}=\{q\}\in I(p^-,p^+)$ for any $i\neq j$. We will analyze the nonlinear interaction of these three solutions. The third-order linearization has been done in Section \ref{chmultilinearization} and we will use the same notations there. The newly generated singularities can be used to detect the Lorentzian structure $g$.

 \subsection{Nonlinear interaction of two waves}\label{sec: two waves}
 For ease of notations, we denote $\Lambda_i=\Lambda(\zeta_i,s_0)$. Since there are no cut points, locally one can write $\Lambda_i=N^*K_i$, where $K_i$ is a codimension one submanifold of $M$. Denote $K_{ij}=K_i\cap K_j$ and $\Lambda_{ij}=N^*(K_i\cap K_j)$. By \cite[Lemma 1.1]{greenleaf1993recovering}, we have
\[
(\rho_i,\mathbf{v}_i)\otimes(\rho_j,\mathbf{v}_j)\in\mathcal{I}^{\mu,\mu+\frac{n}{4}}(\Lambda_{ij},\Lambda_i)+\mathcal{I}^{\mu,\mu+\frac{n}{4}}(\Lambda_{ij},\Lambda_j).
\]
For us, $n$ is always equal to $4$.
Moreover, for any $(q,\zeta_{ij})\in \Lambda_{ij}\setminus(\Lambda_i\cup\Lambda_j)$, we can write $\zeta_{ij}=\zeta_i+\zeta_j$ in a unique way such that $\zeta_i\in N_q^*K_i$, $\zeta_j\in N_q^*K_j$. The principal symbol of $(\rho_i,\mathbf{v}_i)\otimes(\rho_j,\mathbf{v}_j)$ on $\Lambda_{ij}\setminus(\Lambda_i\cup\Lambda_j)$ is then
\[
\sigma\left((\rho_i,\mathbf{v}_i)\otimes(\rho_j,\mathbf{v}_j)\right)(\zeta_{ij})=(2\pi)^{-1}\sigma(\rho_i,\mathbf{v}_i)(\zeta_i)\otimes\sigma(\rho_j,\mathbf{v}_j)(\zeta_j).
\]
It is known that $(\Lambda_{ij}\setminus(\Lambda_i\cup\Lambda_j))\cap\{\tau\pm c_0|\xi|\}=\emptyset$. Therefore
\[
\mathrm{WF}(\rho_{ij},\mathbf{v}_{ij})\subset \Lambda_{ij}\cup \Lambda_i\cup\Lambda_j\cup \Lambda^\tau_{ij},
\]
where $\Lambda^\tau_{ij}$ is the flowout of $\Lambda_{ij}$ under the Hamilton flow $H_\tau$.
Away from $\{\tau=0\}\cup\Lambda_i\cup\Lambda_j$,
\[
(\rho_{ij}, \mathbf{v}_{ij})\in \mathcal{I}^{2\mu+\frac{n}{4}}(\Lambda_{ij}).
\]
Denote $\zeta_{ij}=\zeta_i+\zeta_j:=(\tau_i,\xi_i)+(\tau_j,\xi_j)\notin \{\tau=0\}\cup\Lambda_i\cup\Lambda_j$, and note that $P$ is elliptic at such $\zeta_{ij}$.
The principal symbol of $(\rho_{ij},\mathbf{v}_{ij})$ satisfy the equation
\begin{equation}\label{eqrhov1}
\begin{split}
&(\tau_i+\tau_j)\sigma(\rho_{ij})(\zeta_{ij})+\rho_0(\xi_i+\xi_j)\cdot\sigma(\mathbf{v}_{ij})(\zeta_{ij})\\
=&-(\xi_i+\xi_j)\cdot\left(\sigma(\rho_i)\sigma(\mathbf{v}_j)+\sigma(\rho_j)\sigma(\mathbf{v}_i)\right),
\end{split}
\end{equation}
and
\begin{equation}\label{eqrhov2}
\begin{split}
&(\tau_i+\tau_i)\sigma(\mathbf{v}_{ij})(\zeta_{ij})+\frac{c_0^2}{\rho_0}(\xi_i+\xi_j)\sigma(\rho_{ij})(\zeta_{ij})\\
=&-\xi_j\cdot\sigma(\mathbf{v}_i)\sigma(\mathbf{v}_j)-\xi_{i}\cdot\sigma(\mathbf{v}_j)\sigma(\mathbf{v}_{i})-(\gamma-2)\frac{c_0^2}{\rho_0^2}(\xi_i+\xi_j)\sigma(\rho_i)\sigma(\rho_j).
\end{split}
\end{equation}

Using the following matrix operation
\[
\left(\begin{array}{cc}
\tau&-\rho_0\xi^T\\
-\frac{c^2_0}{\rho_0}\xi &\tau
\end{array}\right)\left(\begin{array}{cc}
\tau&\rho_0\xi^T\\
\frac{c^2_0}{\rho_0}\xi &\tau
\end{array}\right)=\left(\begin{array}{cc}
\tau^2-c_0^2|\xi|^2&0\\
0 &\tau^2-c_0^2\xi\xi^T
\end{array}\right),
\]
we obtain
\begin{equation}\label{symbolrholij}
\begin{split}
&(\tau^2-c_0^2|\xi|^2)\sigma(\rho_{ij})(\zeta_{ij})\\
=&-(\tau_i+\tau_j)\left(\sigma(\rho_i)\xi_j\cdot\sigma(\mathbf{v}_j)+\sigma(\mathbf{v}_j)\cdot\xi_i\sigma(\rho_i)+\sigma(\rho_j)\xi_i\cdot\sigma(\mathbf{v}_i)+\sigma(\mathbf{v}_i)\cdot\xi_j\sigma(\rho_j)\right)\\
&+\rho_0(\xi_i+\xi_j)\cdot\left(\sigma(\mathbf{v}_i)\cdot\xi_j\sigma(\mathbf{v}_j)+\sigma(\mathbf{v}_j)\cdot\xi_i\sigma(\mathbf{v}_i)+(\gamma-2)\frac{c_0^2}{\rho_0^2}(\xi_i+\xi_j)\sigma(\rho_i)\sigma(\rho_j)\right),
\end{split}
%
\end{equation}
and
\begin{equation}\label{symbolvlij}
\begin{split}
&(\tau^2-c_0^2\xi\xi^T)\sigma(\mathbf{v}_{ij})(\zeta_{ij})\\
=&\frac{c_0^2}{\rho_0}(\xi_i+\xi_j)\left(\sigma(\rho_i)\xi_j\cdot\sigma(\mathbf{v}_j)+\sigma(\mathbf{v}_j)\cdot\xi_i\sigma(\rho_i)+\sigma(\rho_j)\xi_i\cdot\sigma(\mathbf{v}_i)+\sigma(\mathbf{v}_i)\cdot\xi_j\sigma(\rho_j)\right)\\
&-(\tau_i+\tau_j)\left(\sigma(\mathbf{v}_i)\cdot\xi_j\sigma(\mathbf{v}_j)+\sigma(\mathbf{v}_j)\cdot\xi_i\sigma(\mathbf{v}_i)+(\gamma-2)\frac{c_0^2}{\rho_0^2}(\xi_i+\xi_j)\sigma(\rho_i)\sigma(\rho_j)\right).
\end{split}
\end{equation}

Since away from $\Lambda_i\cup \Lambda_j$, 
\[
(\rho_i,\mathbf{v}_i)\otimes(\rho_j,\mathbf{v}_j)\in\mathcal{I}^{2\mu+\frac{n}{4}}(\Lambda_{ij}),
\]
we have (cf. \cite[Proposition 2.1]{greenleaf1993recovering})
\[
(\rho_{ij}, \mathbf{v}_{ij})\in \mathcal{I}^{2\mu+\frac{n}{4}+\frac{1}{2},-\frac{1}{2}}(\Lambda_{ij},\Lambda^\tau_{ij}).
\]
Notice that $\Lambda^\tau_{ij}=N^*K^\tau_{ij}$ for some codimension $1$ submanifold $K^\tau_{ij}$.
 \subsection{Nonlinear interaction of three waves}
 First we assume $K_i,K_j$, $i\neq j$, intersect transversally at $K_{ij}$ which is a codimension $2$ submanifold, and for distinct $i,j,k$, $K_{ij},K_k$ intersect transversally at $K_{123}$ which is a codimension $3$ submanifold.  To analyze the propagation of singularities for $(\rho_{123},\mathbf{v}_{123})$, we first need to analyze the wavefront set of $(h_{123},\mathbf{f}_{123})$ at $N^*K_{123}\cap \{\tau\pm c_0|\xi|=0\}$. Take $q\in K_{123}$ and $\zeta\in N_q^*K_{123}$ with $|\zeta|_{g}=0$. Write $\zeta=\zeta_1+\zeta_1+\zeta_3$ with $\zeta_i\in N^*_qK_i$ in a unique way. Note that by our assumption, $\zeta_1,\zeta_2,\zeta_3$ are linearly independent.
 
 Away from $\Lambda_1\cup\Lambda_2\cup \Lambda_3\cup\Lambda_{12}\cup\Lambda_{13}\cup \Lambda_{23}\cup\Lambda^\tau_{12}\cup \Lambda^\tau_{13}\cup \Lambda^\tau_{23}\cup N^*(K_{12}^\tau\cap K_3)\cup N^*(K_{13}^\tau\cap K_2)\cup N^*(K_{23}^\tau\cap K_1)$, we have (cf. \cite[Lemma 3.6]{lassas2018inverse})
\[
(\mathbf{v}_{ij},\rho_{ij})\otimes (\mathbf{v}_k,\rho_k)\in \mathcal{I}^{3\mu+\frac{n}{4}-\frac{1}{2}}(\Lambda_{123}),
\]
and also
\[
(\rho_1,\mathbf{v}_1)\otimes (\rho_2,\mathbf{v}_2)\otimes (\rho_3,\mathbf{v}_3)\in \mathcal{I}^{3\mu+\frac{n}{4}-\frac{1}{2}}(\Lambda_{123}).
\]
Also by \cite[Lemma 4.1]{wang2019inverse}
\[
(D_t,\nabla)\left((\mathbf{v}_{ij},\rho_{ij})\otimes (\mathbf{v}_k,\rho_k)\right),\quad (D_t,\nabla)\left((\rho_1,\mathbf{v}_1)\otimes (\rho_2,\mathbf{v}_2)\otimes (\rho_3,\mathbf{v}_3)\right)\in \mathcal{I}^{3\mu+\frac{n}{4}+\frac{1}{2}}(\Lambda_{123}).
\]

Note that $K^\tau_{ij}\cap K_k$ has codimension $2$. For $\zeta_{ij}\in N_q^*K^\tau_{ij}$ and $\zeta_k\in N_q^*K_k$, $ \zeta_k+\alpha\zeta_{ij}$ could be light-like with $\alpha=0$ and possibly another nonzero $\alpha$. The flowout of $K^\tau_{ij}\cap K_k$ under the Hamilton flow $H_{\tau\pm c_0|\xi|}$ then has codimension $1$.

Denote
\[
\Lambda^{(3)}=N^*(K^\tau_{12}\cap K_3)\cup N^*(K^\tau_{13}\cap K_2)\cup N^*(K^\tau_{23}\cap K_1),
\]
\[
\Lambda^{(3),g}=\Lambda_{\tau\pm c_0|\xi|}(K^\tau_{12}\cap K_3)\cup \Lambda_{\tau\pm c_0|\xi|}(K^\tau_{13}\cap K_2)\cup \Lambda_{\tau\pm c_0|\xi|}(K^\tau_{23}\cap K_1).
\]
Then away from $ \bigcup_{i=1}^3\Lambda_i\cup\bigcup_{i\neq j}\Lambda_{ij}\cup\bigcup_{i\neq j}\Lambda_{ij}^\tau\cup \Lambda^{(3)}$,
\[
(h_{123},\mathbf{f}_{123})\in \mathcal{I}^{3\mu+\frac{n}{4}+\frac{1}{2}}(\Lambda_{123}),
\]
where $\Lambda_{123}=N^*K_{123}$, and therefore away from $ \bigcup_{i=1}^3\Lambda_i\cup\bigcup_{i\neq j}\Lambda_{ij}\cup\bigcup_{i\neq j}\Lambda_{ij}^\tau\cup \Lambda^{(3)}\cup \Lambda^{(3),g}$,
\[
(\rho_{123},\mathbf{v}_{123})\in \mathcal{I}^{3\mu+\frac{n}{4},-\frac{1}{2}}(\Lambda_{123},\Lambda_{123}^g)\cup \mathcal{I}^{3\mu+\frac{n}{4},-\frac{1}{2}}(\Lambda_{123},\Lambda_{123}^\tau).
\]
By possibly shrinking $V$, we can assume, for the rest of the paper, without loss of generality that $V$ is union of integral curves of $\partial_t+\mathbf{v}_0\cdot\nabla$, and then $V\cap \pi(\bigcup_{i\neq j}\Lambda_{ij}\cup\bigcup_{i\neq j}\Lambda_{ij}^\tau\cup\Lambda^{(3)}\cup\Lambda_{123}^\tau)=\emptyset$. We will carefully analyze the singularties of $(\rho_{123},\mathbf{v}_{123})$ in the set $V\setminus \pi(\Lambda^{\tau,(3)}\cup \cup_{i=1}^3\Lambda_i)$. The exceptional set $\pi(\Lambda^{(3),g}\cup \bigcup_{i=1}^3\Lambda_i)$ tends to a set whose Hausdorff dimension is at most $1$ as $s_0\rightarrow 0$.
 \subsection{Three-to-one scattering relation}
 We recall the concept of three-to-one scattering relation introduced in \cite{feizmohammadi2021inverse} and used in, for example, \cite{hintz2024inverse,nursultanov2025determining}.
 \begin{definition}
 A relation $\mathcal{R}\subset (L^{*,+}V)^4$ is a three-to-one scattering relation if the following two conditions hold:
 \begin{enumerate}
 \item if $(\eta_0,\eta_1,\eta_2,\eta_3)\in\mathcal{R}$ then
 \[
 \gamma_{\eta_0}(\mathbb{R}^-)\cap\bigcap_{j=1}^3\gamma_{\eta_j}(\mathbb{R}^+)\neq \emptyset.
 \]
 \item The set $\mathcal{R}$ constains all $\mathcal{R}\subset (L^{*,+}V)^4$ which satisty
  \begin{enumerate}
  \item The bicharacterstics through $\eta_j$, $j=0,1,2,3$, are distinct;
  \item There are $y\in M$ and $s_0<0$, $s_j>0$, $j=1,2,3$ such that $y=\gamma_{\eta_0}(s_0)$ and $y=\gamma_{\eta_j}(s_j)$;
  \item Writing $\zeta_j=\dot{\gamma}_{\eta_j}(s_j)^\flat$, it holds that $\zeta_0\in \mathrm{span}(\zeta_1,\zeta_2,\zeta_3)$.
   \end{enumerate}
 \end{enumerate}
 \end{definition}
 
 Define a conical piece of $L^{*,+}V$ associated to a three-to-one scattering relation $\mathcal{R}$ and $\eta_1,\eta_2,\eta_3\in L^{*,+}V$ by
 \[
\mathrm{CP}(\eta_1,\eta_2,\eta_3)=\{\eta_0\in L^{*,+}V; (\eta_0,\eta_1,\eta_2,\eta_3)\in\mathcal{R} \}.
 \]
 We always take $\eta_1,\eta_2,\eta_3$ such that the bicharacterstics through $\eta_j$, $j=0,1,2,3$, are distinct.
 

 \begin{proposition}\label{nonvanishing}
For any $\eta_0,\eta_1\in L^{*,+}V$ such that $\gamma_{\eta_0}(\mathbb{R}^-)\cap\gamma_{\eta_1}(\mathbb{R}^+)=\{q\}\notin V$, there exist $\eta_2,\eta_3\in L^{*,+}V$ such that $\eta_0\in\mathrm{CP}(\eta_1,\eta_2,\eta_3)$, and $\sigma(\rho_{123},\mathbf{v}_{123})(\eta)$ is generically nonvanishing for $\eta\in\mathrm{CP}(\eta_1,\eta_2,\eta_3)\cap U$, where $U\subset T^*M$ is a small neighborhood of $\eta_0$.
 \end{proposition}
 \begin{proof}
 Assume that $q=\gamma_{\eta_0}(s_0)=\gamma_{\eta_1}(s_1)$. Denote $\zeta_0=\dot{\gamma}_{\eta_0}(s_0)^\flat$, $\zeta_1=\dot{\gamma}_{\eta_1}(s_1)^\flat$.
Choosing proper local coordinates at the point $q$, we can assume without loss of generality that $\zeta_0=\beta_0\theta_0$, $\zeta_1=\beta_1\theta_1$ with
 \[
 \theta_0=(-c_0,\pm\sqrt{1-r_0^2},r_0,0),\quad  \theta_1=(-c_0,1,0,0).
 \]
 We then take
 \[
 \theta_2=(-c_0,\sqrt{1-\varsigma^2},\varsigma,0),\quad   \theta_3=(-c_0,\sqrt{1-\varsigma^2},-\varsigma,0),
 \]
 with $\varsigma$ sufficiently small.
 Notice that $ \theta_0$ can expressed as the unique linear combination of $ \theta_1, \theta_2, \theta_3$ as
 \[
 \theta_0=\alpha_1 \theta_1+\alpha_2 \theta_2+\alpha_3 \theta_3,
 \]
 where
\[
\alpha_1=\frac{-\sqrt{1-\varsigma^2}\pm\sqrt{1-r_0^2}}{1-\sqrt{1-\varsigma^2}},\quad \alpha_2=\frac{1\mp\sqrt{1-r_0^2}}{2(1-\sqrt{1-\varsigma^2})}+\frac{r_0}{2\varsigma},\quad \alpha_3=\frac{1\mp\sqrt{1-r_0^2}}{2(1-\sqrt{1-\varsigma^2})}-\frac{r_0}{2\varsigma}.
\]

Denote $b(r_0)=1\mp\sqrt{1-r_0^2}\neq 0$. 
We consider the small $\varsigma$ asymptotics
 \begin{align*}
 \alpha_1=&-2b(r_0)\varsigma^{-2}+1+\frac{b(r_0)}{2}+\mathcal{O}(\varsigma),\\
 \alpha_2=&b(r_0)\varsigma^{-2}+\frac{r_0}{2}\varsigma^{-1}-\frac{b(r_0)}{4}+\mathcal{O}(\varsigma),\\
 \alpha_3=&b(r_0)\varsigma^{-2}-\frac{r_0}{2}\varsigma^{-1}-\frac{b(r_0)}{4}+\mathcal{O}(\varsigma).
 \end{align*}
 Without loss of generality, we take $\zeta=\theta_0$, $\zeta_i=\alpha_i\theta_i$, $i=1,2,3$. Then $\zeta$ can be written to be the unique linear combination of $\zeta_1,\zeta_2,\zeta_3$
 \[
 \zeta=\zeta_1+\zeta_2+\zeta_3.
 \]
Writing $\zeta=(\tau,\xi)$ and $\zeta_i=(\tau_i,\xi_i)$, we have the asymptotics in small $\varsigma$

\[
  \begin{split}
\tau_1=&2c_0b(r_0)\varsigma^{-2}-c_0-\frac{c_0b(r_0)}{2}+\mathcal{O}(\varsigma),\\
\tau_2=&-c_0b(r_0)\varsigma^{-2}-\frac{c_0r_0}{2}\varsigma^{-1}+\frac{c_0b(r_0)}{4}+\mathcal{O}(\varsigma),\\
\tau_3=&-c_0b(r_0)\varsigma^{-2}+\frac{c_0r_0}{2}\varsigma^{-1}+\frac{c_0b(r_0)}{4}+\mathcal{O}(\varsigma),
 \end{split}
\]
 and
 \[
 \begin{split}
\xi_1&=-2b(r_0)\mathbf{e}_1\varsigma^{-2}+\mathbf{e}_1+\frac{b(r_0)}{2}\mathbf{e}_1+\mathcal{O}(\varsigma),\\
\xi_2&=b(r_0)\mathbf{e}_1\varsigma^{-2}+(\frac{r_0}{2}\mathbf{e}_1+b(r_0)\mathbf{e}_2)\varsigma^{-1}-\frac{3}{4}b(r_0)\mathbf{e}_1+\frac{r_0}{2}\mathbf{e}_2+\mathcal{O}(\varsigma),\\
\xi_3&=b(r_0)\mathbf{e}_1\varsigma^{-2}-(\frac{r_0}{2}\mathbf{e}_1+b(r_0)\mathbf{e}_2)\varsigma^{-1}-\frac{3}{4}b(r_0)\mathbf{e}_1+\frac{r_0}{2}\mathbf{e}_2+\mathcal{O}(\varsigma).
\end{split}
\]


Then we have the following asymptotic behaviors.
\begin{align*}
& \tau_1+\tau_2=c_0b(r_0)\varsigma^{-2}-\frac{c_0r_0}{2}\varsigma^{-1}-c_0-\frac{c_0b(r_0)}{4}+\mathcal{O}(\varsigma),\\
& \xi_1+\xi_2=-b(r_0)\mathbf{e}_1\varsigma^{-2}+(\frac{r_0}{2}\mathbf{e}_1+b(r_0)\mathbf{e}_2)\varsigma^{-1}+\mathbf{e}_1-\frac{b(r_0)}{4}\mathbf{e}_1+\frac{r_0}{2}\mathbf{e}_2+\mathcal{O}(\varsigma),\\
&\tau_1+\tau_3=c_0b(r_0)\varsigma^{-2}+\frac{c_0r_0}{2}\varsigma^{-1}-c_0-\frac{c_0b(r_0)}{4}+\mathcal{O}(\varsigma),\\
& \xi_1+\xi_3=-b(r_0)\mathbf{e}_1\varsigma^{-2}-(\frac{r_0}{2}\mathbf{e}_1+b(r_0)\mathbf{e}_2)\varsigma^{-1}+\mathbf{e}_1-\frac{b(r_0)}{4}\mathbf{e}_1+\frac{r_0}{2}\mathbf{e}_2+\mathcal{O}(\varsigma),\\
& \tau_2+\tau_3=-2c_0b(r_0)\varsigma^{-2}+\frac{c_0b(r_0)}{2}+\mathcal{O}(\varsigma),\\
& \xi_2+\xi_3=2b(r_0)\mathbf{e}_1\varsigma^{-2}-\frac{3b(r_0)}{2}\mathbf{e}_1+r_0\mathbf{e}_2+\mathcal{O}(\varsigma).
\end{align*}


 Note that we can take
 \[
 \sigma(\rho_i)(\zeta_i)=-\rho_0\tau_ia_i(\zeta_i)\varsigma^2,\quad  \sigma(\mathbf{v}_i)(\zeta_i)=c_0^2\xi_ia_i(\zeta_i)\varsigma^2.
 \]
%
%
%
Then the equation \eqref{symbolrholij} reduces to
 \begin{align*}
&((\tau_1+\tau_2)^2-c_0^2|\xi_1+\xi_2|^2)\sigma(\rho_{12})(\zeta_1+\zeta_2)\\
=&(\tau_1+\tau_2)^2\rho_0(\tau_1\tau_2+c_0^2\xi_1\cdot\xi_2)\alpha_1(\zeta_1)\alpha_2(\zeta_2)\varsigma^{4}+\rho_0|\xi_1+\xi_2|^2(c_0^2\xi_1\cdot\xi_2+(\gamma-2)\tau_1\tau_2)\alpha_1(\zeta_1)\alpha_2(\zeta_2)\varsigma^4.
\end{align*}

Using the following asymptotic behaviors
\[
\tau_1\tau_2=-2c_0^2b(r_0)^2\varsigma^{-4}-c_0^2r_0b(r_0)\varsigma^{-3}+c_0^2(b(r_0)+b(r_0)^2)\varsigma^{-2}+\mathcal{O}(\varsigma^{-1}),
\]
\[
\xi_1\cdot\xi_2=-2b(r_0)^2\varsigma^{-4}-r_0b(r_0)\varsigma^{-3}+(b(r_0)+2b(r_0)^2)\varsigma^{-2}+\mathcal{O}(\varsigma^{-1}),
\]
\[
(\tau_1+\tau_2)^2=c_0^2b(r_0)^2\varsigma^{-4}-c_0^2r_0b(r_0)\varsigma^{-3}+c_0^2(\frac{r_0^2}{4}-2b(r_0)-\frac{b(r_0)^2}{2})\varsigma^{-2}+\mathcal{O}(\varsigma^{-1}),
\]
\[
|\xi_1+\xi_2|^2=c_0^2b(r_0)^2\varsigma^{-4}-c_0^2r_0b(r_0)\varsigma^{-3}+c_0^2(\frac{r_0^2}{4}-2b(r_0)+\frac{3b(r_0)^2}{2})\varsigma^{-2}+\mathcal{O}(\varsigma^{-1}),
\]
we obtain
\begin{align*}
&((\tau_1+\tau_2)^2-c_0^2|\xi_1+\xi_2|^2)\sigma(\rho_{12})(\zeta_1+\zeta_2)\\
=&-2(\gamma+1)c_0^4\rho_0b(r_0)^4a_1(\zeta_1)a_2(\zeta_2)\varsigma^{-4}+(\gamma+1)c_0^4\rho_0r_0b(r_0)^3a_1(\zeta_1)a_2(\zeta_2)\varsigma^{-3}\\
&+\frac{\gamma+1}{2}c_0^4\rho_0r_0^2b(r_0)^2a_1(\zeta_1)a_2(\zeta_2)\varsigma^{-2}+(-4\gamma+6)c_0^4\rho_0b(r_0)^4a_1(\zeta_1)a_2(\zeta_2)\varsigma^{-2}\\
&+5(\gamma+1)c_0^4\rho_0b(r_0)^3a_1(\zeta_1)a_2(\zeta_2)\varsigma^{-2}+\mathcal{O}(\varsigma^{-1})a_1(\zeta_1)a_2(\zeta_2).
\end{align*}
Using the asymptotic behaviors
\begin{alignat*}{2}
(\tau_1+\tau_2)^2-c_0^2|\xi_1+\xi_2|^2=&c_0^2(\alpha_1+\alpha_2)^2-c_0^2(\alpha_1+\alpha_2\sqrt{1-\varsigma^2})^2-c_0^2\alpha_2^2\varsigma^2\\
=&c_0^2(2\alpha_1+\alpha_2+\alpha_2\sqrt{1-\varsigma^2})\alpha_2(1-\sqrt{1-\varsigma^2})-c_0^2\alpha_2^2\varsigma^2\\
=&-2c_0^2b(r_0)^2\varsigma^{-2}-c_0^2r_0b(r_0)\varsigma^{-1}+\frac{1}{2}c_0^2b(r_0)^2+c_0^2b(r_0)+\mathcal{O}(\varsigma),\\
\end{alignat*}
we end up with
\begin{align*}
\sigma(\rho_{12})(\zeta_1+\zeta_2)=&(\gamma+1)c_0^2\rho_0b(r_0)^2a_1(\zeta_1)a_2(\zeta_2)\varsigma^{-2}-(\gamma+1)c_0^2\rho_0r_0b(r_0)a_1(\zeta_1)a_2(\zeta_2)\varsigma^{-1}\\
&+\frac{\gamma+1}{4}c_0^2\rho_0r_0^2a_1(\zeta_1)a_2(\zeta_2)+\frac{9\gamma-11}{4}c_0^2\rho_0b(r_0)^2a_1(\zeta_1)a_2(\zeta_2)\\
&-2(\gamma+1)c_0^2\rho_0b(r_0)a_1(\zeta_1)a_2(\zeta_2)+\mathcal{O}(\varsigma)a_1(\zeta_1)a_2(\zeta_2).
 \end{align*}
 Similarly we have
\begin{align*}
\sigma(\rho_{13})(\zeta_1+\zeta_3)=&(\gamma+1)c_0^2\rho_0b(r_0)^2a_1(\zeta_1)a_3(\zeta_3)\varsigma^{-2}+(\gamma+1)c_0^2\rho_0r_0b(r_0)a_1(\zeta_1)a_3(\zeta_3)\varsigma^{-1}\\
&+\frac{\gamma+1}{4}c_0^2\rho_0r_0^2a_1(\zeta_1)a_3(\zeta_3)+\frac{9\gamma-11}{4}c_0^2\rho_0b(r_0)^2a_1(\zeta_1)a_3(\zeta_3)\\
&-2(\gamma+1)c_0^2\rho_0b(r_0)a_1(\zeta_1)a_3(\zeta_3)+\mathcal{O}(\varsigma)a_1(\zeta_1)a_3(\zeta_3).
 \end{align*}
 Note
 \[
 \frac{c_0^2(\xi_1+\xi_2)}{\tau_1+\tau_2}=-c_0\mathbf{e}_1-\frac{1}{2}c_0\mathbf{e}_1\varsigma^2+c_0\mathbf{e}_2\varsigma+\frac{c_0r_0}{b(r_0)}\mathbf{e}_2\varsigma^2+\mathcal{O}(\varsigma^3).
 \]
Then we use the equations \eqref{eqrhov2} to obtain

\begin{align*}
&(\tau_1+\tau_2)\sigma(\mathbf{v}_{12})(\zeta_1+\zeta_2)\\
=&-\frac{c_0^2}{\rho_0}\sigma(\rho_{12})(\xi_1+\xi_2)-c_0^2\varsigma^4(\xi_1\cdot\xi_2)(\xi_1+\xi_2)-\varsigma^4(\gamma-2)c_0^2\tau_1\tau_2(\xi_1+\xi_2).
 \end{align*}
 So
 \begin{align*}
&\sigma(\mathbf{v}_{12})(\zeta_1+\zeta_2)\\
=&(\gamma+1)c_0^3b(r_0)^2a_1a_2\mathbf{e}_1\varsigma^{-2}-(\gamma+1)c_0^3r_0b(r_0)a_1a_2\mathbf{e}_1\varsigma^{-1}-(\gamma+1)c_0^3b(r_0)^2a_1a_2\mathbf{e}_2\varsigma^{-1}\\
&+\frac{\gamma+1}{4}c_0^3r_0^2a_1a_2\mathbf{e}_1-2(\gamma+1)c_0^3b(r_0)a_1a_2\mathbf{e}_1+\frac{3\gamma-1}{4}c_0^3b(r_0)^2a_1a_2\mathbf{e}_1+\mathcal{O}(\varsigma)a_1a_2,
 \end{align*}
 and similarly
  \begin{align*}
&\sigma(\mathbf{v}_{13})(\zeta_1+\zeta_3)\\
=&(\gamma+1)c_0^3b(r_0)^2a_1a_3\mathbf{e}_1\varsigma^{-2}+(\gamma+1)c_0^3r_0b(r_0)a_1a_3\mathbf{e}_1\varsigma^{-1}+(\gamma+1)c_0^3b(r_0)^2a_1a_3\mathbf{e}_2\varsigma^{-1}\\
&+\frac{\gamma+1}{4}c_0^3r_0^2a_1a_3\mathbf{e}_1-2(\gamma+1)c_0^3b(r_0)a_1a_3\mathbf{e}_1+\frac{3\gamma-1}{4}c_0^3b(r_0)^2a_1a_3\mathbf{e}_1+\mathcal{O}(\varsigma)a_1a_3.
 \end{align*}


Next we calculate the principal symbol of $(\rho_{23},\mathbf{v}_{23})$. Note
\[
\tau_2\tau_3=c_0^2b(r_0)^2\varsigma^{-4}-\frac{c_0^2b(r_0)^2}{2}\varsigma^{-2}-\frac{c_0^2r_0^2}{4}\varsigma^{-2}+\mathcal{O}(\varsigma^{-1}),
\]
\[
\xi_2\cdot\xi_3=b(r_0)^2\varsigma^{-4}-\frac{5}{2}b(r_0)^2\varsigma^{-2}-\frac{r_0^2}{4}\varsigma^{-2}+\mathcal{O}(\varsigma^{-1}),
\]
\[
(\tau_2+\tau_3)^2=4c_0^2b(r_0)^2\varsigma^{-4}-2c_0^2b(r_0)^2\varsigma^{-2}+\mathcal{O}(\varsigma^{-1}),
\]
\[
|\xi_2+\xi_3|^2=4b(r_0)^2\varsigma^{-4}-6b(r_0)^2\varsigma^{-2}+\mathcal{O}(\varsigma^{-1}).
\]
By tedious calculation we obtain
\begin{align*}
&((\tau_2+\tau_3)^2-c_0^2|\xi_2+\xi_3|^2)\sigma(\rho_{23})(\zeta_2+\zeta_3)\\
=&4(\gamma+1)c_0^4\rho_0b(r_0)^4a_2a_3\varsigma^{-4}-(\gamma+1)c_0^4\rho_0r_0^2b(r_0)^2a_2a_3\varsigma^{-2}-4(\gamma+5)c_0^4\rho_0b(r_0)^4a_2a_3\varsigma^{-2}+\mathcal{O}(\varsigma^{-1})a_2a_3.
\end{align*}
Using the asymptotic behaviors
\begin{alignat*}{2}
(\tau_2+\tau_3)^2-c_0^2|\xi_2+\xi_3|^2&=4c_0^2b(r_0)^2\varsigma^{-2}-2c_0^2b(r_0)^2-c_0^2r_0^2+\mathcal{O}(\varsigma).
\end{alignat*}
So
\[
\sigma(\rho_{23})=(\gamma+1)c_0^2\rho_0b(r_0)^2a_2a_3\varsigma^{-2}-\frac{\gamma+9}{2}c_0^2\rho_0b(r_0)^2a_2a_3+\mathcal{O}(\varsigma)a_2a_3.
\]
Then we use the equations \eqref{eqrhov2} to obtain

\begin{align*}
&(\tau_2+\tau_3)\sigma(\mathbf{v}_{23})(\zeta_2+\zeta_3)\\
=&-\frac{c_0^2}{\rho_0}\sigma(\rho_{23})(\xi_2+\xi_3)-c_0^2\varsigma^4(\xi_2\cdot\xi_3)(\xi_2+\xi_3)-\varsigma^4(\gamma-2)c_0^2\tau_2\tau_3(\xi_2+\xi_3)\\
=&-2(\gamma+1)b(r_0)^3\mathbf{e}_1\varsigma^{-4}a_2a_3+\frac{\gamma+25}{2}b(r_0)^3\mathbf{e}_1\varsigma^{-2}a_2a_3-(\gamma+1)r_0b(r_0)^2\mathbf{e}_2\varsigma^{-2}a_2a_3+\mathcal{O}(\varsigma^{-1})a_2a_3.
 \end{align*}
Then
\[
\sigma(\mathbf{v}_{23})=(\gamma+1)c_0^3b(r_0)^2\mathbf{e}_1a_2a_3\varsigma^{-2}-6c_0^3b(r_0)^2a_2a_3\mathbf{e}_1+\frac{\gamma+1}{2}c_0^3r_0b(r_0)a_2a_3\mathbf{e}_2+\mathcal{O}(\varsigma)a_2a_3.
\]
~\\

Next we use the equations \eqref{eqrhov1}, \eqref{eqrhov2}
and also
\[
\tau_i\sigma(\rho_i)(\zeta_i)+\rho_0\xi_i\cdot\sigma(\mathbf{v}_i)(\zeta_i)=0,
\]
\[
\tau_i\sigma(\mathbf{v}_i)(\zeta_i)+\frac{c_0^2}{\rho_0}\xi_i\sigma(\rho_i)(\zeta_i)=0,
\]
to write
 \begin{align*}
\sigma(\frac{\mathbf{f}_{123}}{\rho_0})(\zeta)=&-\frac{1}{2}\sigma(\mathbf{v}_{\sigma(1)\sigma(2)})\cdot\xi_{\sigma(3)}\sigma(\mathbf{v}_{\sigma(3)})-\frac{1}{2}\sigma(\mathbf{v}_{\sigma(1)})\cdot(\xi_{\sigma(2)}+\xi_{\sigma(3)})\sigma(\mathbf{v}_{\sigma(2)\sigma(3)})\\
 &-\frac{\gamma-2}{2}\frac{c_0^2}{\rho^2_0}\left(\xi_{\sigma(1)}\sigma(\rho_{\sigma(1)})\sigma(\rho_{\sigma(2)\sigma(3)})+\sigma(\rho_{\sigma(1)})(\xi_{\sigma(2)}+\xi_{\sigma(3)})\sigma(\rho_{\sigma(2)\sigma(3)})\right)\\
 &-(\gamma-2)^2\frac{c_0^2}{\rho_0^3}(\xi_1+\xi_2+\xi_3)\sigma(\rho_1)\sigma(\rho_2)\sigma(\rho_3)\\
= &-\frac{1}{2}\sigma(\mathbf{v}_{\sigma(1)\sigma(2)})\cdot\xi_{\sigma(3)}\sigma(\mathbf{v}_{\sigma(3)})-\frac{1}{2}\sigma(\mathbf{v}_{\sigma(1)})\cdot(\xi_{\sigma(2)}+\xi_{\sigma(3)})\sigma(\mathbf{v}_{\sigma(2)\sigma(3)})\\
&-\frac{1}{2}\rho_0^{-1}\left((\tau_{\sigma(1)}+\tau_{\sigma(2)})\sigma(\rho_{\sigma(1)\sigma(2)})+\rho_0(\xi_{\sigma(1)}+\xi_{\sigma(2)})\cdot\sigma(\mathbf{v}_{\sigma(1)\sigma(2)})\right)\sigma(\mathbf{v}_{\sigma(3)})\\
&-\frac{1}{2}\rho_0^{-1}\left((\tau_{\sigma(1)}+\tau_{\sigma(2)})\sigma(\mathbf{v}_{\sigma(1)\sigma(2)})+\frac{c_0^2}{\rho_0}(\xi_{\sigma(1)}+\xi_{\sigma(2)})\sigma(\rho_{\sigma(1)\sigma(2)})\right)\sigma(\rho_{\sigma(3)})\\
&-\frac{1}{2}\rho_0^{-1}\left(\tau_{\sigma(3)}\sigma(\rho_{\sigma(3)})+\xi_{\sigma(3)}\cdot\sigma(\mathbf{v}_{\sigma(3)})\right)\sigma(\mathbf{v}_{\sigma(1)\sigma(2)})\\
&-\frac{1}{2}\rho_0^{-1}\left(\tau_{\sigma(3)}\sigma(\mathbf{v}_{\sigma(3)})+\xi_{\sigma(3)}\sigma(\rho_{\sigma(3)})\right)\sigma(\rho_{\sigma(1)\sigma(2)})\\
&-\frac{1}{2}\rho_0^{-1}(\xi_{\sigma(1)}+\xi_{\sigma(2)})\cdot\left(\sigma(\rho_{\sigma(1)})\sigma(\mathbf{v}_{\sigma(2)})+\sigma(\rho_{\sigma(2)})\sigma(\mathbf{v}_{\sigma(1)})\right)\sigma(\mathbf{v}_{\sigma(3)})\\
&-\frac{1}{2}\rho_0^{-1}\left(\xi_{\sigma(2)}\cdot\sigma(\mathbf{v}_{\sigma(1)})\sigma(\mathbf{v}_{\sigma(2)})+\xi_{\sigma(1)}\cdot\sigma(\mathbf{v}_{\sigma(2)})\sigma(\mathbf{v}_{\sigma(1)})\right)\sigma(\rho_{\sigma(3)})\\
&-\frac{1}{2}(\gamma-2)\frac{c_0^2}{\rho_0^3}(\xi_{\sigma(1)}+\xi_{\sigma(2)})\sigma(\rho_{\sigma(1)})\sigma(\rho_{\sigma(2)})\sigma(\rho_{\sigma(3)})\\
 &-\frac{\gamma-2}{2}\frac{c_0^2}{\rho^2_0}\left(\xi_{\sigma(1)}\sigma(\rho_{\sigma(1)})\sigma(\rho_{\sigma(2)\sigma(3)})+\sigma(\rho_{\sigma(1)})(\xi_{\sigma(2)}+\xi_{\sigma(3)})\sigma(\rho_{\sigma(2)\sigma(3)})\right)\\
 &-(\gamma-2)(\gamma-3)\frac{c_0^2}{\rho_0^3}(\xi_1+\xi_2+\xi_3)\sigma(\rho_1)\sigma(\rho_2)\sigma(\rho_3).
 \end{align*}
 
 Noticing
 \[
   \begin{split}
& (\xi_{\sigma(1)}+\xi_{\sigma(2)})\cdot\left(\sigma(\rho_{\sigma(1)})\sigma(\mathbf{v}_{\sigma(2)})+\sigma(\rho_{\sigma(2)})\sigma(\mathbf{v}_{\sigma(1)})\right)\sigma(\mathbf{v}_{\sigma(3)})\\
& +\left(\xi_{\sigma(2)}\cdot\sigma(\mathbf{v}_{\sigma(1)})\sigma(\mathbf{v}_{\sigma(2)})+\xi_{\sigma(1)}\cdot\sigma(\mathbf{v}_{\sigma(2)})\sigma(\mathbf{v}_{\sigma(1)})\right)\sigma(\rho_{\sigma(3)})\\
=&(\xi_{\sigma(1)}+\xi_{\sigma(2)}+\xi_{\sigma(3)})\cdot\sigma(\mathbf{v}_{\sigma(2)})\sigma(\mathbf{v}_{\sigma(3)})\sigma(\rho_{\sigma(1)})+(\xi_{\sigma(1)}+\xi_{\sigma(2)}+\xi_{\sigma(3)})\cdot\sigma(\mathbf{v}_{\sigma(1)})\sigma(\mathbf{v}_{\sigma(3)})\sigma(\rho_{\sigma(2)})\\
=&2\xi\cdot\sigma(\mathbf{v}_{\sigma(1)})\sigma(\mathbf{v}_{\sigma(2)})\sigma(\rho_{\sigma(3)}).
   \end{split}
 \]
 
 Then
  \begin{align*}
\sigma(\frac{\mathbf{f}_{123}}{\rho_0})(\zeta)=&-\frac{1}{2}\xi\cdot\sigma(\mathbf{v}_{\sigma(1)\sigma(2)})\sigma(\mathbf{v}_{\sigma(3)})-\frac{1}{2}\xi\cdot\sigma(\mathbf{v}_{\sigma(3)})\sigma(\mathbf{v}_{\sigma(1)\sigma(2)})\\
 &-\frac{1}{2}\rho_0^{-1}\tau\sigma(\rho_{\sigma(1)\sigma(2)})\sigma(\mathbf{v}_{\sigma(3)})-\frac{1}{2}\rho_0^{-1}\tau\sigma(\mathbf{v}_{\sigma(1)\sigma(2)})\sigma(\rho_{\sigma(3)})\\
 &-\frac{\gamma-1}{2}\frac{c_0^2}{\rho_0^2}\xi\sigma(\rho_{\sigma(1)\sigma(2)})\sigma(\rho_{\sigma(3)})\\
 &-2\xi\cdot\sigma(\mathbf{v}_{\sigma(1)})\sigma(\mathbf{v}_{\sigma(2)})\sigma(\rho_{\sigma(3)})\\
 &-(\gamma-2)(\gamma-1)\frac{c_0^2}{\rho_0^3}(\xi_1+\xi_2+\xi_3)\sigma(\rho_1)\sigma(\rho_2)\sigma(\rho_3).
 \end{align*}
 Also we have
 \[
\sigma(h_{123})(\zeta)=-\frac{1}{2}\xi\cdot(\sigma(\rho_{\sigma(1)})\sigma(\mathbf{v}_{\sigma(2)\sigma(3)})+\sigma(\rho_{\sigma(1)\sigma(2)})\sigma(\mathbf{v}_{\sigma(3)})).
 \]
 Now the projection of $(h_{123},\mathbf{f}_{123}/\rho_0)$ onto the eigenspace of $\sigma(P)$ associated with eigenvalues $\tau\pm c_0|\xi|$ is
 \[
   \begin{split}
 & \rho_0\xi\cdot\sigma(\frac{\mathbf{f}_{123}}{\rho_0})(\zeta)- \tau\sigma(h_{123})(\zeta)\\
=& -\rho_0\xi\cdot\sigma(\mathbf{v}_{\sigma(1)\sigma(2)})\xi\cdot\sigma(\mathbf{v}_{\sigma(3)})-\frac{\gamma-1}{2}\rho_0^{-1}\tau^2\sigma(\rho_{\sigma(1)\sigma(2)})\sigma(\rho_{\sigma(3)})\\
&-2\xi\cdot\sigma(\mathbf{v}_{\sigma(1)})\xi\cdot\sigma(\mathbf{v}_{\sigma(2)})\sigma(\rho_{\sigma(3)})\\
&-(\gamma-2)(\gamma-1)\frac{c_0^2}{\rho_0^2}\sigma(\rho_1)\sigma(\rho_2)\sigma(\rho_3).
   \end{split}
 \]

Note that
 \begin{align*}
\xi\cdot\xi_1=&-2b(r_0)(1-b(r_0))\varsigma^{-2}+(1-b(r_0))+\frac{b(r_0)}{2}(1-b(r_0))+\mathcal{O}(\varsigma),\\
\xi\cdot\xi_2=&b(r_0)(1-b(r_0))\varsigma^{-2}+\frac{r_0}{2}(1-b(r_0))\varsigma^{-1}+r_0b(r_0)\varsigma^{-1}-\frac{3}{4}b(r_0)(1-b(r_0))+\frac{r_0^2}{2}+\mathcal{O}(\varsigma),\\
\xi\cdot\xi_3=&b(r_0)(1-b(r_0))\varsigma^{-2}-\frac{r_0}{2}(1-b(r_0))\varsigma^{-1}-r_0b(r_0)\varsigma^{-1}-\frac{3}{4}b(r_0)(1-b(r_0))+\frac{r_0^2}{2}+\mathcal{O}(\varsigma).
\end{align*}
 We calculate
 \[
 \begin{split}
 &\frac{1}{2}\tau^2\rho_0^{-1}\sigma(\rho_{\sigma(1)\sigma(2)})\sigma(\rho_{\sigma(3)})\\
 =&\frac{3}{2}(\gamma+1)\rho_0c_0^5r_0^2b(r_0)a_1a_2a_3-3(\gamma+1)\rho_0c_0^5b(r_0)^2a_1a_2a_3+\frac{11\gamma+7}{2}\rho_0c_0^5b(r_0)^3a_1a_2a_3+\mathcal{O}(\varsigma)a_1a_2a_3,
  \end{split}
 \]
 \[
    \begin{split}
&\frac{1}{2}\xi\cdot\sigma(\mathbf{v}_{\sigma(1)\sigma(2)})\xi\cdot\sigma(\mathbf{v}_{\sigma(3)})\\
=&\frac{3(\gamma+1)}{2}\rho_0c_0^5r_0^2b(r_0)a_1a_2a_3-\frac{3\gamma+23}{2}\rho_0c_0^5r_0^2b(r_0)^3a_1a_2a_3+3(\gamma+1)r_0^2b(r_0)^2a_1a_2a_3\\
&+\frac{\gamma+21}{2}\rho_0c_0^5b(r_0)^3a_1a_2a_3
-2(\gamma+1)\rho_0c_0^5r_0b(r_0)^3a_1a_2a_3\\
&+2(\gamma+1)\rho_0c_0^5r_0b(r_0)^4a_1a_2a_3-\frac{\gamma+1}{2}\rho_0c_0^5r_0^4b(r_0)a_1a_2a_3-3(\gamma+1)\rho_0c_0^5b(r_0)^2a_1a_2a_3\\
&+\mathcal{O}(\varsigma)a_1a_2a_3,
   \end{split}
 \]

 \begin{flalign*}
&\quad\quad2\xi\cdot\sigma(\mathbf{v}_{\sigma(1)})\xi\cdot\sigma(\mathbf{v}_{\sigma(2)})\sigma(\rho_{\sigma(3)})=-24\rho_0c_0^5b(r_0)^3(1-r_0^2)a_1a_2a_3+\mathcal{O}(\varsigma)a_1a_2a_3,&
\end{flalign*}
 \begin{flalign*}
 &\quad\quad\frac{c_0^2}{\rho_0^2}\sigma(\rho_1)\sigma(\rho_2)\sigma(\rho_3)=-2\rho_0c_0^5b(r_0)^3a_1a_2a_3+\mathcal{O}(\varsigma)a_1a_2a_3.&
\end{flalign*}
 
 When $r_0=1$, $b(r_0)=1$, we have
 \[
 \xi\cdot\sigma(\mathbf{f}_{123})(\zeta)- \tau\sigma(h_{123})(\zeta)=\left(-2(\gamma-1)(\gamma+3)+\mathcal{O}(\varsigma)\right)a_1(\zeta_1)a_2(\zeta_2)a_3(\zeta_3),
 \]
 which is nonzero for sufficiently small $\varsigma$. Denoting $\xi=(\cos\beta,\sin\beta,0)$, we can express $\xi\cdot\sigma(\mathbf{f}_{123})- \tau\sigma(h_{123})=A(\beta)r_0c_0^5a_1a_2a_3+\mathcal{O}(\varsigma)a_1a_2a_3$, where $\beta\in (\delta,2\pi-\delta)$ with some small $\delta>0$. Here $\delta$ and $\varsigma$ are chosen such that $(-\varsigma,\varsigma)\subset(-\sin\delta,\sin\delta)\subset (-|r_0|,|r_0|)$. Keep in mind that $\cos\beta$ should be away from $1$. Since $\xi\cdot\sigma(\mathbf{f}_{123})- \tau\sigma(h_{123})$ is analytic in $\beta$,
it is generically non-vanishing for $\beta\in (\delta,2\pi-\delta)$. 
Therefore $\sigma(\rho_{123})(\eta)$ is generically non-vanishing for $\eta\in\mathrm{CP}(\eta_1,\eta_2,\eta_3)$ sufficiently close to $\eta_0$.
\end{proof}

\section{Determination of the conformal structure}\label{chconformal}
\begin{proposition}\label{hausdorffthreewaves}
 If $\mathrm{CP}(\eta_1,\eta_2,\eta_3)=\emptyset$, the Hausdorff dimension of $\mathrm{sing\,supp}(\rho_{123},\mathbf{v}_{123})$ tends to a set $\mathcal{Y}(\eta_1,\eta_2,\eta_3)$ whose Hausdorff dimension is at most $1$.
 \end{proposition}
\begin{proof}
If $\gamma_{\eta_j}\cap\gamma_{\eta_k}=\emptyset$ for any $j\neq k$, then for $s_0>0$ sufficiently small, $\mathrm{sing\,supp}(\rho_{123},\mathbf{v}_{123})$ is contained in $\pi(\Lambda_1\cup\Lambda_2\cup \Lambda_3)$, which tends to a set of Hausdorff dimension $1$.

Now assume $\gamma_{\eta_1}\cap\gamma_{\eta_2}\cap\gamma_{\eta_3}=\{q\}$. Note that $\zeta_1,\zeta_2,\zeta_3$ can not be linearly dependent by \cite[Lemma 6.12]{feizmohammadi2021inverse}. By the analysis in the previous section, the singular support of $(\rho_{123},\mathbf{v}_{123})\vert_V$ is contained in $\pi(\Lambda^{(3),g}\cup \bigcup_{i=1}^3\Lambda_i)$ which tends to a union of finitely many bicharacteristic curves, whose Hausdorff dimension is $1$, as $s_0\rightarrow 0$.

Next without loss of generality assume $\gamma_{\eta_1}\cap\gamma_{\eta_2}=\{q\}$ and $q\notin\gamma_{\eta_3}$. It is possible that $\gamma_{\eta_3}\cap\Lambda^\tau_{12}\neq\emptyset$. Still the set $\Lambda_{\tau\pm c_0|\xi|}(K_{12}^\tau \cap K_3)\setminus (\Lambda_{12}^\tau \cap \Lambda_3)$ tends to a set of Hausdorff dimension at most $1$ by previous analysis.
\end{proof}

Denote
\[
C(q)=\{(\gamma_{\zeta}(s),\dot{\gamma}_{\zeta}(s)^\flat):\zeta\in L^{*,+}_q(M),s\in\mathbb{R}^+\}
\]
to be the flowout from the point $q$. 
Clearly
\[
\pi(C(q))=\mathcal{L}^+(q)=\{\gamma_{\zeta}(t)\in M;\zeta\in L_q^{*,+} M,t\geq 0\}\subset M
\]
is the future directed light-cone emanating from the point $q$. 
Denote
\[
E_V(q)=\{\eta\in \overline{C}(q):\pi(\eta)\in V\}
\]
to be the set of earliest observations in $V$ associated with the point $q$.
 
  \begin{proposition}\label{prop4}
 Fix $\eta_1\in L^{*,+}V$ and $q\in\gamma_{\eta_1}(\mathbb{R}^+)\cap I(p_-,p_+)$, then $\eta_0^{(1)},\eta_0^{(2)}\in L^{*,+}V$, where $\eta_0^{(1)},\eta_0^{(2)}$ is not on the null bi-characteristic curve through $\eta_1$, satisfy $\eta_0^{(1)},\eta_0^{(2)}\subset E_V(q)$ if and only if there is a continuous one-parameter family $(\eta_2(\varpi),\eta_3(\varpi))$ such that $\eta_2(\varpi),\eta_3(\varpi)\in L^{*,+}V$,
 \[
 \gamma_{\eta_1}\cap \gamma_{\eta_2(\varpi)}\cap  \gamma_{\eta_3(\varpi)}=\{q\},
 \]
 and $\mathrm{CP}(\eta_1,\eta_2(\varpi),\eta_3(\varpi))\neq\emptyset$ for any $\varpi\in [\varpi^{(1)},\varpi^{(2)}]$ with
 \[
\eta_0^{(1)}\in\mathrm{CP}(\eta_1,\eta_2(\varpi^{(1)}),\eta_3(\varpi^{(1)})),\quad \eta_0^{(2)}\in\mathrm{CP}(\eta_1,\eta_2(\varpi^{(2)}),\eta_3(\varpi^{(2)})).
\]
Furthermore, for each $\varpi$ there exists $\eta_0(\varpi)\in \mathrm{CP}(\eta_1,\eta_2(\varpi),\eta_3(\varpi))$, which is continuous in $\varpi$, such that $\eta_0(\varpi^{(1)})=\eta_0^{(1)}$, $\eta_0(\varpi^{(2)})=\eta_0^{(2)}$ and $\sigma(\rho_{123},\mathbf{v}_{123})(\eta_0;\eta_1,\eta_2(\varpi),\eta_3(\varpi))$ is generically non-vanishing for $\eta_0\in \mathrm{CP}(\eta_1,\eta_2(\varpi),\eta_3(\varpi))$ sufficiently close to $\eta_0(\varpi)$.
 \end{proposition}
 \begin{proof}
  The sufficiency of the condition is obvious. We only need to prove the necessity. Denote
  \[
  \zeta_1=\dot{\gamma}_{\eta_1}(s_1)^\flat,
  \]
  where $\gamma_{\eta_1}(s_1)=q$.
Without loss of generality, assume under certain coordinates at $q$ we can take
  \[
\zeta_0(\varpi)=(-1,\pm\sqrt{1-r_0^2(\varpi)},r_0(\varpi)\cos\theta(\varpi)),r_0(\varpi)\sin\theta(\varpi))),\quad  \zeta_1=(-1,1,0,0),
 \]
 where $\zeta_0(\varpi^{(1)})=\zeta_0^{(1)}$, $\zeta_0(\varpi^{(2)})=\zeta_0^{(2)}$ and $r_0,\theta$ are continuous in $\varpi\in [\varpi^{(1)},\varpi^{(2)}]$.
Then we take
 \[
 \zeta_2(\varpi)=(-1,\sqrt{1-\varsigma^2},\varsigma\cos\theta(\varpi),\varsigma\sin\theta(\varpi)),\quad   \zeta_3(\varpi)=(-1,\sqrt{1-\varsigma^2},-\varsigma\cos\theta(\varpi),-\varsigma\sin\theta(\varpi)),
 \]
 with $\varsigma=\varsigma(\varpi)$ sufficiently small, and $\eta_j(\varpi)=\dot{\gamma}_{\zeta_j(\varpi)}(s_j(\varpi))\in L^{*,+}V$ with some $s_j(\varpi)\in\mathbb{R}^-$. Noticing that the non-vanishing property of the principal symbol is guaranteed by the proof of Proposition \ref{nonvanishing}, the proof is complete.
 \end{proof}

  \begin{lemma}
 If $\mathrm{CP}(\eta_1,\eta_2,\eta_3)\neq\emptyset$ and $\mathrm{CP}(\eta_1,\eta_2,\eta_3')\neq\emptyset$, then
 \[
\gamma_{\eta_1}\cap \gamma_{\eta_2}\cap \gamma_{\eta_3}=\gamma_{\eta_1}\cap \gamma_{\eta_2}\cap \gamma_{\eta_3'}\neq\emptyset.
 \]
 \end{lemma}
 
 The above lemma is obvious and means that we can always fix two null-geodesics and vary the third one without changing the intersecting point. 
 
Use the same notations as in Proposition \ref{prop4}. For any $\varpi$, there is a small neighborhood $U(\varpi)$ of $(\eta_2(\varpi),\eta_3(\varpi))$ in $L^{*,+}(V)\times L^{*,+}(V)$ such that for any $(\eta_2,\eta_3)\subset U(\varpi)$, if $\gamma_{\eta_1}\cap\gamma_{\eta_2}\cap \gamma_{\eta_3}\neq \emptyset$ then $\mathrm{CP}(\eta_1,\eta_2,\eta_3)\neq\emptyset$. By compactness argument, we can cover $(\eta_2([\varpi^{(1)},\varpi^{(2)}]),\eta_3([\varpi^{(1)},\varpi^{(2)}]))$ by finitely many open set $U(\varpi^{(k)})$, $k=1,2,\cdots, K$. Then $(\eta_2([\varpi^{(1)},\varpi^{(2)}]),\eta_3([\varpi^{(1)},\varpi^{(2)}]))$ has a $\delta$-neighborhood such that for any $(\eta_2,\eta_3)$ in this neighborhood, we have if $\gamma_{\eta_1}\cap\gamma_{\eta_2}\cap \gamma_{\eta_3}\neq \emptyset$, then $\mathrm{CP}(\eta_1,\eta_2,\eta_3)\neq \emptyset$, and $\sigma(\rho_{123},\mathbf{v}_{123})(\eta_0;\eta_1,\eta_2,\eta_3)$ is generically non-vanishing at least for $\eta_0$ in an open subset of $\mathrm{CP}(\eta_1,\eta_2,\eta_3)$. Then we can have a (different) continuous one-parameter family $(\eta_2(\varpi),\eta_3(\varpi))$ with a division $\varpi^{(1)}=\varpi_0<\varpi_2<\cdots<\varpi_M=\varpi^{(2)}$ of $[\varpi^{(1)},\varpi^{(2)}]$ such that for each $k$, 
\[
\eta_2(\varpi)=\eta_2(\varpi_{k-1})\quad\text{or}\quad \eta_3(\varpi)=\eta_3(\varpi_{k-1})
\]
for any $\varpi\in (\varpi_{k-1},\varpi_k)$. Still we keep the property
\[
\gamma_{\eta_1}\cap \gamma_{\eta_2(\varpi)}\cap \gamma_{\eta_3(\varpi)}=\{q\},
\]
for any $\varpi\in [\varpi^{(1)},\varpi^{(2)}]$. Notice that on each $[\varpi_{k-1},\varpi_k]$, either $\eta_2(\varpi)$ or $\eta_3(\varpi)$ remains the same.

\begin{lemma}\label{Cq1q2}
If $q_1\neq q_2$, then $\pi(C(q_1)\cap C(q_2))$ has Hausdorff dimension at most $1$.
\end{lemma}
\begin{proof}
If $\eta_0\in C(q_1)\cap C(q_2)$, then $\eta_0$ must be on the null bicharacteristic through both $q_1$ and $q_2$.
\end{proof}
  
  Therefore if $\emptyset\neq\mathrm{CP}(\eta_1,\eta_2,\eta_3)\subset E_V(q)$ for some $q$, then $q$ is uniquely determined since $\pi(\mathrm{CP}(\eta_1,\eta_2,\eta_3))$ has Hausdorff dimension $2$. If $\sigma(\rho_{123},\mathbf{v}_{123})(\eta_0)$ is nonvanishing for some $\eta_0\in\mathrm{CP}(\eta_1,\eta_2,\eta_3)$, by successively fixing two of $\{\eta_j\}_{j=1}^3$ and perturbing the third one, one can recover a conical piece $U$ of $E_V(q)$ such that $\pi(U)$ has Hausdorff dimension $3$ (Here $\eta_1,\eta_2,\eta_3$ are not necessarily close). For each $q\in I_g(p^-,p^+)$, there exists some $(\eta_1,\eta_2,\eta_3)$, which can be continuously varied in this way so that we can recover the whole $E_V(q)$ by above discussion. The incomplete pieces can be incorporated into the complete ones thanks to Lemma \ref{Cq1q2}.
  
  Therefore the source-to-solution map determines the earliest light observation sets
  \[
 \{ \mathcal{E}_V(q)\vert q\in I_g(p_-,p_+)\}=\{\pi(E_V(q))\vert q\in I_g(p_-,p_+)\}.
  \]
  Then by \cite[Theorem 1.2]{kurylev2018inverse}, we can recover the conformal structure of the Lorentzian metric
  \[
 \left( \begin{array}{cc}
 -1+c_0^{-2}\mathbf{v}_0^T\mathbf{v}_0&-c_0^{-2}\mathbf{v}_0^T\\
 -c_0^{-2}\mathbf{v}_0 &c_0^{-2}I
  \end{array}\right)
  \]
  in the set $I_g(p^-,p^+)$. To be more precise, denote $L_V^{(j)}$, $j=1,2$, to be the source-to-solution map associated with $(\rho_0^{(j)},\mathbf{v}_0^{(j)})$. The conclusion of this section is summarized in the following theorem.
  \begin{theorem}\label{thm: determination of conformal class}
  If $L_V^{(1)}=L_V^{(2)}$, then there exists a diffeomorphism $\Phi$ and a scalar function $\lambda\in C^\infty(I_{g^{(1)}}(p^-,p^+))$ such that
  \[
  \Phi(I_{g^{(1)}}(p^-,p^+))=I_{g^{(2)}}(p^-,p^+),
  \]
 where $\Phi=\mathrm{Id}$, $\lambda=0$ on $V$ and 
  \[
  g^{(1)}=e^\lambda \Phi^*g^{(2)}\quad\text{in }I_{g^{(1)}}(p^-,p^+).
  \]
  Here
   \[
 g^{(j)}=\left( \begin{array}{cc}
 -1+(c^{(j)}_0)^{-2}(\mathbf{v}^{(j)}_0)^T\mathbf{v}_0^{(j)}&-(c^{(j)}_0)^{-2}(\mathbf{v}^{(j)}_0)^T\\
 -(c^{(j)}_0)^{-2}\mathbf{v}^{(j)}_0 &(c^{(j)}_0)^{-2}I
  \end{array}\right).
  \]
  \end{theorem}

\section{Advective flow from two acoustic waves}\label{chadvective}
In this section, we show that if the background solution has nonzero vorticity, then the interaction of two acoustic waves generate an advective flow in the generic setting. This will be used later to determine the exact background solution.

\subsection{Principal symbol computation}
We divide the multilinearized momentum equation by $\rho_0$ and isolate the right-hand side source vector $\mathbf{f}_{ij}$. After simplification, the second-order system \eqref{momentumij} and \eqref{continuityij} governing the interaction between primary waves $i$ and $j$ is given by
\[
D_t \rho_{ij} + \rho_0 \nabla \cdot \mathbf{v}_{ij} + \mathbf{v}_{ij}\cdot \nabla \rho_0 + \rho_{ij}\nabla \cdot \mathbf{v}_0 = h_{ij},
\]
\[
D_t \mathbf{v}_{ij} + \mathbf{v}_{ij}\cdot \nabla \mathbf{v}_0 + \frac{c_0^2}{\rho_0}\nabla \rho_{ij}
+ \frac{c_0^2}{\rho_0^2}(\gamma-2)\nabla \rho_0 \rho_{ij}
= \mathbf{f}_{ij},
\]
where we recall that $D_t = \partial_t + \mathbf{v}_0 \cdot \nabla$.
The source terms $\mathbf{f}_{ij}$ and $h_{ij}$ on the right-hand sides are:
\begin{align*}
h_{ij}
&= -\rho_i \nabla \cdot \mathbf{v}_j - \mathbf{v}_j\cdot \nabla \rho_i
-\rho_j \nabla \cdot \mathbf{v}_i - \mathbf{v}_i\cdot \nabla \rho_j,\\
\mathbf{f}_{ij}
&= -\mathbf{v}_i\cdot \nabla \mathbf{v}_j - \mathbf{v}_j\cdot \nabla \mathbf{v}_i
- A\gamma(\gamma-2)\rho_0^{\gamma-3}\nabla(\rho_i\rho_j)
- A\gamma(\gamma-2)(\gamma-3)\rho_0^{\gamma-4}\nabla \rho_0\, \rho_i\rho_j.
\end{align*}
We can rewrite the equation into
\begin{equation}\label{eq: rho_ij simplified}
    D_t \rho_{ij} + \rho_0 \nabla \cdot \mathbf{v}_{ij} + \mathbf{v}_{ij}\cdot \nabla \rho_0 + \rho_{ij}\nabla \cdot \mathbf{v}_0 = -\nabla \cdot (\rho_i \mathbf{v}_j + \rho_j \mathbf{v}_i),
\end{equation}
and
\begin{equation}\label{eq: v_ij simplified}
    \begin{split}
        &D_t \mathbf{v}_{ij} + \mathbf{v}_{ij}\cdot \nabla \mathbf{v}_0 + \nabla\left(\frac{c_0^2}{\rho_0} \rho_{ij}\right)\\
        &= -\nabla\left(\mathbf{v}_i \cdot \mathbf{v}_j + (\gamma - 2)\frac{c_0^2}{\rho_0^2}\rho_i \rho_j\right) + \mathbf{v}_i \times (\nabla \times \mathbf{v}_j) + \mathbf{v}_j \times (\nabla \times \mathbf{v}_i).
    \end{split}
\end{equation}

Assume that $\gamma_i$ and $\gamma_j$ intersect at some point $q$. Recall the computations for linearized  equation in Section \ref{sec:linearizedeq}. The principal symbol of $(\rho_k,\mathbf{v}_k)$ on $\Lambda_k$ should be proportional to the eigenvector. Assume $(\rho_i, \mathbf{v}_i)$ corresponds to ``positive" acoustic wave, meaning $\tau_i + c_0|\xi_i| = 0$, and
\[
\sigma(\rho_i)(\tau_i, \xi_i) = \alpha_i c_0|\xi_i|,\quad \sigma(\mathbf{v}_i)(\tau_i, \xi_i) = \alpha_i \frac{c_0^2}{\rho_0}\xi_i;
\]
and assume $(\rho_j, \mathbf{v}_j)$ corresponds to ``negative" acoustic wave, meaning $\tau_j - c_0|\xi_j| = 0$, and
\[
 \sigma(\rho_j)(\tau_j, \xi_j) = -\alpha_j c_0|\xi_j|,\quad \sigma(\mathbf{v}_j)(\tau_j, \xi_j) = \alpha_j \frac{c_0^2}{\rho_0}\xi_j.
\]
Denote $\xi = \xi_i + \xi_j$ and $\tau = \tau_i + \tau_j$, advective singularity occurs at $\tau = \tau_i+\tau_j = 0$. This requires $|\xi_i| = |\xi_j|$. We shall denote $\sigma(\mathbf{v}_{ij})$ as its principal symbol on $\Lambda_{ij} \backslash \Lambda_{ij}^\tau$, and $\sigma^1(\mathbf{v}_{ij})$ as its principal symbol on $\Lambda_{ij}^\tau \backslash \Lambda_{ij}$, recall the transmission condition from Remark 2.7 and Section 3 of \cite{melrose1979lagrangian}:
\[
\mathrm{i}\tau \sigma(\mathbf{v}_{ij}) = \sigma^1(\mathbf{v}_{ij}) \qquad \text{on} \quad \Lambda_{ij} \cap \Lambda_{ij}^\tau.
\]
Finally, we assume the acoustic waves have classical symbols (this can be achieved by constructing the source terms to have classical symbols), and denote $\sigma_\ell$ the homogeneous of degree $\ell$ terms in the symbol. Let $\mu_k + 1/2$ be the principal order for $(\rho_k, \mathbf{v}_k)$ (so that $(\rho_k, \mathbf{v}_k) \in \mathcal{I}^{\mu_k}(\Lambda_k)$), and denote $\mu = \mu_i + \mu_j +1$.

From the end of Section \ref{sec: two waves}, we know $(\rho_{ij}, \mathbf{v}_{ij}) \in \mathcal{I}^\mu(\Lambda_{ij}\backslash \Lambda_{ij}^\tau)$.
Consider the principal symbol on $\Lambda_{ij} \backslash \Lambda_{ij}^\tau$ from the continuity equation \eqref{eq: rho_ij simplified}:
\begin{align*}
    \tau \sigma_\mu(\rho_{ij}) + \rho_0 \xi \cdot \sigma_\mu(\mathbf{v}_{ij}) &= - \xi \cdot (\sigma_{\mu_i+1/2}(\rho_i)\sigma_{\mu_j+1/2}(\mathbf{v}_j) + \sigma_{\mu_j+1/2}(\rho_j)\sigma_{\mu_i+1/2}(\mathbf{v}_i))\\
    &= -(\xi_i + \xi_j) \cdot \left[\alpha_i\alpha_j\frac{c_0^3}{\rho_0}(|\xi_i|\xi_j - |\xi_j|\xi_i)\right].
\end{align*}
In particular, if we take limit of the RHS approaching $\Lambda_{ij}\cap \Lambda_{ij}^\tau$, meaning we rescale $\xi_i$ and $\xi_j$ such that $|\xi_i| - |\xi_j| \to 0$ and so $\tau = \tau_i + \tau_j \to 0$, we obtain
\begin{align*}
    \rho_0 \xi \cdot \sigma_\mu(\mathbf{v}_{ij}) &= -(\xi_i + \xi_j) \cdot \left[\alpha_i\alpha_j\frac{c_0^3}{\rho_0}|\xi_i|(\xi_j - \xi_i)\right]\\
    &=\alpha_i\alpha_j\frac{c_0^3}{\rho_0}|\xi_i|(|\xi_i|^2-|\xi_j|^2)\\
    &=0.
\end{align*}
The above computation implies the longitudinal part of $\sigma_\mu(\mathbf{v}_{ij})$, namely $\xi \cdot \sigma_\mu(\mathbf{v}_{ij})$, must vanish as it approaches the intersection. In other words, we have
\[
\mathrm{i}\tau \sigma_\mu(\mathbf{v}_{ij}) \cdot \xi = 0 \qquad \text{on} \quad \Lambda_{ij} \cap \Lambda_{ij}^\tau.
\]

Now consider the principal symbol of the momentum equation \eqref{eq: v_ij simplified} on $\Lambda_{ij} \backslash \Lambda_{ij}^\tau$:
\[
\mathrm{i}\tau \sigma_\mu(\mathbf{v}_{ij}) + \mathrm{i}\frac{c_0^2}{\rho_0}\sigma_\mu(\rho_{ij})\xi  = -\mathrm{i}\frac{c_0^2}{\rho_0}\sigma_\mu(B)\xi + \sigma_{\mu+1}(\mathbf{v}_i \times (\nabla \times \mathbf{v}_j) + \mathbf{v}_j \times (\nabla \times \mathbf{v}_i)),
\]
where $B = \frac{\rho_0}{c_0^2}\mathbf{v}_i \cdot \mathbf{v}_j + (\gamma-2)\rho_0^{-1}\rho_i\rho_j$.
Note that $\sigma_{\mu_k+1/2+1}(\nabla \times \mathbf{v}_k) = 0$ because $\sigma_{\mu_k+1/2}(\mathbf{v}_k)$ is proportional to $\xi_k$. Use the result that $\mathrm{i}\tau\sigma_\mu(\mathbf{v}_{ij})$ is orthogonal to $\xi$ on $\Lambda_{ij}^\tau \cap \Lambda_{ij}$, we have
\[
\mathrm{i}\tau \sigma_\mu(\mathbf{v}_{ij}) = 0, \qquad \sigma_\mu(\rho_{ij}) = -\sigma_\mu(B) \qquad \text{on} \quad \Lambda_{ij}^\tau \cap \Lambda_{ij}.
\]
In particular, this suggests that we need to look at one degree lower. For the momentum equation \eqref{eq: v_ij simplified}, we gather terms of order $\mu$:
\begin{align*}
    &\mathrm{i}\tau \sigma_{\mu-1}(\mathbf{v}_{ij}) + \mathrm{i}\frac{c_0^2}{\rho_0}\sigma_{\mu-1}(\rho_{ij})\xi +\nabla\left(\frac{c_0^2}{\rho_0}\right)\sigma_\mu(\rho_{ij}) \\
    &= -\mathrm{i}\frac{c_0^2}{\rho_0}\sigma_{\mu-1}(B)\xi -\nabla\left(\frac{c_0^2}{\rho_0}\right)\sigma_{\mu}(B) + \sigma_{\mu}(\mathbf{v}_i \times (\nabla \times \mathbf{v}_j) + \mathbf{v}_j \times (\nabla \times \mathbf{v}_i)).
\end{align*}
On $\Lambda_{ij}^\tau \cap 
\Lambda_{ij}$, use the fact that $\sigma_{\mu}(\rho_{ij}) = -\sigma_\mu(B)$ and take projection $\Pi_\xi$ onto the orthogonal complement of $\xi$:
\[
\mathrm{i}\tau \Pi_\xi\sigma_{\mu-1}(\mathbf{v}_{ij}) = \Pi_\xi \sigma_{\mu}(\mathbf{v}_i \times (\nabla \times \mathbf{v}_j) + \mathbf{v}_j \times (\nabla \times \mathbf{v}_i)).
\]

\subsection{WKB approximation}
To evaluate the singularity for this symmetric curl term, we use WKB approximation. Specifically, consider the WKB expansions of the primary acoustic waves into the source terms:
\[
\rho_k \sim e^{i\omega \phi_k}\bigl(a_k^{(0)} + \omega^{-1} a_k^{(1)} + \cdots \bigr),
\]
\[
\mathbf{v}_k \sim e^{i\omega \phi_k}\bigl(\mathbf{b}_k^{(0)} + \omega^{-1} \mathbf{b}_k^{(1)} + \cdots \bigr).
\]
From the first order linearization, the principal amplitudes are
\[
a_k^{(0)} = -\alpha_k D_t\phi_k\quad\text{and}\quad\mathbf{b}_k^{(0)} = \frac{c_0^2}{\rho_0}\alpha_k \nabla \phi_k,
\]
where $\alpha^k$ is the scalar transport amplitude.

If we substitute them into the expression, the derivative needs to hit the phase function to generate $O(\omega)$ term, which gives
\[
\mathbf{b}_i^{(0)} \times (\nabla \phi_j \times \mathbf{b}_j^{(0)}) + \mathbf{b}_j^{(0)} \times (\nabla \phi_i \times \mathbf{b}_i^{(0)}).
\]
However, note that $\mathbf{b}_k^{(0)}$ is proportional to $\nabla \phi_k$, so the $O(\omega)$ terms cancel out, matching the principal symbol computation previously.

We now start to evaluate the sub-principal contribution of the cross products. Expanding $\nabla \times \mathbf{v}_j$ and collecting $\mathcal{O}(\omega^0)$ terms yield:
\[
\nabla \times \mathbf{b}_j^{(0)} + \mathrm{i} \nabla\phi_j \times \mathbf{b}_j^{(1)}.
\]

To resolve $\mathbf{b}_j^{(1)}$, recall \eqref{transporteqb0}:
\[
D_t \mathbf{b}_j^{(0)} + \mathbf{b}_j^{(0)} \cdot \nabla \mathbf{v}_0 + \nabla \left( \frac{c_0^2}{\rho_0} a_0^j \right) + \mathrm{i} \mathbf{b}_j^{(1)} D_t \phi_j + \mathrm{i} a_j^{(1)} \frac{c_0^2}{\rho_0} \nabla\phi_j = 0.
\]
We isolate $\mathrm{i} \nabla\phi_j \times \mathbf{b}_j^{(1)}$ by taking the cross product with $\frac{-1}{D_t \phi_j} \nabla\phi_j$. This operation conveniently annihilates the $a_1^j$ term:
\[
\mathrm{i} \nabla\phi_j \times \mathbf{b}_j^{(1)} = \frac{-1}{D_t \phi_j} \nabla\phi_j \times [D_t(X_j \nabla\phi_j) + X_j \nabla\phi_j \cdot \nabla \mathbf{v}_0 - \nabla(X_j D_t \phi_j)],
\]
where we denote $X_k = \frac{c_0^2}{\rho_0}\alpha^k$ for notation simplicity.

We apply the identity $D_t \nabla\phi_j + \nabla\phi_j \cdot \nabla \mathbf{v}_0 = \nabla D_t \phi_j - \nabla\phi_j \times (\nabla \times \mathbf{v}_0)$ to the first two terms in the bracket:
\[
D_t(X_j \nabla\phi_j) + X_j \nabla\phi_j \cdot \nabla \mathbf{v}_0 = D_t X_j \nabla\phi_j + X_j \nabla D_t \phi_j - X_j \nabla\phi_j \times (\nabla \times \mathbf{v}_0).
\]
Subtracting the third term $\nabla(X_j D_t \phi_j) = \nabla X_j D_t \phi_j + X_j \nabla D_t \phi_j$, the bracket simplifies to
\[
D_t X_j \nabla\phi_j - \nabla X_j D_t \phi_j - X_j \nabla\phi_j \times (\nabla \times \mathbf{v}_0).
\]
Taking cross product of this and $\frac{-1}{D_t \phi_j} \nabla\phi_j$, the parallel term $D_t X_j \nabla\phi_j$ vanishes. We obtain
\[
\mathrm{i} \nabla\phi_j \times \mathbf{b}_j^{(1)} = \nabla\phi_j \times \nabla X_j + \frac{X_j}{D_t \phi_j} \nabla\phi_j \times (\nabla\phi_j \times (\nabla \times \mathbf{v}_0)).
\]
We add the remaining curl term $\nabla \times \mathbf{b}_j^{(0)} = \nabla \times (X_j \nabla\phi_j) = \nabla X_j \times \nabla\phi_j = -\nabla\phi_j \times \nabla X_j$. This perfectly cancels the first term, leaving the $\mathcal{O}(\omega^0)$ terms for $\nabla \times \mathbf{v}_j$ to be simply:
\[
\frac{X_j}{D_t \phi_j} \nabla\phi_j \times (\nabla\phi_j \times (\nabla \times \mathbf{v}_0)).
\]

The interaction $\mathbf{v}_i \times (\nabla \times \mathbf{v}_j)$ at $\mathcal{O}(\omega^0)$ level is thus the cross product of $\mathbf{b}_i^{(0)}$ and the above term:
\[
\frac{X_i X_j}{D_t \phi_j} \nabla\phi_i \times \left[ \nabla\phi_j \times (\nabla\phi_j \times (\nabla \times \mathbf{v}_0)) \right].
\]
Applying the vector triple product $A \times (B \times M) = B(A \cdot M) - M(A \cdot B)$ with $B = \nabla\phi_j$, the first resulting term is proportional to $\nabla\phi_j$. We thus obtain the $\mathcal{O}(\omega^0)$ term for $\mathbf{v}_i \times (\nabla \times \mathbf{v}_j)$ is
\[
\frac{X_i X_j}{D_t \phi_j} \nabla\phi_j \left[ \nabla\phi_i \cdot (\nabla\phi_j \times (\nabla \times \mathbf{v}_0)) \right]
-\frac{X_i X_j}{D_t \phi_j} (\nabla\phi_i \cdot \nabla\phi_j) \left[ \nabla\phi_j \times (\nabla \times \mathbf{v}_0) \right].
\]
Use now $A \cdot (B \times C) = (A \times B) \cdot C$ we obtain
\[
\frac{X_i X_j}{D_t \phi_j} \nabla\phi_j \left[ (\nabla\phi_i \times \nabla\phi_j) \cdot (\nabla \times \mathbf{v}_0)) \right]
-\frac{X_i X_j}{D_t \phi_j} (\nabla\phi_i \cdot \nabla\phi_j) \left[ \nabla\phi_j \times (\nabla \times \mathbf{v}_0) \right].
\]
By symmetry, the $\mathcal{O}(\omega^0)$ term for $\mathbf{v}_j \times (\nabla \times \mathbf{v}_i)$ is

\[
\frac{X_i X_j}{D_t \phi_i} \nabla\phi_i \left[ (\nabla\phi_j \times \nabla\phi_i) \cdot (\nabla \times \mathbf{v}_0) \right] -\frac{X_i X_j}{D_t \phi_i} (\nabla\phi_i \cdot \nabla\phi_j) \left[ \nabla\phi_i \times (\nabla \times \mathbf{v}_0) \right].
\]
Summing these we obtain
\begin{align*}
    &-X_iX_j \left(\frac{\nabla\phi_i}{D_t\phi_i} - \frac{\nabla\phi_j}{D_t\phi_j}\right) \left[ (\nabla\phi_i \times \nabla\phi_j) \cdot (\nabla \times \mathbf{v}_0) \right]\\
    &-X_iX_j (\nabla\phi_i \cdot \nabla\phi_j) \left[ \left(\frac{\nabla\phi_i}{D_t\phi_i} + \frac{\nabla\phi_j}{D_t\phi_j}\right) \times (\nabla \times \mathbf{v}_0) \right].
\end{align*}

Since $\Lambda_{ij}^\tau$ corresponds to $\tau = 0$, in order for two acoustic waves to generate an advective wave, we need $\tau_i + \tau_j = 0$. This implies $|\xi_i| = |\xi_j|$, namely $D_t\phi_i = -D_t\phi_j$. The WKB analysis concludes that (ignoring non-vanishing half-densities and Maslov contributions)
\[
\sigma_\mu(\mathbf{v}_i \times (\nabla \times \mathbf{v}_j) + \mathbf{v}_j \times (\nabla \times \mathbf{v}_i)) = C \xi + \frac{1}{c_0}(\sigma_{\mu_i+1/2}(\mathbf{v}_i) \cdot \sigma_{\mu_j+1/2}(\mathbf{v}_j))\left(\frac{\xi_i}{|\xi_i|}-\frac{\xi_j}{|\xi_j|}\right) \times (\nabla \times \mathbf{v}_0).
\]

\subsection{Propagation of singularity on $\Lambda_{ij}^\tau$}
Now we patch everything together. The WKB analysis as well as the symbol computation for momentum equation tells us that, on the intersection $\Lambda_{ij} \cap \Lambda_{ij}^\tau$:
\begin{align*}
    \Pi_\xi\sigma^1(\mathbf{v}_{ij}) &= \Pi_\xi(\mathrm{i}\tau \sigma_{\mu-1}(\mathbf{v}_{ij}))\\
    &= \Pi_\xi\sigma_\mu(\mathbf{v}_i \times (\nabla \times \mathbf{v}_j) + \mathbf{v}_j \times (\nabla \times \mathbf{v}_i))\\
    &= \frac{1}{c_0}(\sigma_{\mu_i+1/2}(\mathbf{v}_i) \cdot \sigma_{\mu_j+1/2}(\mathbf{v}_j))\Pi_\xi\left[\left(\frac{\xi_i}{|\xi_i|}-\frac{\xi_j}{|\xi_j|}\right) \times (\nabla \times \mathbf{v}_0)\right],
\end{align*}
where $\xi$ is proportional to $\frac{\xi_i}{|\xi_i|}+\frac{\xi_j}{|\xi_j|}$. In fact, we can see from the RHS that the principal symbol on the intersection is homogeneous of degree $\mu = \mu_i+\mu_j+1$, namely $\sigma^1(\mathbf{v}_{ij}) = \sigma^1_\mu(\mathbf{v}_{ij})$.

This expression is zero if and only if $\left(\frac{\xi_i}{|\xi_i|}-\frac{\xi_j}{|\xi_j|}\right) \times (\nabla \times \mathbf{v}_0) = k\xi$ for some real $k$. In particular, if we take dot product with $\nabla \times \mathbf{v}_0$ both sides, we see a necessary condition for it to be 0 is
\[
0 = k\xi \times (\nabla \times \mathbf{v}_0).
\]
In other words, if we assume the background flow $(\rho_0, \mathbf{v}_0)$ has non-zero curl at the intersection region, then a generic choice of directions $\xi_i/|\xi_i|, \xi_j/|\xi_j|$ would satisfy $\Pi_\xi\sigma^1(\mathbf{v}_{ij})(0, \frac{\xi_i}{|\xi_i| } + \frac{\xi_j}{|\xi_j|}) \neq 0$. Note that $\sigma_{\mu_k}(\mathbf{v}_k)$ is proportional to $\xi_k$, so technically speaking we also need $\xi_i \cdot \xi_j \neq 0$, but as the two waves are arbitrarily close to each other (in other words, $\xi_i$ and $-\xi_j$ are close to each other and never orthogonal),
the dot product will always be nonzero in our case.

Near the projection of $\Lambda_{ij}^\tau \backslash \Lambda_{ij}$ (namely in a neighborhood of $\ell$ away from collision point $p$), the source terms $h_{ij}, \mathbf{f}_{ij}$ are smooth, so the continuity and momentum equations reduce to free transport equations there. Specifically, the principal symbol of the operator on $\Lambda_{ij}^\tau \backslash \Lambda_{ij}$ is
\[
\begin{pmatrix}
    0 & \rho_0\xi^T \\ \frac{c_0^2}{\rho_0}\xi & 0
\end{pmatrix}.
\]
By \cite[Proposition 2.7]{dencker1982polarization}, the polarization set of $(\rho_{ij}, \mathbf{v}_{ij})$ is in the kernel of the matrix, meaning $\rho_{ij}$ is smooth there and $\sigma^1(\mathbf{v}_{ij})$ is always orthogonal to $\xi$. From the momentum equation and Section 4 of \cite{dencker1982polarization}, one obtains the transport equation for $\Pi_\xi\sigma^1(\mathbf{v}_{ij}) = \sigma^1(\mathbf{v}_{ij})$:
\[
H_\tau \sigma^1(\mathbf{v}_{ij}) + \left((\nabla \mathbf{v}_0)^T -\frac{1}{2}\nabla \cdot \mathbf{v}_0\right) \sigma^1(\mathbf{v}_{ij}) = 0 \quad \text{on} \quad \Lambda_{ij}^\tau \backslash \Lambda_{ij}.
\]
As a result, for a generic collision the principal symbol of $\mathbf{v}_{ij}$ is nonzero along the advective flow, when the background flow has non-zero curl at $q$.

The order of the principal symbol, as computed above, is $\mu_i + \mu_j + 1$, meaning the corresponding distribution on $\Lambda_{ij}^\tau \backslash \Lambda_{ij}$ has order $\mu_i+\mu_j+1+1/2-4/4 = \mu_i+\mu_j + 1/2$. This is indeed one degree more regular then the previous estimate that $(\rho_{ij}, \mathbf{v}_{ij}) \in \mathcal{I}^{\mu_i+\mu_j+3/2}(\Lambda_{ij}^\tau \backslash \Lambda_{ij})$ at the end of Section \ref{sec: two waves}. We summarize the result into the following lemma.

\begin{lemma}\label{lem: principal symbol of v_ij}
    Suppose $\nabla \times \mathbf{v}_0 \neq 0$ at $q$. Then for two distinct incident directions $\zeta_i = (\tau_i, \xi_i)\in L^*_qM$ and $\zeta_j = (\tau_j, \xi_j)\in L^*_qM$ such that $|\xi_i| = |\xi_j|$, $\tau_i = -c_0|\xi_i|$ and $\tau_j = c_0|\xi_j|$, the principal symbol (ignoring non-vanishing half-densities and Maslov bundle) of $\mathbf{v}_{ij}$ at $\zeta_{ij} = \zeta_i + \zeta_j\in \Lambda_{ij}^\tau$ is:
    \[
    \sigma^1(\mathbf{v}_{ij})(q, \zeta_{ij}) = \frac{1}{c_0}(\sigma(\mathbf{v}_i)(q, \zeta_i) \cdot \sigma(\mathbf{v}_j)(q, \zeta_j))\Pi_\xi\left[\left(\frac{\xi_i}{|\xi_i|}-\frac{\xi_j}{|\xi_j|}\right) \times (\nabla \times \mathbf{v}_0)\right].
    \]
    In particular, it is nonzero for a generic choice of $\zeta_i$ and $\zeta_j$, and stays nonzero along the Hamiltonian flow $H_\tau$ from $(q, \zeta_{ij})$. If $(\rho_k, \mathbf{v}_k) \in \mathcal{I}^{\mu_k}(\Lambda_k)$, then on $\Lambda_{ij}^\tau \backslash \Lambda_{ij}$, $(\rho_{ij}, \mathbf{v}_{ij}) \in \mathcal{I}^{\mu_i+\mu_j+1/2}(\Lambda_{ij}^\tau \backslash \Lambda_{ij})$.
\end{lemma}


\subsection{Curl-free background flow}
In this subsection, we show that propagation of singularity does not occur along advective flow $\ell$ if $\nabla \times \mathbf{v}_0$ vanishes in a neighborhood of $\ell$. In fact, when $\mathbf{f} = 0$, divide by $\rho_0$ both sides and take curl of \eqref{eulereq1}, one obtains
\[
D_t\Omega_0 + (\nabla \cdot \mathbf{v}_0)\Omega_0 - (\Omega_0 \cdot \nabla)\mathbf{v}_0 = 0
\]
after simplification, where $\Omega_0 = \nabla \times \mathbf{v}_0$. As a result, if $\nabla \times \mathbf{v}_0 = 0$ at some point $q$, then it stays zero along the advective flow of $q$. Vanishing in a neighborhood of $\ell$ is then equivalent to vanishing near $q$. This is the \textit{Kelvin's Circulation Theorem} (cf., for example, \cite[pp. 475]{taylor2023partial}).

Consider the linearized operator with source $\mathbf{f}$:
\[
\partial_t \mathbf{v} + \mathbf{v}_0 \cdot \nabla \mathbf{v} + \mathbf{v} \cdot \nabla \mathbf{v}_0 + \nabla\left(\frac{c_0^2}{\rho_0}\rho\right) = \mathbf{f}/\rho_0.
\]
Denote $\Omega = \nabla \times \mathbf{v}$, and apply the curl operator to both sides, the LHS becomes
\begin{align*}
    \nabla \times \mathrm{LHS} &= \partial_t \Omega + \nabla \times (\mathbf{v}_0 \cdot \nabla \mathbf{v} + \mathbf{v} \cdot \nabla \mathbf{v}_0)\\
    &= \partial_t\Omega - \nabla \times (\mathbf{v}_0 \times \Omega) - \nabla \times (\mathbf{v} \times \Omega_0)\\
    &= \partial_t\Omega - \Omega \cdot \nabla \mathbf{v}_0 + \mathbf{v}_0 \cdot \nabla \Omega -\mathbf{v}_0 \nabla \cdot \Omega +\Omega \nabla \cdot \mathbf{v}_0  - \nabla \times (\mathbf{v} \times \Omega_0)\\
    &= D_t\Omega - \Omega \cdot \nabla \mathbf{v}_0 + \Omega \nabla \cdot \mathbf{v}_0,
\end{align*}
where we used the fact that $A \cdot \nabla B + B \cdot \nabla A = \nabla(A\cdot B)-A \times (\nabla \times B) - B \times (\nabla \times A)$, and $\Omega_0 = 0$ in a neighborhood of $\ell$. Thus we obtain a transport equation for $\Omega$ along the background advective flows near $\ell$:
\[
D_t \Omega - \Omega\cdot \nabla \mathbf{v}_0 + \Omega \nabla \cdot \mathbf{v}_0 = \nabla \times (\mathbf{f}/\rho_0).
\]

Suppose now that $\ell$ lies completely outside of the observation region $V$. In particular, for first order linearization solution $\mathbf{v}_j$, the source term is $\mathbf{f}_j/\rho_0$ which is 0 near $\ell$. Thus $\Omega_j = \nabla \times \mathbf{v}_j$ solves a free transport equation along advective flows near $\ell$ with zero initial data for $t < 0$, implying that $\Omega_j$ vanishes near $\ell$. Now we consider second order terms. From previous computations, the source term for the momentum equation has the form
\[
-\nabla\left(\mathbf{v}_i \cdot \mathbf{v}_j + (\gamma - 2)\frac{c_0^2}{\rho_0}\rho_i \rho_j\right) + \mathbf{v}_i \times (\nabla \times \mathbf{v}_j) + \mathbf{v}_j \times (\nabla \times \mathbf{v}_i).
\]
Take curl, this simplifies to
\[
\nabla \times (\mathbf{v}_i \times \Omega_j) + \nabla \times (\mathbf{v}_j \times \Omega_i).
\]
Since $\Omega_i, \Omega_j$ vanishes near $\ell$, this vanishes as well. Then $\Omega_{ij} = \nabla \times \mathbf{v}_{ij}$ satisfies again a free transport equation with zero initial data in the past, hence $\Omega_{ij} = 0$ near $\ell$. In particular, use Poincar\'e lemma, $\nabla \times \mathbf{v}_{ij} = 0$ implies $\mathbf{v}_{ij}$ is a total derivative term near $p$, hence the full symbol at any $\xi \in \Lambda_{ij}\cap \Lambda_{ij}^\tau$ is proportional to $\xi$ itself, which vanishes under the orthogonal projection $\Pi_\xi$. This concludes that there is no propagation of singularity along advective flow for the second order terms $(\rho_{ij}, \mathbf{v}_{ij})$.

\section{Interaction of an advection flow and two acoustic waves}\label{cahinterafaw}
Now take
\[
f=\epsilon_1f_1+\epsilon_2f_2+\epsilon_3f_3+\epsilon_4f_4.
\]
Take future pointing lightlike $\eta_1,\eta_2\in L^{*,+}V$ such that $\gamma_{\eta_1}$ and $\gamma_{\eta_2}$ intersect at point $q=\gamma_{\eta_1}(s_1)=\gamma_{\eta_2}(s_2)$ in $I_g(p^-,p^+)$. Construct sources $f_1$ and $f_2$ that generate distorted plane waves $(\rho_1,\mathbf{v}_1)\in\mathcal{I}(\Lambda(\zeta_1,s_0))$, $(\rho_2,\mathbf{v}_2)\in\mathcal{I}(\Lambda(\zeta_2,s_0))$. Assume that they generate advective flow $\mathbf{v}_{12}=\frac{\partial^2}{\partial\epsilon_1\partial\epsilon_2}\mathbf{v}\vert_{\epsilon_1=\epsilon_2=0}$ with conormal singularity in the direction $\zeta_{12}\in\mathrm{span}\{\zeta_1,\zeta_2\}\cap\{\tau=0\}$, where $\zeta_i=\dot{\gamma}_{\eta_i}(s_i)^\flat$. Detailed analysis has been done in the previous section. Take $\eta_3,\eta_4\in L^{*,+}V$ such that $\gamma_{\eta_3}$ and $\gamma_{\eta_4}$ intersect at point $q_1=\gamma_{\eta_3}(s_3)=\gamma_{\eta_4}(s_4)$ in $I_g(p^-,p^+)$, where $q_1$ is very close to $q$. Construct distorted plane waves $(\rho_3,\mathbf{v}_3)\in\mathcal{I}^\mu(\Lambda(\zeta_3,s_0)),\,(\rho_4,\mathbf{v}_4)\in\mathcal{I}^\mu((\zeta_4,s_0))$. Assume $q_1\in\gamma_{\zeta_{12}}(\mathbb{R}^+)$. Then $\gamma_{\eta_i}\cap\gamma_{\eta_j}=\emptyset$ for $i=1,2,\,j=3,4$ under the no cut points assumption. We will analyze the waves generated by the interaction of $\mathbf{v}_{12}$ and $(\rho_3,\mathbf{v}_3)$, $(\rho_4,\mathbf{v}_4)$. See Figure \ref{fig: determine advective flow}.\\

Denote
\[
(\rho_{1234},\mathbf{v}_{1234})=\frac{\partial^4}{\partial\epsilon_1\partial\epsilon_2\partial\epsilon_3\partial\epsilon_4}(\rho,\mathbf{v})\Bigg\vert_{\epsilon_1=\epsilon_2=\epsilon_3=\epsilon_4=0}.
\]
Then $(\rho_{1234},\mathbf{v}_{1234})$ satisfies the following equations.
\begin{align}\label{h1234}
&\partial_t\rho_{1234}+\rho_0\nabla\cdot\mathbf{v}_{1234}+\mathbf{v}_{1234}\cdot\nabla\rho_0+\rho_{1234}\nabla\cdot\mathbf{v}_0+\mathbf{v}_0\cdot\nabla\rho_{1234}\nonumber\\
=&-\nabla\cdot\left(\rho_{\sigma(3)}\mathbf{v}_{12\sigma(4)}+\rho_{12\sigma(3)}\mathbf{v}_{\sigma(4)}\right)+\mathrm{l.o.t.}\\
=&:h_{1234}\nonumber,
\end{align}
and
\begin{align}\label{f1234}
&\rho_0\frac{\partial\mathbf{v}_{1234}}{\partial t}+\rho_0\mathbf{v}_0\cdot\nabla\mathbf{v}_{1234}+\rho_0\mathbf{v}_{1234}\cdot\nabla \mathbf{v}_0+A\gamma\rho_0^{\gamma-1}\nabla\rho_{1234}+A\gamma(\gamma-2)\rho_0^{\gamma-2}\nabla\rho_0\rho_{1234}\nonumber\\
=&-\rho_0\mathbf{v}_{12\sigma(3)}\cdot\nabla\mathbf{v}_{\sigma(4)}-\frac{1}{2}\rho_0\mathbf{v}_{\sigma(3)\sigma(4)}\cdot\nabla\mathbf{v}_{12}-\rho_0\mathbf{v}_{\sigma(3)}\cdot\nabla\mathbf{v}_{12\sigma(4)}-\frac{1}{2}\rho_0\mathbf{v}_{12}\cdot\nabla\mathbf{v}_{\sigma(3)\sigma(4)}\nonumber\\
&-A\gamma(\gamma-2)\rho_0^{\gamma-2}(\nabla\rho_{\sigma(3)}\rho_{12\sigma(4)}+\rho_{\sigma(3)}\nabla \rho_{12\sigma(4)})+\mathrm{l.o.t.}\\
=&:\mathbf{f}_{1234}\nonumber,
\end{align}
where $\sigma$ is a permutation of $3$ and $4$. We choose $s_0>0$ sufficiently small such that
\[
K_i\cap K_j=\emptyset,\,i=1,2,\, j=3,4.
\]
Then $K_i\cap K_{34}^\tau=\emptyset$, $i=1,2$ since $q<q_1$. 
Here
\[
\begin{split}
&\frac{\partial\rho_{12j}}{\partial t}+\rho_0\nabla\cdot\mathbf{v}_{12j}+\mathbf{v}_{12j}\cdot\nabla\rho_0+\rho_{12j}\nabla\cdot\mathbf{v}_0+\mathbf{v}_0\cdot\nabla\rho_{12j}\\
=&-\rho_{12}\nabla\cdot\mathbf{v}_j-\mathbf{v}_j\cdot\nabla\rho_{12}-\rho_j\nabla\cdot\mathbf{v}_{12}-\mathbf{v}_{12}\cdot\nabla\rho_j,
\end{split}
\]
\[
\begin{split}
&\rho_0\frac{\partial\mathbf{v}_{12j}}{\partial t}+\rho_0\mathbf{v}_0\cdot\nabla\mathbf{v}_{12j}+\rho_0\mathbf{v}_{12j}\cdot\nabla\mathbf{v}_0+A\gamma\rho_0^{\gamma-1}\nabla\rho_{12j}+A\gamma(\gamma-2)\nabla\rho_0\rho_{12j}\\
=&-\rho_0\mathbf{v}_{12}\cdot\nabla\mathbf{v}_j-\rho_0\mathbf{v}_j\cdot\nabla\mathbf{v}_{12}-A\gamma(\gamma-2)\nabla(\rho_0^{\gamma-2}\rho_{12}\rho_j),
\end{split}
\]
for $j=3,4$, and
\[
\begin{split}
&\frac{\partial\rho_{34}}{\partial t}+\rho_0\nabla\cdot\mathbf{v}_{34}+\mathbf{v}_{34}\cdot\nabla\rho_0+\rho_{34}\nabla\cdot\mathbf{v}_0+\mathbf{v}_0\cdot\nabla\rho_{34}\\
=&-\rho_{3}\nabla\cdot\mathbf{v}_4-\mathbf{v}_4\cdot\nabla\rho_{3}-\rho_4\nabla\cdot\mathbf{v}_{3}-\mathbf{v}_{3}\cdot\nabla\rho_4,
\end{split}
\]
\[
\begin{split}
&\rho_0\frac{\partial\mathbf{v}_{34}}{\partial t}+\rho_0\mathbf{v}_0\cdot\nabla\mathbf{v}_{34}+\rho_0\mathbf{v}_{34}\cdot\nabla\mathbf{v}_0+A\gamma\rho_0^{\gamma-1}\nabla\rho_{34}+A\gamma(\gamma-2)\nabla\rho_0\rho_{34}\\
=&-\rho_0\mathbf{v}_{3}\cdot\nabla\mathbf{v}_4-\rho_0\mathbf{v}_4\cdot\nabla\mathbf{v}_3-A\gamma(\gamma-2)\nabla(\rho_0^{\gamma-2}\rho_3\rho_4).
\end{split}
\]
Away from $\Lambda_{12}\cup \Lambda_1\cup\Lambda_2$
\[
\mathbf{v}_{12}\in\mathcal{I}(K_{12}^\tau),
\]
away from $\Lambda_{12}^\tau\cup\Lambda_1\cup\Lambda_2\cup \Lambda_{12}\cup \Lambda_j$
\[
(\rho_{12j},\mathbf{v}_{12j})\in \mathcal{I}(K_{12}^\tau\cap K_j),
\]
away from $\Lambda_3\cup \Lambda_4\cup \Lambda_{34}^\tau$
\[
(\rho_{34},\mathbf{v}_{34})\in \mathcal{I}(K_3\cap K_4).
\]
Finally, away from $\bigcup_{j=1}^4\Lambda_j\cup\Lambda_{12}\cup\Lambda_{12}^\tau\cup\Lambda_{34}\cup\Lambda_{34}^\tau\cup N^*(K_{12}^\tau\cap K_3)\cup N^*(K_{12}^\tau\cap K_4)$,
\[
(h_{1234},\mathbf{f}_{1234})\in \mathcal{I}(K_{12}^\tau\cap K_3\cap K_4).
\]
All the other terms in $h_{1234}$ and $\mathbf{f}_{1234}$ not explicitly listed in \eqref{h1234} and \eqref{f1234} are lower order ones, if we consider $(h_{1234},\mathbf{f}_{1234})$ as a conormal distribution as above. Also we have used the fact that $\sigma(\rho_{12})=0$.
~\\

We will analyze the wavefront set of $(\rho_{1234},\mathbf{v}_{1234})$. For the ease of computation we assume $\rho_0=c_0=1$ at the point of interest.  At certain local coordinates at $q$, we can assume $\zeta_1,\zeta_2$ are both small perturbations of $(-1,1,0,0)$, and $\gamma_{\zeta_0}(\mathbb{R}^+)\cap V\neq\emptyset$ with $\zeta_0=\theta_0=(-1,\pm\sqrt{1-r_0^2},r_0,0)$. Here $r_0\in [-1,1]$ and $\pm\sqrt{1-r_0^2}\neq 1$. Then
\[
\zeta_{12}=c(\zeta_1-\zeta_2)=c(0,\xi_1-\xi_2)
\]
with some constant $c$,
since for advective singularity $\tau_{12}$ must be zero. We can take
\[
\zeta_1 = (-1, \sqrt{1-2\varepsilon^2}, \varepsilon, \varepsilon), \quad \zeta_2 = (-1, \sqrt{1-2\varepsilon^2}, -\varepsilon, \varepsilon).
\]
Then
\[
\zeta_{12}=2c\varepsilon(0,0,1,0).
\]
Then the advective flow $\mathbf{v}_{12}$ has conormal singularity in the direction $\mathbf{e}_2$ at point $q$.
Still denote the direction of the conoromal singularity of $\mathbf{v}_{12}$ at $q_1$ to be $\zeta_{12}$. If $q_1$ and $q$ are close enough, 
 we can choose proper local coordinates at $q_1$ such that 
and $\gamma_{\theta_0}(\mathbb{R}^+)\cap V\neq \emptyset$, $\gamma_{\theta_3}(\mathbb{R}^-)\cap V\neq\emptyset$, $\gamma_{\theta_4}(\mathbb{R}^-)\cap V\neq\emptyset$, $\zeta_{12}\in\mathrm{span}\{\theta_{12}\}$ with $\theta_{12},\theta_0,\theta_0,\theta_3,\theta_4\in T^*_{q_1}$ expressed as
\[
\theta_{12}=(0,0,1,0),\quad \theta_0=(-1,\pm\sqrt{1-r_0^2},r_0,0),
\]
\[
\theta_3=(-1,1,0,0),\quad \theta_{4}=(-1,\sqrt{1-\varsigma^2},\varsigma,0),
\]
where $\varsigma$ is sufficiently small.\\
\begin{figure}[htbp]
\centering
\includegraphics[width=0.4\textwidth]{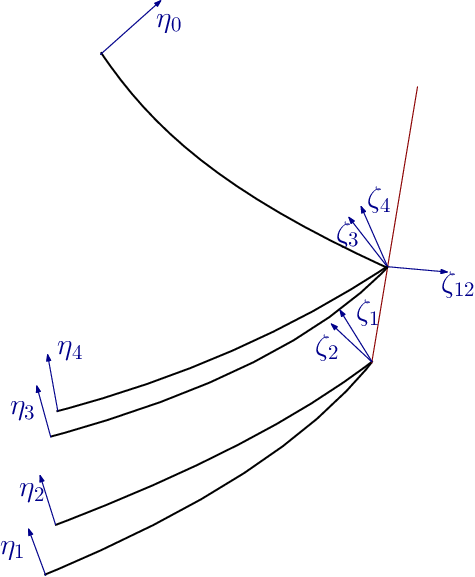}
\caption{Two primary acoustic waves, propagating along null-geodesics $\gamma_{\eta_1}$ and $\gamma_{\eta_2}$, intersect at a point $q$, generating an advective flow with a conormal singularity in the direction $\zeta_{12}$. This advective flow then propagates to a subsequent point $q_1$ (we also refer to the direction there as $\zeta_{12}$), where it interacts with two additional acoustic waves propagating along $\gamma_{\eta_3}$ and $\gamma_{\eta_4}$ to generate an outgoing acoustic wave along $\gamma_{\eta_0}$.}
\label{fig: determine advective flow}
\end{figure}

Notice that $\theta_0$ can be written as the unique linear combination of the three vectors $\theta_{12},\theta_3,\theta_4$. Actually we can write
\[
\theta_0=\left(r_0-\frac{\varsigma(1\mp\sqrt{1-r_0^2})}{1-\sqrt{1-\varsigma^2}}\right)\theta_{12}+\frac{\pm\sqrt{1-r_0^2}-\sqrt{1-\varsigma^2}}{1-\sqrt{1-\varsigma^2}}\theta_3+\frac{1\mp\sqrt{1-r_0^2}}{1-\sqrt{1-\varsigma^2}}\theta_4.
\]
We will calculate the principal symbol of $(h_{1234},\mathbf{f}_{1234})$. Without loss of generality, we write $\zeta=\theta_0$ and calculate $(h_{1234},\mathbf{f}_{1234})(\zeta)$.
Denote
\[
\zeta_{12}=\left(r_0-\frac{\varsigma(1\mp\sqrt{1-r_0^2})}{1-\sqrt{1-\varsigma^2}}\right)\theta_{12},
\]
\[
\zeta_3=\frac{\pm\sqrt{1-r_0^2}-\sqrt{1-\varsigma^2}}{1-\sqrt{1-\varsigma^2}}\theta_3,\quad \zeta_4=\frac{1\mp\sqrt{1-r_0^2}}{1-\sqrt{1-\varsigma^2}}\theta_4.
\]
We also write
\[
\zeta_j=(\tau_j,\xi_j)
\]
where $j=(12),3$ or $4$.

Denote as in Section \ref{nonlinearinteraction}
\[
b(r_0)=1\mp\sqrt{1-r_0^2}.
\]
We will use the following asymptotic behaviors
\[
\tau_{12}=0,\quad\tau_3=2b(r_0)\varsigma^{-2}+\mathcal{O}(1),\quad\quad\tau_4=-2b(r_0)\varsigma^{-2}+\mathcal{O}(1),
\]
and
\[
\xi_{12}=(-2b(r_0)\varsigma^{-1}+\mathcal{O}(1))\mathbf{e}_2,
\]
\[
\xi_3=-2b(r_0)\varsigma^{-2}\mathbf{e}_1+\mathcal{O}(1),\quad\quad \xi_4=2b(r_0)\varsigma^{-2}\mathbf{e}_1+2b(r_0)\varsigma^{-1}\mathbf{e}_2+\mathcal{O}(1).
\]

Let use first calculate the principal symbol of $(\rho_{34},\mathbf{v}_{34})$. Note that
\[
\tau_3+\tau_4=-1,
\]
\[
\xi_3+\xi_4=\left(\pm\sqrt{1-r_0^2},\frac{1\mp\sqrt{1-r_0^2}}{1-\sqrt{1-\varsigma^2}}\varsigma,0\right)=2b(r_0)\varsigma^{-1}\mathbf{e}_2\pm\sqrt{1-r_0^2}\mathbf{e}_1+\mathcal{O}(\varsigma).
\]
Take
\[
\sigma(\rho_3)(\zeta_3)=-\varsigma^2\tau_3a_3=-2b(r_0)\alpha_3+\mathcal{O}(\varsigma),\quad \sigma(\mathbf{v}_3)(\zeta_3)=\varsigma^2\xi_3a_3=-2b(r_0)a_3\mathbf{e}_1+\mathcal{O}(\varsigma),
\]
\[
\sigma(\rho_4)(\zeta_4)=-\varsigma^2\tau_4a_4=2b(r_0)a_4+\mathcal{O}(\varsigma),\quad \sigma(\mathbf{v}_4)(\zeta_4)=\varsigma^2\xi_4a_4=2b(r_0)\mathbf{e}_1+\mathcal{O}(\varsigma).
\]
By calculation, we obtain
\begin{equation}
\begin{split}
&((\tau_3+\tau_4)^2-|\xi_3+\xi_4|^2)\sigma(\rho_{34})(\zeta_{34})\\
=&a_3a_4\varsigma^{4}\left(\tau_3\tau_4+\xi_3\cdot\xi_4\right)+a_3a_4\varsigma^{4}(\xi_3+\xi_4)\cdot(\xi_3+\xi_4)(\xi_3\cdot\xi_4+(\gamma-2)\tau_3\tau_4)\\
=&-16(\gamma-1)b(r_0)^4a_3a_4\varsigma^{-2}+\mathcal{O}(\varsigma^{-1}).
\end{split}
\end{equation}
Note that
\[
(\tau_3+\tau_4)^2-|\xi_3+\xi_4|^2=-4b(r_0)^2\varsigma^{-2}+\mathcal{O}(1).
\]
Then we have
\[
\sigma(\rho_{34})(\zeta_{3}+\zeta_4)=4(\gamma-1)b(r_0)^2a_3(\zeta_3)a_4(\zeta_4)+\mathcal{O}(\varsigma).
\]
By \eqref{eqrhov2} we obtain
\[
\begin{split}
&\sigma(\mathbf{v}_{34})(\zeta_3+\zeta_4)\\
=&-\sigma(\rho_{34})(\zeta_3+\zeta_4)(\xi_3+\xi_4)-\xi_4\cdot\sigma(\mathbf{v}_3)(\zeta_3)\sigma(\mathbf{v}_4)(\zeta_4)-\xi_3\cdot\sigma(\mathbf{v}_4)(\zeta_4)\sigma(\mathbf{v}_3)(\zeta_3)\\
&-(\gamma-2)(\xi_3+\xi_4)\sigma(\rho_3)(\zeta_3)\sigma(\rho_4)(\zeta_4)\\
=&\mathcal{O}(1).
\end{split}
\]




Next we calculate the principal symbol of $(\rho_{124},\mathbf{v}_{124})$ and $(\rho_{123},\mathbf{v}_{123})$. Note that
\[
\zeta_{124}=\zeta_{12}+\zeta_4=\left(\frac{1\mp\sqrt{1-r_0^2}}{1-\sqrt{1-\varsigma^2}},\frac{1\mp\sqrt{1-r_0^2}}{1-\sqrt{1-\varsigma^2}}\sqrt{1-\varsigma^2},r_0,0\right).
\]
We will use the asymptotic behaviors
\[
\tau_{124}=-2b(r_0)\varsigma^{-2}+\mathcal{O}(1),\quad \xi_{124}=2b(r_0)\varsigma^{-2}\mathbf{e}_1+\mathcal{O}(1),
\]
and
\[
\tau_{124}^2-|\xi_{124}|^2=4b(r_0)^2\varsigma^{-2}+\mathcal{O}(1).
\]

Recall that $\sigma(\rho_{12})=0$, we have
\[
\begin{split}
&(\tau_{124}^2-|\xi_{124}|^2)\sigma(\rho_{124})(\zeta_{124})\\
=&-(\tau_{12}+\tau_4)\left((\sigma(\rho_4)\xi_{12}\cdot\sigma(\mathbf{v}_{12})+(\sigma(\rho_4)\xi_{4}\cdot\sigma(\mathbf{v}_{12})\right)\\
&+(\xi_{12}+\xi_4)\cdot\left(\sigma(\mathbf{v}_{12})\cdot\xi_4\sigma(\mathbf{v}_4)+\sigma(\mathbf{v}_4)\cdot\xi_{12}\sigma(\mathbf{v}_{12})\right)\\
=&\varsigma^{2}\tau_{124}\tau_4a_4\sigma(\mathbf{v}_{12})\cdot\xi_{124}+\varsigma^{2}a_4(\xi_{124}\cdot\xi_4)\sigma(\mathbf{v}_{12})\cdot\xi_4+\varsigma^{2}a_4(\xi_{124}\cdot\sigma(\mathbf{v}_{12}))\xi_4\cdot\xi_{12}\\
=&16b(r_0)^3\sigma(\mathbf{v}_{12})\cdot\mathbf{e}_1a_4\varsigma^{-4}+\mathcal{O}(\varsigma^{-2}),
\end{split}
\]
where we have used the fact $\sigma(\mathbf{v}_{12})\cdot\mathbf{e}_2=0$.
Then
\[
\begin{split}
\sigma(\rho_{124})(\zeta_{124})=4b(r_0)\varsigma^{-2}a_4(\zeta_4)\sigma(\mathbf{v}_{12})(\zeta_{12})\cdot\mathbf{e}_1+\mathcal{O}(1).
\end{split}
\]
We also use the identity
\[
\begin{split}
\tau_{124}\sigma(\mathbf{v}_{124})=&-\xi_{124}\sigma(\rho_{124})-\xi_{12}\cdot\sigma(\mathbf{v}_{4})\sigma(\mathbf{v}_{12})-\xi_4\cdot\sigma(\mathbf{v}_{12})\sigma(\mathbf{v}_4)\\
=&-8b(r_0)^2\varsigma^{-4}a_4(\sigma(\mathbf{v}_{12})\cdot\mathbf{e}_1)\mathbf{e}_1+\mathcal{O}(\varsigma^{-2}).
\end{split}
\]
Then
\[
\sigma(\mathbf{v}_{124})(\zeta_{124})=4b(r_0)\varsigma^{-2}a_4(\zeta_4)(\sigma(\mathbf{v}_{12})(\zeta_{12})\cdot\mathbf{e}_1)\mathbf{e}_1+\mathcal{O}(1).
\]


Now we proceed with
\[
\zeta_{123}=\left(-\frac{\pm\sqrt{1-r_0^2}-\sqrt{1-\varsigma^2}}{1-\sqrt{1-\varsigma^2}},\frac{\pm\sqrt{1-r_0^2}-\sqrt{1-\varsigma^2}}{1-\sqrt{1-\varsigma^2}},r_0-\frac{\varsigma(1\mp\sqrt{1-r_0^2})}{1-\sqrt{1-\varsigma^2}},0\right).
\]
We will use the asymptotic behaviors
\[
\tau_{123}=2b(r_0)\varsigma^{-2}+\mathcal{O}(1),\quad \xi_{123}=-2b(r_0)\varsigma^{-2}\mathbf{e}_1-2b(r_0)\varsigma^{-1}\mathbf{e}_2+\mathcal{O}(1),
\]
and
\[
\tau_{123}^2-|\xi_{123}|^2=-4b(r_0)^2\varsigma^{-2}+4r_0b(r_0)\varsigma^{-1}+\mathcal{O}(1).
\]
Similar as above, we have
\[
\begin{split}
&(\tau_{123}^2-|\xi_{123}|^2)\sigma(\rho_{123})(\zeta_{123})\\
=&\varsigma^{2}\tau_{123}\tau_3a_3\sigma(\mathbf{v}_{12})\cdot\xi_{123}+\varsigma^{2}a_3(\xi_{123}\cdot\xi_3)\sigma(\mathbf{v}_{12})\cdot\xi_3+\varsigma^{2}a_3(\xi_{123}\cdot\sigma(\mathbf{v}_{12}))\xi_3\cdot\xi_{12}\\
=&-16b(r_0)^3\sigma(\mathbf{v}_{12})\cdot\mathbf{e}_1a_3\varsigma^{-4}+\mathcal{O}(\varsigma^{-2}).
\end{split}
\]
So
\[
\sigma(\rho_{123})(\zeta_{123})=4b(r_0)\varsigma^{-2}a_3(\zeta_3)\sigma(\mathbf{v}_{12})(\zeta_{12})\cdot\mathbf{e}_1+4r_0\varsigma^{-1}a_3(\zeta_3)\sigma(\mathbf{v}_{12})(\zeta_{12})\cdot\mathbf{e}_1+\mathcal{O}(1).
\]

We also have
\[
\begin{split}
&\tau_{123}\sigma(\mathbf{v}_{123})\\
=&-\xi_{123}\sigma(\rho_{123})-\xi_{12}\cdot\sigma(\mathbf{v}_{3})\sigma(\mathbf{v}_{12})-\xi_3\cdot\sigma(\mathbf{v}_{12})\sigma(\mathbf{v}_3)\\
=&8b(r_0)^2\varsigma^{-4}a_3(\sigma(\mathbf{v}_{12})\cdot\mathbf{e}_1)\mathbf{e}_1+8r_0b(r_0)\varsigma^{-3}a_3(\sigma(\mathbf{v}_{12})\cdot\mathbf{e}_1)\mathbf{e}_1+8b(r_0)^2\varsigma^{-3}a_3(\sigma(\mathbf{v}_{12})\cdot\mathbf{e}_1)\mathbf{e}_2+\mathcal{O}(\varsigma^{-2}).
\end{split}
\]
Then
\[
\begin{split}
\sigma(\mathbf{v}_{123})(\zeta_{123})=&4b(r_0)\varsigma^{-2}a_3(\zeta_3)(\sigma(\mathbf{v}_{12})(\zeta_{12})\cdot\mathbf{e}_1)\mathbf{e}_1+4r_0\varsigma^{-1}a_3(\zeta_3)(\sigma(\mathbf{v}_{12})(\zeta_{12})\cdot\mathbf{e}_1)\mathbf{e}_1\\
&+4b(r_0)\varsigma^{-1}a_3(\zeta_3)(\sigma(\mathbf{v}_{12})(\zeta_{12})\cdot\mathbf{e}_1)\mathbf{e}_2.
\end{split}
\]

Finally we calculate the principal symbol of $(\rho_{1234},\mathbf{v}_{1234})$.
We compute
\begin{align*}
 &\sigma(\mathbf{f}_{1234})(\zeta)\\
 =&-\sigma(\mathbf{v}_{123})\cdot\xi_4\sigma(\mathbf{v}_4)-\sigma(\mathbf{v}_{124})\cdot\xi_3\sigma(\mathbf{v}_3)-\sigma(\mathbf{v}_{34})\cdot\xi_{12}\sigma(\mathbf{v}_{12})\\
 &-\sigma(\mathbf{v}_{12})\cdot(\xi_3+\xi_4)\sigma(\mathbf{v}_{34})-\sigma(\mathbf{v}_3)\cdot\xi_{124}\sigma(\mathbf{v}_{124})-\sigma(\mathbf{v}_4)\cdot\xi_{123}\sigma(\mathbf{v}_{123})\\
 &-(\gamma-2)\left(\xi_{3}\sigma(\rho_3)\sigma(\rho_{124})+\xi_{4}\sigma(\rho_4)\sigma(\rho_{123})+\sigma(\rho_3)\xi_{124}\sigma(\rho_{124})+\sigma(\rho_4)\xi_{123}\sigma(\rho_{123})\right)\\
 =&-\sigma(\mathbf{v}_{123})\cdot\xi_4\sigma(\mathbf{v}_4)-\sigma(\mathbf{v}_{124})\cdot\xi_3\sigma(\mathbf{v}_3)-\sigma(\mathbf{v}_{34})\cdot\xi_{12}\sigma(\mathbf{v}_{12})\\
 &-\sigma(\mathbf{v}_{12})\cdot\xi_{34}\sigma(\mathbf{v}_{34})-\sigma(\mathbf{v}_3)\cdot\xi_{124}\sigma(\mathbf{v}_{124})-\sigma(\mathbf{v}_4)\cdot\xi_{123}\sigma(\mathbf{v}_{123})\\
 &-(\tau_{123}\sigma(\rho_{123})+\xi_{123}\cdot\sigma(\mathbf{v}_{123}))\sigma(\mathbf{v}_{4})-(\tau_{124}\sigma(\rho_{124})+\xi_{124}\cdot\sigma(\mathbf{v}_{124}))\sigma(\mathbf{v}_{3})\\
 &-(\tau_{34}\sigma(\rho_{34})+\xi_{34}\cdot\sigma(\mathbf{v}_{34}))\sigma(\mathbf{v}_{12})\\
 &-\xi_{123}\cdot(\sigma(\rho_3)\sigma(\mathbf{v}_{12}))\sigma(\mathbf{v}_4)-\xi_{124}\cdot(\sigma(\rho_4)\sigma(\mathbf{v}_{12}))\sigma(\mathbf{v}_3)-\xi_{34}\cdot(\sigma(\rho_3)\sigma(\mathbf{v}_4))\sigma(\mathbf{v}_{12})\\
 &-\left(\tau_{123}\sigma(\mathbf{v}_{123})+\xi_{123}\sigma(\rho_{123})\right)\sigma(\rho_4)-\left(\tau_{124}\sigma(\mathbf{v}_{124})+\xi_{124}\sigma(\rho_{124})\right)\sigma(\rho_3)\\
&-(\xi_3\cdot\sigma(\mathbf{v}_{12})\sigma(\mathbf{v}_3)+\xi_{12}\cdot\sigma(\mathbf{v}_3)\sigma(\mathbf{v}_{12}))\sigma(\rho_4)-(\xi_4\cdot\sigma(\mathbf{v}_{12})\sigma(\mathbf{v}_4)+\xi_{12}\cdot\sigma(\mathbf{v}_4)\sigma(\mathbf{v}_{12}))\sigma(\rho_3)\\
&-\left(\tau_{3}\sigma(\rho_{3})+\xi_{3}\cdot\sigma(\mathbf{v}_{3})\right)\sigma(\mathbf{v}_{124})-\left(\tau_{4}\sigma(\rho_{4})+\xi_{4}\cdot\sigma(\mathbf{v}_{4})\right)\sigma(\mathbf{v}_{123})\\
&-\left(\tau_3\sigma(\mathbf{v}_3)+\xi_3\sigma(\rho_3)\right)\sigma(\rho_{124})-\left(\tau_4\sigma(\mathbf{v}_4)+\xi_4\sigma(\rho_4)\right)\sigma(\rho_{123})\\
&-(\gamma-2)\left(\xi_{3}\sigma(\rho_3)\sigma(\rho_{124})+\xi_4\sigma(\rho_4)\sigma(\rho_{123})+\sigma(\rho_{3})\xi_{124}\sigma(\rho_{124})+\sigma(\rho_{4})\xi_{123}\sigma(\rho_{123})\right).
\end{align*}

Noticing
\[
 \begin{split}
 &-\xi_{123}\cdot(\sigma(\rho_3)\sigma(\mathbf{v}_{12}))\sigma(\mathbf{v}_4)-\xi_{124}\cdot(\sigma(\rho_4)\sigma(\mathbf{v}_{12}))\sigma(\mathbf{v}_3)-\xi_{34}\cdot(\sigma(\rho_3)\sigma(\mathbf{v}_4))\sigma(\mathbf{v}_{12})\\
&-(\xi_3\cdot\sigma(\mathbf{v}_{12})\sigma(\mathbf{v}_3)+\xi_{12}\cdot\sigma(\mathbf{v}_3)\sigma(\mathbf{v}_{12}))\sigma(\rho_4)-(\xi_4\cdot\sigma(\mathbf{v}_{12})\sigma(\mathbf{v}_4)+\xi_{12}\cdot\sigma(\mathbf{v}_4)\sigma(\mathbf{v}_{12}))\sigma(\rho_3)\\
=&\sigma(\rho_3)(\xi_{1234}\cdot\sigma(\mathbf{v}_4)\sigma(\mathbf{v}_{12})+\xi_{1234}\cdot\sigma(\mathbf{v}_{12})\sigma(\mathbf{v}_4))+\sigma(\rho_4)(\xi_{1234}\cdot\sigma(\mathbf{v}_3)\sigma(\mathbf{v}_{12})+\xi_{1234}\cdot\sigma(\mathbf{v}_{12})\sigma(\mathbf{v}_3))\\
=&\mathcal{O}(1)
 \end{split}
\]
and
\[
-\xi_{1234}\cdot\sigma(\mathbf{v}_{34})\sigma(\mathbf{v}_{12})-\tau_{34}\sigma(\rho_{34})=\mathcal{O}(1).
\]
Then we have
\[
 \begin{split}
  &\sigma(\mathbf{f}_{1234})(\zeta)\\
  =&-\xi_{1234}\cdot\sigma(\mathbf{v}_{123})\sigma(\mathbf{v}_4)-\xi_{1234}\cdot\sigma(\mathbf{v}_4)\sigma(\mathbf{v}_{123})-\tau_{1234}\sigma(\rho_4)\sigma(\mathbf{v}_{123})\\
  &-\xi_{1234}\cdot\sigma(\mathbf{v}_{124})\sigma(\mathbf{v}_3)-\xi_{1234}\cdot\sigma(\mathbf{v}_3)\sigma(\mathbf{v}_{124})-\tau_{1234}\sigma(\rho_3)\sigma(\mathbf{v}_{124})\\
  &-\tau_{1234}\sigma(\rho_{123})\sigma(\mathbf{v}_4)-\tau_{1234}\sigma(\rho_{124})\sigma(\mathbf{v}_3)-(\gamma-1)\sigma(\rho_{123})\sigma(\rho_4)\xi_{1234}-(\gamma-1)\sigma(\rho_{124})\sigma(\rho_3)\xi_{1234}\\
  &+\mathcal{O}(1).
 \end{split}
\]
Also we have
 \[
  \begin{split}
   &\sigma(h_{1234})(\zeta)\\
   =&-\xi_{1234}\cdot\left(\sigma(\rho_3)\sigma(\mathbf{v}_{124})+\sigma(\rho_4)\sigma(\mathbf{v}_{123})+\sigma(\rho_{123})\sigma(\mathbf{v}_4)+\sigma(\rho_{124})\sigma(\mathbf{v}_3)+\sigma(\rho_{34})\sigma(\mathbf{v}_{12})\right).
  \end{split}
 \]
Summing up, we end up with
 \[
 \begin{split}
 & \xi\cdot\sigma(\mathbf{f}_{1234})(\zeta)- \tau\sigma(h_{1234})(\zeta)\\
=& -2\xi\cdot\sigma(\mathbf{v}_{123})\xi\cdot\sigma(\mathbf{v}_{4})-2\xi\cdot\sigma(\mathbf{v}_{124})\xi\cdot\sigma(\mathbf{v}_{3})-(\gamma-1)(\sigma(\rho_{123})\sigma(\rho_4)+\sigma(\rho_{124})\sigma(\rho_3))+\mathcal{O}(1)\\
=&-\left( 32r_0b(r_0)^2(1-b(r_0))+16r_0(1-r_0^2)b(r_0)\right)\varsigma^{-1}a_3a_4\sigma(\mathbf{v}_{12})\cdot\mathbf{e}_1\\
&+8(\gamma-1)r_0b(r_0)\varsigma^{-1}a_3a_4\sigma(\mathbf{v}_{12})\cdot\mathbf{e}_1+\mathcal{O}(1).
\end{split}
 \]
Notice that when $r_0=1$, $b(r_0)=1$,
\[
\xi\cdot\sigma(f_{1234})(\zeta)- \tau\sigma(h_{1234})(\zeta)=\left(8(\gamma-1)\varsigma^{-1}+\mathcal{O}(1)\right)a_3(\zeta_3)a_4(\zeta_4)\sigma(\mathbf{v}_{12})(\zeta_{12})\cdot\mathbf{e}_1.
\]
To show it is nonzero for sufficiently small $\varsigma$, we argue that $\sigma(\mathbf{v}_{12})(\zeta_{12}) \cdot \mathbf{e}_1 \neq 0$.

We note that by Lemma \ref{lem: principal symbol of v_ij},
\[
\sigma(\mathbf{v}_{12})(\zeta_{12}) \cdot \mathbf{e}_1 \propto \Pi_{\xi_1/|\xi_1| - \xi_2/|\xi_2|}\left[\left(\frac{\xi_1}{|\xi_1|} + \frac{\xi_2}{|\xi_2|}\right) \times (\nabla \times \mathbf{v}_0)\right] \cdot \mathbf{e}_1,
\]
where $A \propto B$ means $A = \kappa B$ for some $\kappa \neq 0$, and
\[
\frac{\xi_1}{|\xi_1|} - \frac{\xi_2}{|\xi_2|} = 2\varepsilon\mathbf{e}_2, \quad \frac{\xi_1}{|\xi_1|} + \frac{\xi_2}{|\xi_2|} = 2\sqrt{1-2\varepsilon^2}\mathbf{e}_1 + 2\varepsilon\mathbf{e}_3.
\]
Combine the orthogonal projection and dot product, we have $\Pi_\xi[\mathbf{w}] \cdot \mathbf{e}_1 = \mathbf{w} \cdot \mathbf{e}_1$ for all vector $\mathbf{w}$. The term inside the bracket is orthogonal to $\sqrt{1-2\varepsilon^2}\mathbf{e}_1 + \varepsilon\mathbf{e}_3$, meaning it is spanned by $\mathbf{e}_2$ and $\varepsilon\mathbf{e}_1 - \sqrt{1 - 2\varepsilon^2}\mathbf{e}_3$. Note that
\[
\left(\frac{\xi_1}{|\xi_1|} + \frac{\xi_2}{|\xi_2|}\right) \times (\nabla \times \mathbf{v}_0) \propto \mathbf{e}_2 \iff (\nabla \times \mathbf{v}_0) \cdot \mathbf{e}_2 = 0.
\]
By slightly changing the coordinate system (namely choosing another outgoing $\theta_0$ and another incoming vector to be $(1, 1, 0, 0)$), we may assume $(\nabla \times \mathbf{v}_0) \cdot \mathbf{e}_2 \neq 0$. Hence the term before the orthogonal projection has nonzero $\mathbf{e}_1$ component, implying $\sigma(\mathbf{v}_{12})(\zeta_{12}) \cdot \mathbf{e}_1 \neq 0$.

So $\xi\cdot\sigma(\mathbf{f}_{1234})(\zeta)- \tau\sigma(h_{1234})(\zeta)$ is generically nonzero for $\zeta\in\mathrm{span}\{\zeta_{12},\zeta_{3},\zeta_{4}\}\cap L^{*,+}_{q_1}$ at least in a neighborhood of $\zeta_0$. Then the singularities of the acoustic waves $(\rho_{1234},\mathbf{v}_{1234})$ generated by $(h_{1234}, \mathbf{f}_{1234})$ will propagate to $V$. The set $\mathrm{sing\,supp}(\rho_{1234})\cap V$ tends to a set of Hausdorff dimension $2$ as $s_0\rightarrow 0$.

\section{Recovery of background solution}\label{chfinal}

\subsection{Recovery of the advective integral curve}

By Theorem \ref{thm: determination of conformal class}, the conformal structure of the metric has been determined. In other words, all the null-geodesics can be determined up to diffeomorphism and reparameterization. Let $\gamma^{(j)}_{\zeta_i}$ be the null-geodesic in Lorentzian metric $g^{(j)}$. We take
\[
\gamma^{(1)}_{\eta_1}\cap \gamma^{(1)}_{\eta_2}=\{q\},\quad \gamma^{(1)}_{\eta_3}\cap \gamma^{(1)}_{\eta_4}=\{q_1\},
\]
and then
\[
\gamma^{(2)}_{\eta_1}\cap \gamma^{(2)}_{\eta_2}=\{\Phi(q)\},\quad \gamma^{(2)}_{\eta_3}\cap \gamma^{(2)}_{\eta_4}=\{\Phi(q_1)\}.
\]
We can vary the choices of $\eta_i$, $i=1,2,3,4$, without changing the intersection points.
Assume $q< q_1$ and then $\Phi(q)< \Phi(q_1)$. 
If $q_1$ is not on the advective flow through $q$, then $K^\tau_{12}\cap K_3\cap K_4=\emptyset$ for sufficiently small $s_0$, but it is still possible that $ K_{12}^\tau\cap K_j\neq \emptyset$ for $j=3$ or $4$. However for this case the singularity of $(\rho_{1234},\mathbf{v}_{1234})\vert_V$ is contained in $\bigcup_{j=3}^4\Lambda_{\tau\pm c_0|\xi|}(K_{12}^\tau\cap K_j)\cup\bigcup_{i=1}^4\Lambda_i$. Therefore $\mathrm{sing\,supp}(\rho_{1234})$
tends to a set of Hausdorff dimension at most $1$ as $s_0\rightarrow 0$.
If $q_1$ is on the advective flow through $q$ and in addition $q_1$ and $q$ are close enough, we can always choose $\eta_1,\eta_2,\eta_3,\eta_4$ such that $\mathrm{sing\, supp}\sigma(\rho_{1234})\cap V$ tends to a set whose Hausdorff dimension is $2$ by the argument in previous section, see also Figure \ref{fig: determine advective flow}. Assume 
\[
L^{(1)}_V=L^{(2)}_V.
\]This means that we can say that $q_1$ is on the advective integral curve through $q$ if and only if $\Phi(q_1)$ is on the advective integral curve through $\Phi(q)$ if $q_1$ and $q$ are close enough. In this way we can now conclude
\[
\Phi_*(\partial_t+\mathbf{v}_0^{(1)}\cdot\nabla) \propto (\partial_t+\mathbf{v}_0^{(2)}\cdot\nabla)\quad\text{in }I_{g^{(2)}(p^-,p^+)}.
\]
Therefore $\ell$ is an integral curve for $\partial_t+\mathbf{v}_0^{(1)}\cdot \nabla$ if and only if $\Phi(\ell)$ is an integral curve for $\partial_t+\mathbf{v}_0^{(2)}\cdot \nabla$ up to reparameterization.

\subsection{Recovery of $\rho_0$ and $\mathbf{v}_0$}
\begin{proof}[Proof of Theorem \ref{thm:mainthm}.]
Denote
 \[
 \tilde{g}^{(j)} = (c_0^{(j)})^2g^{(j)}=\left( \begin{array}{cc}
 -(c_0^{(j)})^2+(\mathbf{v}_0^{(j)})^T\mathbf{v}^{(j)}_0&-(\mathbf{v}^{(j)}_0)^T\\
 -\mathbf{v}^{(j)}_0 &I
  \end{array}\right)
  \]
 to be the representative of the conformal class of $g^{(j)}$ whose restriction on $\{t = \text{constant}\}$ slice is the Euclidean metric. Summarizing previous results, we have shown that if $L_V^{(1)} = L_V^{(2)}$, then:
\begin{enumerate}
    \item $\tilde{g}^{(1)}|_V = \tilde{g}^{(2)}|_V$;
    \item the metrics $\tilde{g}^{(1)}$ and $\tilde{g}^{(2)}$ are conformally diffeomorphic to each other on causal diamonds, namely there exists a diffeomorphism $\Phi: I_{g^{(1)}}(p^-, p^+) \to I_{g^{(2)}}(p^-, p^+)$ and a conformal factor $\varphi>0$ on $I_{g^{(1)}}(p^-, p^+)$ such that $\Phi|_V = \text{Id}$, $\varphi|_V = 1$, and $\tilde{g}^{(1)} = \varphi\Phi^*\tilde{g}^{(2)}$;
    \item $\ell$ is an integral curve for $\partial_t + \mathbf{v}_0^{(1)}\cdot \nabla$ in $I_{g^{(1)}}(p^-, p^+)$ if and only if $\Phi(\ell)$ is an integral curve for $\partial_t + \mathbf{v}_0^{(2)}\cdot \nabla$ in $I_{g^{(2)}}(p^-, p^+)$ up to reparameterization.
\end{enumerate}
We will show that $I_{g^{(1)}}(p^-, p^+) = I_{g^{(2)}}(p^-, p^+)$, and $(\rho_0^{(1)}, \mathbf{v}_0^{(1)}) = (\rho_0^{(2)}, \mathbf{v}_0^{(2)})$ on the causal diamond.

A straightforward computation shows that
\begin{align*}
    \tilde{g}^{(j)}(\partial_t + \mathbf{v}_0^{(j)}\cdot \nabla, \partial_{x^k}) = -v_0^{k,(j)} + v_0^{k,(j)} = 0,
\end{align*}
where $\mathbf{v}_0^{(j)} = (v_0^{1,(j)}, v_0^{2,(j)}, v_0^{3,(j)})$. Namely the orthogonal complement of $\partial_t + \mathbf{v}_0^{(j)}\cdot \nabla$ with respect to $\tilde{g}^{(j)}$ is $\text{span}\{\partial_{x^1}, \partial_{x^2}, \partial_{x^3}\}$ for both $j = 1, 2$. By (3), at any $(t, x) \in I_{g^{(1)}}(p^-, p^+)$,
\[
\Phi_*(\partial_t + \mathbf{v}_0^{(1)}(t, x)\cdot \nabla) = k(t, x) (\partial_t + \mathbf{v}_0^{(2)}(\Phi(t, x))\cdot \nabla)
\]
for some non-zero scaling $k(t, x)$. By (2), we thus have
\begin{align*}
    \tilde{g}^{(2)}(\partial_t + \mathbf{v}_0^{(2)}(\Phi(t, x))\cdot \nabla, \Phi_*(\partial_{x^k})) &= \frac{1}{k(t, x)}\Phi^*\tilde{g}^{(2)}(\partial_t + \mathbf{v}_0^{(1)}(t, x)\cdot \nabla, \partial_{x^k}) \\
    &= \frac{1}{k(t, x)\varphi(t, x)}\tilde{g}^{(1)}(\partial_t + \mathbf{v}_0^{(1)}(t, x)\cdot \nabla, \partial_{x^k}) \\
    &= 0.
\end{align*}
As a result, $\Phi_*$ preserves $\text{span}\{\partial_{x^1}, \partial_{x^2}, \partial_{x^3}\}$. Since $\Phi$ is smooth, it must map time slice to time slice, namely $\Phi(\{t = c_1\} \cap I_{g^{(1)}}(p^-, p^+)) = \{t = c_2\} \cap I_{g^{(2)}}(p^-, p^+)$. Moreover, as $\partial_{x^1}, \partial_{x^2}, \partial_{x^3}$ are all spacelike with respect to $\tilde{g}^{(1)}$, any causal direction must have non-vanishing $\partial_t$ component, hence causal curves from $p^-$ to $p^+$ must be strictly monotone in $t$. If $\{t = c_1\} \cap I_{g^{(1)}}(p^-, p^+) \neq \emptyset$, then $c_1$ must be in between $t^- = t(p^-)$ and $t^+ = t(p^+)$. As $\mu$ is a continuous timelike curve from $p^-$ to $p^+$, it is fully in $I_{g^{(1)}}(p^-, p^+)$, and hence intersects with $\{t = c_1\} \cap I_{g^{(1)}}(p^-, p^+)$. By (1), $\Phi$ is the identity map on $V$ which includes $\mu$, so $c_1 = c_2$. We have thus shown that
\[
\Phi(\{t = c\} \cap I_{g^{(1)}}(p^-, p^+)) = \{t = c\} \cap I_{g^{(2)}}(p^-, p^+), \quad \forall c \in (t^-, t^+).
\]

On each time slice $\{t = c\} \cap I_g(p^-, p^+)$, both metrics $\tilde{g}^{(1)}$ and $\tilde{g}^{(2)}$ are Euclidean metrics, denoted by $\delta$. Thus $\Phi_c := \Phi|_{\{t = c\} \cap I_{g_1}(p^-, p^+)}: \{t = c\} \cap I_{g_1}(p^-, p^+) \to \{t = c\} \cap I_{g_2}(p^-, p^+)$ is a diffeomorphism such that $\delta = \varphi \Phi_c^* \delta$. By Liouville's Theorem for conformal mapping, $\Phi_c$ must be a M\"obius transformation (cf., for example, \cite{kushelman2024liouville}). Since $\{t = c\} \cap I_{g_1}(p^-, p^+) \cap V \neq \emptyset$, $\Phi_c|_{\{t = c\} \cap I_g(p^-, p^+) \cap V} = \text{Id}$ and $\varphi|_{\{t = c\} \cap I_g(p^-, p^+) \cap V} = 1$, real analyticity of M\"obius transform forces the conformal map to be the identity map. Namely, $\{t = c\} \cap I_{g_1}(p^-, p^+) = \{t = c\} \cap I_{g_2}(p^-, p^+)$, $\Phi_c = \text{Id}$ and $\varphi = 1$ on $\{t = c\} \cap I_{g_1}(p^-, p^+)$. As this holds for any $c \in (t^-, t^+)$, $\Phi = \text{Id}$ and $\varphi \equiv 1$. This gives $\tilde{g}^{(1)} = \tilde{g}^{(2)}$, so $(c_0^{(1)}, \mathbf{v}_0^{(1)}) = (c_0^{(2)}, \mathbf{v}_0^{(2)})$ on the causal diamond due to the specific form of the metric. Thus, $(\rho_0^{(1)}, \mathbf{v}_0^{(1)}) = (\rho_0^{(2)}, \mathbf{v}_0^{(2)})$ on the same causal diamond $I_g(p^-, p^+)$.
\end{proof}

\bibliographystyle{abbrv}
\bibliography{biblio}

@article{kurylev2018inverse,
  title={Inverse problems for {L}orentzian manifolds and non-linear hyperbolic equations},
  author={Kurylev, Yaroslav and Lassas, Matti and Uhlmann, Gunther},
  journal={Inventiones Mathematicae},
  volume={212},
  number={3},
  pages={781--857},
  year={2018},
  publisher={Springer}
}

@article{lassas2018inverse,
  title={Inverse problems for semilinear wave equations on {L}orentzian manifolds},
  author={Lassas, Matti and Uhlmann, Gunther and Wang, Yiran},
  journal={Communications in Mathematical Physics},
  volume={360},
  number={2},
  pages={555--609},
  year={2018},
  publisher={Springer}
}

@article{de2018nonlinear,
  title={Nonlinear interaction of waves in elastodynamics and an inverse problem},
  author={de Hoop, Maarten and Uhlmann, Gunther and Wang, Yiran},
  journal={Mathematische Annalen},
  volume={376},
  number={1-2},
  pages={765--795},
  year={2020},
  publisher={Springer}
}

@article{helin2018correlation,
  title={Correlation based passive imaging with a white noise source},
  author={Helin, Tapio and Lassas, Matti and Oksanen, Lauri and Saksala, Teemu},
  journal={Journal de Math{\'e}matiques Pures et Appliqu{\'e}es},
  volume={116},
  pages={132--160},
  year={2018},
  publisher={Elsevier}
}

@article{lassas2025coefficient,
  title={Coefficient Determination for Nonlinear {S}chr{\"o}dinger Equations on Manifolds},
  author={Lassas, Matti and Oksanen, Lauri and Sahoo, Suman Kumar and Salo, Mikko and Tetlow, Alexander},
  journal={SIAM Journal on Mathematical Analysis},
  volume={57},
  number={4},
  pages={4425--4458},
  year={2025},
  publisher={SIAM}
}

@book{hormander2009analysis,
  title={The analysis of linear partial differential operators {IV}: Fourier Integral Operators},
  author={H{\"o}rmander, Lars},
  year={2009},
  publisher={Springer Science \& Business Media}
}

@article{melrose1979lagrangian,
  title={Lagrangian intersection and the {C}auchy problem},
  author={Melrose, Richard B and Uhlmann, Gunther A},
  journal={Communications on Pure and Applied Mathematics},
  volume={32},
  number={4},
  pages={483--519},
  year={1979},
  publisher={Wiley Online Library}
}

@article{greenleaf1993recovering,
  title={Recovering singularities of a potential from singularities of scattering data},
  author={Greenleaf, Allan and Uhlmann, Gunther},
  journal={Communications in mathematical physics},
  volume={157},
  number={3},
  pages={549--572},
  year={1993},
  publisher={Springer}
}

@article{wang2019inverse,
  title={Inverse problems for quadratic derivative nonlinear wave equations},
  author={Wang, Yiran and Zhou, Ting},
  journal={Communications in Partial Differential Equations},
  volume={44},
  number={11},
  pages={1140--1158},
  year={2019},
  publisher={Taylor \& Francis}
}

@inproceedings{feizmohammadi2021inverse,
  title={Inverse problems for nonlinear hyperbolic equations with disjoint sources and receivers},
  author={Feizmohammadi, Ali and Lassas, Matti and Oksanen, Lauri},
  booktitle={Forum of Mathematics, Pi},
  volume={9},
  pages={e10},
  year={2021},
  organization={Cambridge University Press}
}

@article{nursultanov2025determining,
  title={Determining Lorentzian manifold from non-linear wave observation at a single point},
  author={Nursultanov, Medet and Oksanen, Lauri and Tzou, Leo},
  journal={Journal of Differential Equations},
  volume={444},
  pages={113563},
  year={2025},
  publisher={Elsevier}
}

@article {dencker1982polarization,
    AUTHOR = {Dencker, Nils},
     TITLE = {On the propagation of polarization sets for systems of real
              principal type},
   JOURNAL = {J. Functional Analysis},
  FJOURNAL = {Journal of Functional Analysis},
    VOLUME = {46},
      YEAR = {1982},
    NUMBER = {3},
     PAGES = {351--372},
      ISSN = {0022-1236},
   MRCLASS = {58G17 (35Q20 35S05 58G15)},
  MRNUMBER = {661876},
MRREVIEWER = {K.\ A.\ Yagdjian},
       DOI = {10.1016/0022-1236(82)90051-9},
       URL = {https://doi.org/10.1016/0022-1236(82)90051-9},
}

@misc{ticequasilinear,
  title={Quasilinear symmetric hyperbolic systems},
  author={Tice, Ian},
  url={https://www.math.cmu.edu/~iantice/notes/quasilinear_hyperbolic_systems.pdf},
  year={2007}
}

@misc{taylor2023partial,
  title={Partial {D}ifferential {E}quations {III}: {N}onlinear {E}quations},
  author={Taylor, Michael Eugene},
  year={2023},
  publisher={Springer}
}

@article{uhlmann2023inverse,
  title={An inverse boundary value problem arising in nonlinear acoustics},
  author={Uhlmann, Gunther and Zhang, Yang},
  journal={SIAM Journal on Mathematical Analysis},
  volume={55},
  number={2},
  pages={1364--1404},
  year={2023},
  publisher={SIAM}
}

@article{lassas2025gaussian,
  title={Gaussian beam interactions and inverse source problems for nonlinear wave equations},
  author={Lassas, Matti and Liimatainen, Tony and Pohjola, Valter and Tyni, Teemu},
  journal={arXiv preprint arXiv:2510.11494},
  year={2025}
}

@article{eptaminitakis2024weakly,
  title={Weakly nonlinear geometric optics for the {W}estervelt equation and recovery of the nonlinearity},
  author={Eptaminitakis, Nikolas and Stefanov, Plamen},
  journal={SIAM Journal on Mathematical Analysis},
  volume={56},
  number={1},
  pages={801--819},
  year={2024},
  publisher={SIAM}
}

@article{qiu2026inverse,
  title={Inverse boundary value problems of determining nonlinear coefficients for the {JMGT} equation},
  author={Qiu, Dong and Xu, Xiang and Ye, Yeqiong and Zhou, Ting},
  journal={arXiv preprint arXiv:2603.14194},
  year={2026}
}

@article{kaltenbacher2023simultaneous,
  title={On the simultaneous reconstruction of the nonlinearity coefficient and the sound speed in the {W}estervelt equation},
  author={Kaltenbacher, Barbara and Rundell, William},
  journal={Inverse Problems},
  volume={39},
  number={10},
  pages={105001},
  year={2023},
  publisher={IOP Publishing}
}

@article{sa2024recovery,
  title={Recovery of a general nonlinearity in the semilinear wave equation},
  author={S{\'a} Barreto, Ant{\^o}nio and Stefanov, Plamen},
  journal={Asymptotic Analysis},
  volume={138},
  number={1-2},
  pages={27--68},
  year={2024},
  publisher={SAGE Publications Sage UK: London, England}
}

@article{sa2022recovery,
  title={Recovery of a cubic non-linearity in the wave equation in the weakly non-linear regime},
  author={S{\'a} Barreto, Ant{\^o}nio and Stefanov, Plamen},
  journal={Communications in Mathematical Physics},
  volume={392},
  number={1},
  pages={25--53},
  year={2022},
  publisher={Springer}
}

@article{kurylev2022inverse,
  title={Inverse problem for {E}instein-scalar field equations},
  author={Kurylev, Yaroslav and Lassas, Matti and Oksanen, Lauri and Uhlmann, Gunther},
  journal={Duke Mathematical Journal},
  volume={171},
  number={16},
  pages={3215--3282},
  year={2022},
  publisher={Duke University Press}
}

@article{uhlmann2020determination,
  title={Determination of Space-Time Structures from Gravitational Perturbations},
  author={Uhlmann, Gunther and Wang, Yiran},
  journal={Communications on Pure and Applied Mathematics},
  volume={73},
  number={6},
  pages={1315--1367},
  year={2020},
  publisher={Wiley Online Library}
}

@article{uhlmann2022inverse,
  title={Inverse boundary value problems for wave equations with quadratic nonlinearities},
  author={Uhlmann, Gunther and Zhang, Yang},
  journal={Journal of Differential Equations},
  volume={309},
  pages={558--607},
  year={2022},
  publisher={Elsevier}
}

@article{lassas2020uniqueness,
  title={Uniqueness, reconstruction and stability for an inverse problem of a semi-linear wave equation},
  author={Lassas, Matti and Liimatainen, Tony and Potenciano-Machado, Leyter and Tyni, Teemu},
  journal={Journal of Differential Equations},
  volume={337},
  pages={395--435},
  year={2022},
  publisher={Elsevier}
}

@article{lassas2022inverse,
  title={An inverse problem for a semi-linear wave equation: a numerical study},
  author={Lassas, Matti and Liimatainen, Tony and Potenciano-Machado, Leyter and Tyni, Teemu},
  journal={arXiv preprint arXiv:2203.09427},
  year={2022}
}

@article{lassas2021stability,
  title={Stability estimates for inverse problems for semi-linear wave equations on {L}orentzian manifolds},
  author={Lassas, Matti and Liimatainen, Tony and Potenciano-Machado, Leyter and Tyni, Teemu},
  journal={arXiv preprint arXiv:2106.12257},
  year={2021}
}

@article{tzou2023determining,
  title={Determining Riemannian manifolds from nonlinear wave observations at a single point},
  author={Tzou, Leo},
  journal={Inverse Problems},
  volume={39},
  number={11},
  pages={115001},
  year={2023},
  publisher={IOP Publishing}
}

@article{chen2025retrieving,
  title={Retrieving {Y}ang--{M}ills--{H}iggs fields in {M}inkowski space from active local measurements},
  author={Chen, Xi and Lassas, Matti and Oksanen, Lauri and Paternain, Gabriel P},
  journal={Mathematische Annalen},
  volume={391},
  number={2},
  pages={2385--2428},
  year={2025},
  publisher={Springer}
}

@article{chen2025inverse,
  title={An inverse problem for the Standard Model of particle physics},
  author={Chen, Xi and Lassas, Matti and Oksanen, Lauri and Paternain, Gabriel P},
  journal={arXiv preprint arXiv:2505.24454},
  year={2025}
}

@article{uhlmann2024determination,
  title={Determination of the density in a nonlinear elastic wave equation},
  author={Uhlmann, Gunther and Zhai, Jian},
  journal={Mathematische Annalen},
  volume={390},
  number={2},
  pages={2825--2858},
  year={2024},
  publisher={Springer}
}

@article{hintz2024inverse,
  title={Inverse Nonlinear Scattering by a Metric},
  author={Hintz, Peter and Barreto, Ant{\^o}nio S{\'a} and Uhlmann, Gunther and Zhang, Yang},
  journal={arXiv preprint arXiv:2411.09671},
  year={2024}
}

@article{alexakis2024inverse,
  title={Inverse scattering problems for non-linear wave equations on Lorentzian manifolds},
  author={Alexakis, Spyros and Isozaki, Hiroshi and Lassas, Matti and Tyni, Teemu},
  journal={arXiv preprint arXiv:2411.09354},
  year={2024}
}

@article{uhlmann2021inverse,
  title={On an inverse boundary value problem for a nonlinear elastic wave equation},
  author={Uhlmann, Gunther and Zhai, Jian},
  journal={Journal de Math{\'e}matiques Pures et Appliqu{\'e}es},
  volume={153},
  pages={114--136},
  year={2021},
  publisher={Elsevier}
}

@article{chen2021inverse,
  title={Inverse problem for the {Y}ang--{M}ills equations},
  author={Chen, Xi and Lassas, Matti and Oksanen, Lauri and Paternain, Gabriel P},
  journal={Communications in Mathematical Physics},
  volume={384},
  pages={1187--1225},
  year={2021},
  publisher={Springer}
}

@article{hintz2022inverse,
  title={An inverse boundary value problem for a semilinear wave equation on Lorentzian manifolds},
  author={Hintz, Peter and Uhlmann, Gunther and Zhai, Jian},
  journal={International Mathematics Research Notices},
  volume={2022},
  number={17},
  pages={13181--13211},
  year={2022},
  publisher={Oxford University Press}
}

@article{hintz2022dirichlet,
  title={The {D}irichlet-to-{N}eumann map for a semilinear wave equation on {L}orentzian manifolds},
  author={Hintz, Peter and Uhlmann, Gunther and Zhai, Jian},
  journal={Communications in Partial Differential Equations},
  volume={47},
  number={12},
  pages={2363--2400},
  year={2022},
  publisher={Taylor \& Francis}
}

@article{acosta2022nonlinear,
  title={Nonlinear ultrasound imaging modeled by a {W}estervelt equation},
  author={Acosta, Sebastian and Uhlmann, Gunther and Zhai, Jian},
  journal={SIAM Journal on Applied Mathematics},
  volume={82},
  number={2},
  pages={408--426},
  year={2022},
  publisher={SIAM}
}

@article{chen2021detection,
  title={Detection of {H}ermitian connections in wave equations with cubic non-linearity},
  author={Chen, Xi and Lassas, Matti and Oksanen, Lauri and Paternain, Gabriel P},
  journal={Journal of the European Mathematical Society},
  volume={24},
  number={7},
  pages={2191--2232},
  year={2021}
}

@article{feizmohammadi2019recovery,
 title={Recovery of zeroth order coefficients in non-linear wave equations},
  volume={21}, 
  DOI={10.1017/S1474748020000122}, 
  number={2}, 
  journal={Journal of the Institute of Mathematics of Jussieu}, 
  publisher={Cambridge University Press}, 
  author={Feizmohammadi, Ali and Oksanen, Lauri}, 
  year={2022}, 
  pages={367–393}}

@article{balehowsky2022inverse,
  title={An Inverse Problem for the Relativistic {B}oltzmann Equation},
  author={Balehowsky, Tracey and Kujanp{\"a}{\"a}, Antti and Lassas, Matti and Liimatainen, Tony},
  journal={Communications in Mathematical Physics},
  volume={396},
  number={3},
  pages={983--1049},
  year={2022},
  publisher={Springer}
}

@article{chen2025stable,
  title={Stable inversion of potential in nonlinear wave equations with cubic nonlinearity},
  author={Chen, Xi and Lu, Shuai and Zhang, Ruochong},
  journal={Mathematische Annalen},
  volume={392},
  number={3},
  pages={4283--4314},
  year={2025},
  publisher={Springer}
}

@article{oksanen2025inverse,
  title={Inverse problem for connections in semi-linear wave equations on {L}orentzian manifolds},
  author={Oksanen, Lauri and Zhang, Ruochong},
  journal={arXiv preprint arXiv:2509.25971},
  year={2025}
}

@article{alexakis2022lorentzian,
  title={Lorentzian {C}alder{\'o}n problem under curvature bounds},
  author={Alexakis, Spyros and Feizmohammadi, Ali and Oksanen, Lauri},
  journal={Inventiones mathematicae},
  volume={229},
  number={1},
  pages={87--138},
  year={2022},
  publisher={Springer}
}

@article{alexakis2024lorentzian,
  title={Lorentzian {C}alder{\'o}n problem near the Minkowski geometry},
  author={Alexakis, Spyros and Feizmohammadi, Ali and Oksanen, Lauri},
  journal={Journal of the European Mathematical Society},
  volume={27},
  number={9},
  pages={3771--3792},
  year={2024}
}

@article{belishev1992reconstruction,
  title={To the reconstruction of a {R}iemannian manifold via its spectral data ({BC}--{M}ethod)},
  author={Belishev, Michael I and Kurylev, Yarosiav V},
  journal={Communications in Partial Differential Equations},
  volume={17},
  number={5-6},
  pages={767--804},
  year={1992},
  publisher={Taylor \& Francis}
}

@inproceedings{belishev1987approach,
  title={An approach to multidimensional inverse problems for the wave equation},
  author={Belishev, Mikhail Igorevich},
  booktitle={Doklady Akademii Nauk},
  volume={297},
  number={3},
  pages={524--527},
  year={1987},
  organization={Russian Academy of Sciences}
}

@book{grigis1994microlocal,
  title={Microlocal analysis for differential operators: an introduction},
  author={Grigis, Alain and Sj{\"o}strand, Johannes},
  volume={196},
  year={1994},
  publisher={Cambridge university press}
}

@article{guillemin1981oscillatory,
  title={Oscillatory integrals with singular symbols},
  author={Guillemin, V and Uhlmann, G},
  journal={Duke Math. J.},
  volume={48},
  number={1},
  pages={251--267},
  year={1981}
}

@article{de2015diffraction,
  title={Diffraction from conormal singularities},
  author={de Hoop, Maarten and Uhlmann, Gunther and Vasy, Andr{\'a}s},
  journal={Ann. Sci. {\'E}c. Norm. Sup{\'e}r.(4)},
  volume={48},
  number={2},
  pages={351--408},
  year={2015}
}

@book{kachalov2001inverse,
  title={Inverse boundary spectral problems},
  author={Kachalov, Alexander and Kurylev, Yaroslav and Lassas, Matti},
  year={2001},
  publisher={Chapman and Hall/CRC}
}

@article{kurylev2018connection,
  title={Inverse problems for the connection {L}aplacian},
  author={Kurylev, Yaroslav and Oksanen, Lauri and Paternain, Gabriel P},
  journal={Journal of Differential Geometry},
  volume={110},
  number={3},
  pages={457--494},
  year={2018},
  publisher={Lehigh University}
}

@article{lassas2014inverse,
  title={Inverse problem for the {R}iemannian wave equation with {D}irichlet data and {N}eumann data on disjoint sets},
  author={Lassas, Matti and Oksanen, Lauri},
  journal={Duke Mathematical Journal},
  volume={163},
  number={6},
  pages={1071--1103},
  year={2014},
  publisher={Duke University Press}
}

@article{stefanov2005stable,
  title={Stable determination of generic simple metrics from the hyperbolic {D}irichlet-to-{N}eumann map},
  author={Stefanov, Plamen and Uhlmann, Gunther},
  journal={International Mathematics Research Notices},
  volume={2005},
  number={17},
  pages={1047--1061},
  year={2005},
  publisher={OUP}
}

@book{dafermos2005hyberbolic,
  title={Hyberbolic conservation laws in continuum physics},
  author={Dafermos, Constantine M},
  year={2005},
  publisher={Springer}
}

@article{kushelman2024liouville,
  title={On {L}iouville’s theorem for conformal maps},
  author={Kushelman, Mathew and McGrath, Peter},
  journal={The American Mathematical Monthly},
  volume={131},
  number={7},
  pages={619--623},
  year={2024},
  publisher={Taylor \& Francis}
}

\end{document}